\documentclass[smallextended,final]{svjour3}     
\usepackage[usenames]{xcolor}
\usepackage[letterpaper,top=2in, bottom=1.5in, left=1in, right=1in]{geometry}

\smartqed 

\usepackage[normalem]{ulem} 

\usepackage{amssymb}
\usepackage{mathrsfs}
\usepackage{amssymb,amsmath,amsfonts}
\usepackage{mathtools} 

\newcommand{\cut}[1]{{}}
\newcommand{\commwy}[1]{{\color{red}(#1 -- Wotao)}} 

\newcommand{\vu}{{\mathbf{u}}}
\newcommand{\vv}{{\mathbf{v}}}

\newcommand{\vx}{{\mathbf{x}}}

\newcommand{\vz}{{\mathbf{z}}}

\newcommand{\vX}{{\mathbf{X}}}

\newcommand{\cB}{{\mathcal{B}}}
\newcommand{\cC}{{\mathcal{C}}}

\newcommand{\cF}{{\mathcal{F}}}
\newcommand{\cG}{{\mathcal{G}}}
\newcommand{\cH}{{\mathcal{H}}}
\newcommand{\cI}{{\mathcal{I}}}

\newcommand{\cX}{{\mathcal{X}}}

\newcommand{\EE}{{\mathbb{E}}}

\newcommand{\RR}{\mathbb{R}}

\newcommand{\St}{{\mathrm{subject~to}}} 
\newcommand{\Prob}{{\mathrm{Prob}}} 
\newcommand{\Proj}{{\mathrm{Proj}}}

\newcommand{\prox}{\mathbf{prox}}
\newcommand{\refl}{\mathbf{refl}}

\newcommand{\TPRS}{T_{\mathrm{PRS}}}
\newcommand{\TFBS}{T_{\mathrm{FBS}}}

\DeclareMathOperator*{\argmin}{arg\,min}

\DeclareMathOperator*{\Min}{minimize}

\DeclareMathOperator*{\Fix}{Fix}
\DeclareMathOperator*{\zer}{zer}

\DeclarePairedDelimiter{\dotp}{\langle}{\rangle}

\newcommand{\bc}{\begin{center}}
\newcommand{\ec}{\end{center}}

\newcommand{\bdm}{\begin{displaymath}}
\newcommand{\edm}{\end{displaymath}}

\newcommand{\beq}{\begin{equation}}
\newcommand{\eeq}{\end{equation}}

\newcommand{\bfl}{\begin{flushleft}}
\newcommand{\efl}{\end{flushleft}}

\newcommand{\bt}{\begin{tabbing}}
\newcommand{\et}{\end{tabbing}}

\newcommand{\beqn}{\begin{align}}
\newcommand{\eeqn}{\end{align}}

\newcommand{\beqs}{\begin{align*}} 
\newcommand{\eeqs}{\end{align*}}  

\newtheorem{assumption}{Assumption}
\usepackage[pdftex]{graphicx}
\usepackage{caption}
\usepackage{multirow}
\usepackage{subcaption}
\usepackage{booktabs}
\usepackage[ruled,vlined]{algorithm2e}
\usepackage{amssymb}
\usepackage[backref,colorlinks=true, linkcolor=blue, citecolor=blue, urlcolor=blue, pdfborder={0 0 0}]{hyperref}

\makeatletter
\def\label@in@display{\gdef\df@label}
\makeatother

\title{ARock: an Algorithmic Framework for Asynchronous Parallel Coordinate Updates}

\author{Zhimin Peng \and
        Yangyang Xu \and
        Ming Yan \and
        Wotao Yin}

\institute{Z. Peng \and W. Yin\at
              Department of Mathematics, University of California, Los Angeles, CA 90095, USA\\
              \email{zhimin.peng / wotaoyin@math.ucla.edu}\\
           Y. Xu\at Institute for Mathematics and its Application, University of Minnesota, Minneapolis, MN 55455, USA\\ \email{yangyang@ima.umn.edu} \\
            M. Yan\at Department of Computational Mathematics, Science and Engineering, Michigan State University, East Lansing, MI 48824, USA \\ \email{yanm@math.msu.edu}}
\date{\today}
\journalname{Report}  

\begin{document}

\maketitle

\begin{abstract}
Finding a fixed point to a nonexpansive operator, i.e., $x^*=Tx^*$, abstracts many problems in numerical linear algebra, optimization, and other areas of scientific computing. To solve fixed-point problems, we propose ARock, an algorithmic framework  in which multiple agents (machines, processors, or cores) update  $x$ in an  asynchronous parallel fashion. Asynchrony is crucial to parallel computing since it reduces  synchronization wait, relaxes communication bottleneck, and thus speeds up computing significantly. At each step of ARock, an  agent updates a randomly selected coordinate $x_i$ based on possibly out-of-date information on $x$. The agents share $x$  through either global memory or communication. If writing $x_i$ is atomic, the agents can read and write $x$ without memory locks.

Theoretically, we show that if the nonexpansive operator $T$ has a fixed point, then with probability one,  ARock generates a sequence that converges to a fixed points of $T$. Our conditions on $T$ and  step sizes are  weaker than comparable work. Linear convergence  is also obtained.

We propose special cases of ARock for linear systems, convex optimization, machine learning, as well as distributed and decentralized consensus problems. Numerical experiments of solving sparse logistic regression problems are presented.
\end{abstract}
\vspace{-4mm}
\tableofcontents

\thispagestyle{plain}
\markboth{Z. Peng, Y. Xu, M. Yan, and W. Yin}{ARock: Async-Parallel Coordinate Updates}

\section{Introduction}
Technological advances in data gathering and storage have led to a rapid proliferation of big data in diverse areas such as climate studies, cosmology, medicine, the Internet, and engineering \cite{house2014big}. The data involved in many of these modern applications 
are large and grow quickly. Therefore, parallel computational approaches are needed. This paper introduces a new approach to asynchronous parallel computing with convergence guarantees.

In a synchronous(sync) parallel iterative algorithm, the  agents must wait for the slowest agent to finish an iteration before they can all proceed to the next one (Figure \ref{fig:async_vs_sync:a}). Hence, the slowest agent  may cripple the system. In contract,  the agents in an asynchronous(async) parallel iterative algorithm run continuously with little idling (Figure \ref{fig:async_vs_sync:b}). However, the iterations are disordered, and an agent may carry out an iteration without the newest information from other agents.

\begin{figure}[!h]
        \centering
       \begin{subfigure}[b]{0.4\textwidth}
                \includegraphics[width=\textwidth]{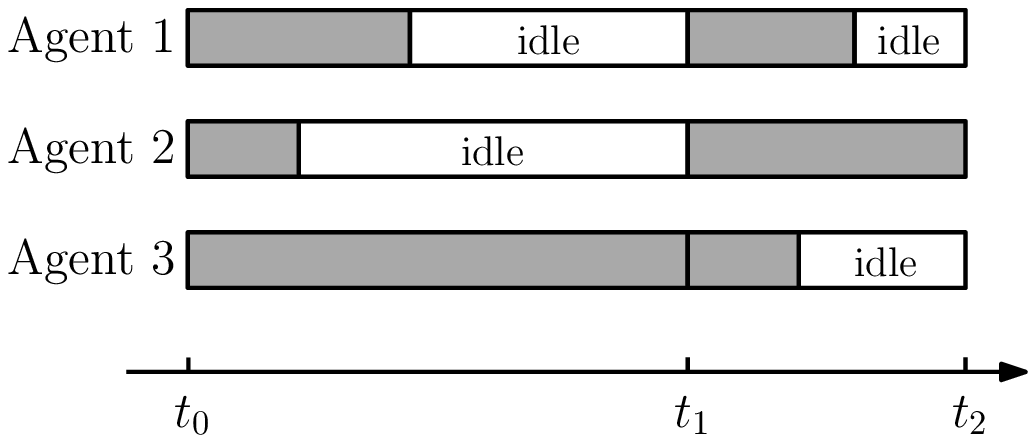}
                \caption{Sync-parallel computing\label{fig:async_vs_sync:a}}
        \end{subfigure}
        ~ 
        \begin{subfigure}[b]{0.4\textwidth}
                \includegraphics[width=\textwidth]{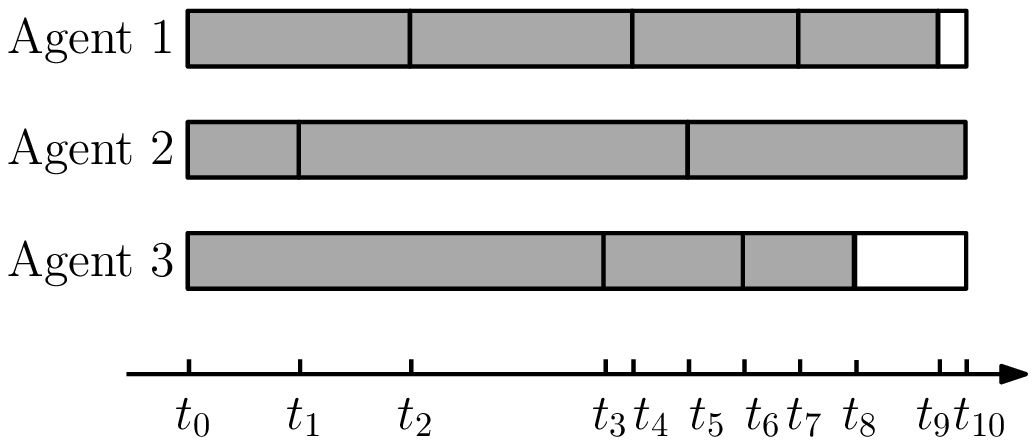}
                \caption{Async-parallel computing\label{fig:async_vs_sync:b}}
        \end{subfigure}
        \vspace{-3mm}
        \caption{Sync-parallel computing (left) versus async-parallel computing (right).
}\label{fig:async_vs_sync}
\vspace{-3mm}
\end{figure}

Asynchrony has other advantages \cite{bertsekas1991some}: the system is more tolerant to computing faults and communication glitches; 
it is also easy to incorporate new agents.

On the other hand, it is more difficult to analyze asynchronous algorithms and ensure their convergence. It becomes impossible to find a sequence of  iterates that one completely determines the next. Nonetheless, we let any update be a new iteration and propose an async-parallel algorithm (ARock) for the generic fixed-point iteration. It converges if  the fixed-point operator is nonexpansive (Def. \ref{def:quasi_lip}) and has a fixed point.

Let $\cH_1,\ldots,\cH_m$ be Hilbert spaces and  $\cH:=\cH_1\times\cdots\times \cH_m$ be their Cartesian product. For a \emph{nonexpansive operator}  $T:\cH\to\cH$, our problem is to
\begin{equation}\label{eqn:fix_point_set}
\text{find } x^* \in \cH \qquad \text{ such that }  \qquad x^*=Tx^* .
\end{equation}
{Finding a fixed point to $T$ is equivalent to finding a zero of $S\equiv I-T,$ denoted by $x^*$ such that $0= Sx^*$. Hereafter, we will use both $S$ and $T$ for convenience.}

Problem \eqref{eqn:fix_point_set} is widely applicable in linear and nonlinear equations, statistical regression, machine learning, convex optimization, and optimal control. 
 A generic framework for problem \eqref{eqn:fix_point_set} is the Krasnosel'ski\u i--Mann (KM) iteration~\cite{krasnosel1955two}:
\begin{equation}\label{eqn:km}
x^{k+1} = x^k + \alpha \, (Tx^k-x^k),\quad \mbox{or equivalently,}\quad x^{k+1} = x^k - \alpha  Sx^k,
\end{equation}
where $\alpha \in (0, 1)$ is the step size. 
If  $\Fix T$ --- the set of fixed points of $T$ (zeros of $S$) --- is nonempty, then the sequence $(x^k)_{k\ge 0}$ converges weakly\cut{\footnote{A sequence $(y^k)_{k\ge 0}\subset\cH$ converges weakly to a point $y$ if $\dotp{y^k,z}\to\dotp{y,z}$ for every $z\in\cH$. A sequence $(y^k)_{k\ge 0}\subset\cH$ converges strongly to a point $y$ if $\|y^k-y\|\rightarrow 0$. Strong convergence gives weak convergence, but not vice versa. If $\cH$ has a finite dimension, then they are equivalent. }} to a point in $\Fix T$ and $(Tx^k-x^k)_{k\ge 0}$ converges strongly to 0. The KM iteration generalizes algorithms in convex optimization, linear algebra, differential equations, and monotone inclusions.  Its special cases include the following iterations: alternating projection, gradient descent, projected gradient descent, proximal-point algorithm,
Forward-Backward Splitting (FBS) \cite{passty1979ergodic}, Douglas-Rachford Splitting (DRS)~\cite{lions1979splitting}, 
 a three-operator splitting~\cite{davis2015three}, and the Alternating Direction Method of Multipliers (ADMM)~\cite{lions1979splitting,glowinski1975approximation}. 

In ARock, a set of  $p$ agents, $p\ge 1$, solve problem~\eqref{eqn:fix_point_set} by updating the coordinates $x_i\in\cH_i$, $i=1,\ldots,m$, in a  random and asynchronous fashion. Algorithm~\ref{alg:asyn_core} describes the framework. Its special forms for several applications are given in Section \ref{sec:application} below.

\begin{algorithm}[H]\label{alg:asyn_core}
\SetKwInOut{Input}{Input}\SetKwInOut{Output}{output}
\Input{$x^0\in\cH$,  $K>0$, a distribution $(p_1,\ldots,p_m)>0$ with $\sum_{i=1}^mp_i=1$;}
global iteration counter $k\gets 0$\;
 \While{$k < K$, every agent asynchronously and continuously}{
  select $i_k\in\{1,\ldots,m\}$ with $\mathrm{Prob}(i_k=i)=p_i$\;
  perform an update to $x_{i_k}$ according to  \eqref{eqn:asyn_update}\;
  update the global counter $k \leftarrow k+1$\;
 }
 \caption{ARock: a framework for async-parallel coordinate updates}
\end{algorithm}

Whenever an agent updates a coordinate, the global iteration counter $k$ increases by one.
The $k$th update is applied to $x_{i_k}\in\cH_{i_k}$, where $i_k\in \{1,\ldots,m\}$ is an independent random variable. Each coordinate update  has the form:
\begin{equation}
\label{eqn:asyn_update}
\textstyle x^{k+1} = x^k - \frac{\eta_k}{mp_{i_k}} \, S_{i_k}  \hat{x}^{k},
\end{equation}
where $\eta_k>0$ is a scalar whose range will be set later, $S_{i_k} {x} := (0, ..., 0, (Sx)_{i_k}, 0, ..., 0),$ and $mp_{i_k}$ is used to normalize nonuniform selection probabilities.
In the uniform case, namely,  $p_i \equiv \frac{1}{m}$ for all $i$, we have $mp_{i_k} \equiv 1$, which simplifies the update \eqref{eqn:asyn_update} to
\beq\label{eqn:simple_asyn_update}
x^{k+1} = x^k - \eta_k S_{i_k}  \hat{x}^{k}.
\eeq 
Here, the point $\hat x^k$ is what an agent reads from global memory to its local cache and to which $S_{i_k}$ is applied, and $x^k$ denotes the state of $x$ in global memory just before the update \eqref{eqn:asyn_update} is applied. In a sync-parallel algorithm, we have $\hat{x}^k=x^k$, but  in ARock, due to possible  updates  to $x$ by other agents, $\hat x^k$ can be different from $x^k$. This is a key difference between  sync-parallel and async-parallel algorithms. In Subsection \ref{subset:aync_mem} below, we will establish the relationship between $\hat x^k$ and $x^k$ as
\begin{equation}\label{eqn:inconsist}
\textstyle\hat{x}^k= x^k+\sum_{d\in J(k)}(x^{d}-x^{d+1}),
\end{equation}
where $J(k)\subseteq \{k-1,\ldots,k-\tau\}$ and $\tau \in \mathbb{Z}^+$ is the maximum number of other updates  to $x$ during the computation of \eqref{eqn:asyn_update}. Equation \eqref{eqn:inconsist} has appeared in \cite{liu2014asynchronous}.

The update \eqref{eqn:asyn_update} is only computationally worthy if  $S_{i} x$ is much cheaper to compute than   $Sx$. Otherwise, it is more preferable to apply the full KM update \eqref{eqn:km}. In Section \ref{sec:application}, we will present several applications that have the favorable structures for ARock. The recent work \cite{peng2016coordinate} studies coordinate friendly structures more thoroughly.

The convergence of ARock (Algorithm~\ref{alg:asyn_core}) is stated in Theorems~\ref{thm:convergence} and \ref{thm:strongly_monotone}. Here we include a shortened version, leaving detailed bounds to the full theorems:
\begin{theorem}[Global and linear convergence]\label{thm:simple_convergence}
Let $T:\cH\to\cH$ be a nonexpansive operator that has a  fixed point. 
Let $(x^k)_{k\ge 0}$ be the sequence generated by Algorithm \ref{alg:asyn_core} with  properly bounded step sizes $\eta_k$. Then, with probability one, $(x^k)_{k\ge 0}$  converges weakly to a fixed point of $T$. {This  convergence becomes  strong if $\cH$ has a finite dimension.}

In addition, if $T$ is demicompact (see Definition~\ref{def:monotone} below), then with probability one, $(x^k)_{k\ge 0}$ converges strongly to a fixed point of $T$ .

Furthermore, if $S\equiv I-T$ is quasi-strongly monotone (see Definition~\ref{def:quasi_lip} below), then $T$ has a unique fixed-point $x^*$,  $(x^k)_{k\ge 0}$  converges strongly to $x^*$ with probability one, and $\EE\|x^k-x^*\|^2$ converges to 0 at a linear rate.
\end{theorem}
\cut{\commrw{Do your examples demonstrate the importance of each of these three results (corresponding to three different conditions on the operator T)?}}

In the theorem, the weak convergence result only requires $T$ to be nonexpansive and has a fixed point. In addition, the computation requires: (a) bounded step sizes; (b) random coordinate selection; and (c) a finite maximal delay $\tau$. Assumption (a) is standard, and we will see the bound can be $O(1)$.  Assumption (b) is essential to both the analysis and the numerical performance of our algorithms. Assumption (c) is \emph{not} essential; an infinite delay with a light tail is allowed (but we leave it to future work). {The strong convergence result applies to all the examples in Section \ref{sec:application}, and the linear convergence result applies to Examples \ref{subsec:smooth} and \ref{subsec:nonsmooth}  when the corresponding operator $S$ is quasi-strongly monotone. Step sizes $\eta_k$ are discussed in Remarks \ref{rm:stepsize} and \ref{rmk:linear}.}

\subsection{On random coordinate selection}
ARock employs \emph{random coordinate selection}. This subsection discusses its advantages and disadvantages.

Its main disadvantage is that an agent cannot  caching
the data associated with a coordinate. The variable $x$ and its related data must be either stored  in  global memory or  passed through communication.\cut{An exception arises in the network setting where the coordinates are locally assigned to the agents and the coordinate updates are activated according to independent Poisson processes; then, the coordinate selection is effectively random and independent though no data needs to be moved.} A secondary disadvantage is that pseudo-random number generation takes  time, which becomes relatively significant if each coordinate update is cheap. (The network optimization examples  in Subsections~\ref{sec:dgd} and~\ref{subsec:async_admm_des} are exceptions, where  data  are naturally stored in a  distributed fashion and random coordinate assignments are  the  results of  Poisson processes.) \cut{\commrw{Add the explanation of Poisson Processes here}}

There are several  advantages of random coordinate selection. It realizes the user-specified update frequency $p_i$ for every component $x_i$, $i=1,\ldots,m$, even when different agents have different computing powers and different coordinate updates cost different amounts of computation. Therefore, random assignment ensures load balance.  The algorithm is also fault tolerant in the sense that if one or more agents fail, it will still converge to a fixed-point of $T$. In addition, it has been observed numerically on certain problems~\cite{chang2008coordinate} that random coordinate selection  accelerates convergence. \cut{Furthermore, in~\cite{xu2013block,xu2014globally}, the algorithms with random coordinate selection keep clear of low-quality local solutions in nonconvex optimization, which are difficult to avoid by fixed-order coordinate selection. \commrw{Is the utility of random coordinate updates in nonconvex optimization relevant for the argument that it's an interesting model to analyze in the convex case?}}



\subsection{Uncoordinated memory access}\label{subset:aync_mem}
In ARock, since multiple agents simultaneously read and update $x$ in global memory, $\hat x^k$ --- the result of  $x$ that is read from global memory by an agent to its local cache for computation --- may not equal $x^j$ for any $j\le k$, that is, $\hat x^k$ may never be consistent with a state of $x$ in global memory. This is known as \emph{inconsistent read}. In contrast, \emph{consistent read} means that $\hat x^k = x^j$ for some $j \le k$, i.e., $\hat x^k$ is consistent with a state of $x$ that existed in global memory.

We illustrate inconsistent read and consistent read in the following example, which is depicted in Figure \ref{fig:consist_vs_inconsistent}. Consider $x=[x_1,x_2, x_3, x_4]^T\in\RR^4$ and $x^0=[0, 0, 0, 0]^T$ initially, at time $t_0$. Suppose at time $t_1$,  agent 2 updates $x_1$ from 0 to 1, yielding  $x^1=[1, 0, 0, 0]^T$;  then, at time $t_2$, agent 3 updates $x_4$ from 0 to 2, further yielding $x^2=[1, 0, 0, 2]^T$. Suppose that agent 1 starts reading $x$ from the first component $x_1$ at $t_0$. For consistent read (Figure \ref{fig:consist}), agent 1 acquires a memory lock and only releases the lock after finishing reading all of $x_1$, $x_2$, $x_3$, and $x_4$. Therefore, agent 1 will read in $[0, 0, 0, 0]^T$. Inconsistent read, however, allows agent 1 to proceed without a memory lock: agent 1 starts reading $x_1$ at $t_0$ (Figure \ref{fig:inconsist}) and reaches the last component, $x_4$, after $t_2$; since $x_4$ is updated by agent 3 prior to it is read by agent 1, agent 1 has read $[0, 0, 0, 2]^T$, which is different from any of $x^0,x^1$, and $x^2$.
\begin{figure}[!h]
        \centering
       \begin{subfigure}[b]{0.47\textwidth}
                \includegraphics[width=\textwidth]{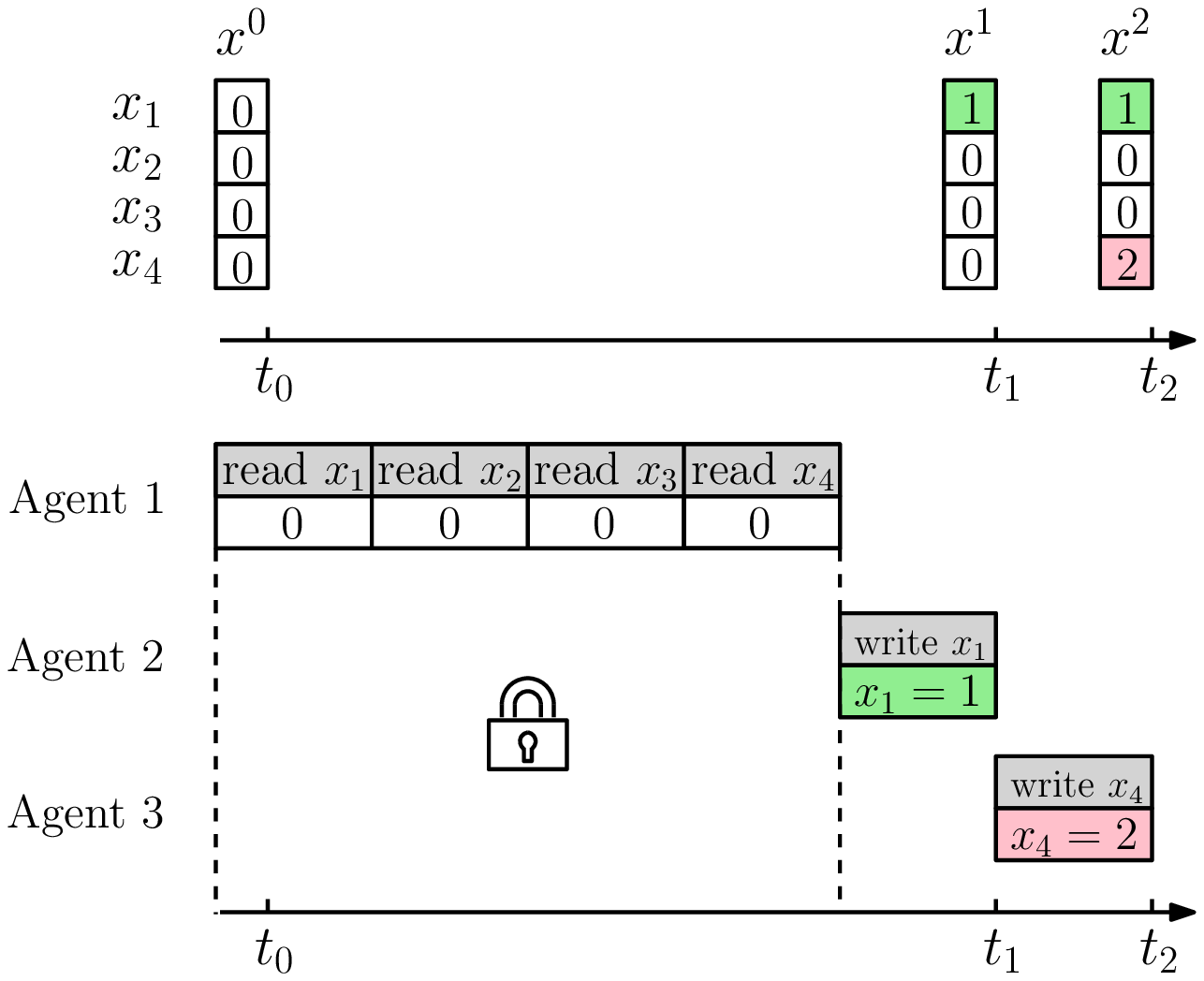}
                \caption{Consistent read. While agent 1 reads $x$ in memory, it acquires a global lock.}\label{fig:consist}
        \end{subfigure}
        ~ 
        \begin{subfigure}[b]{0.47\textwidth}
                \includegraphics[width=\textwidth]{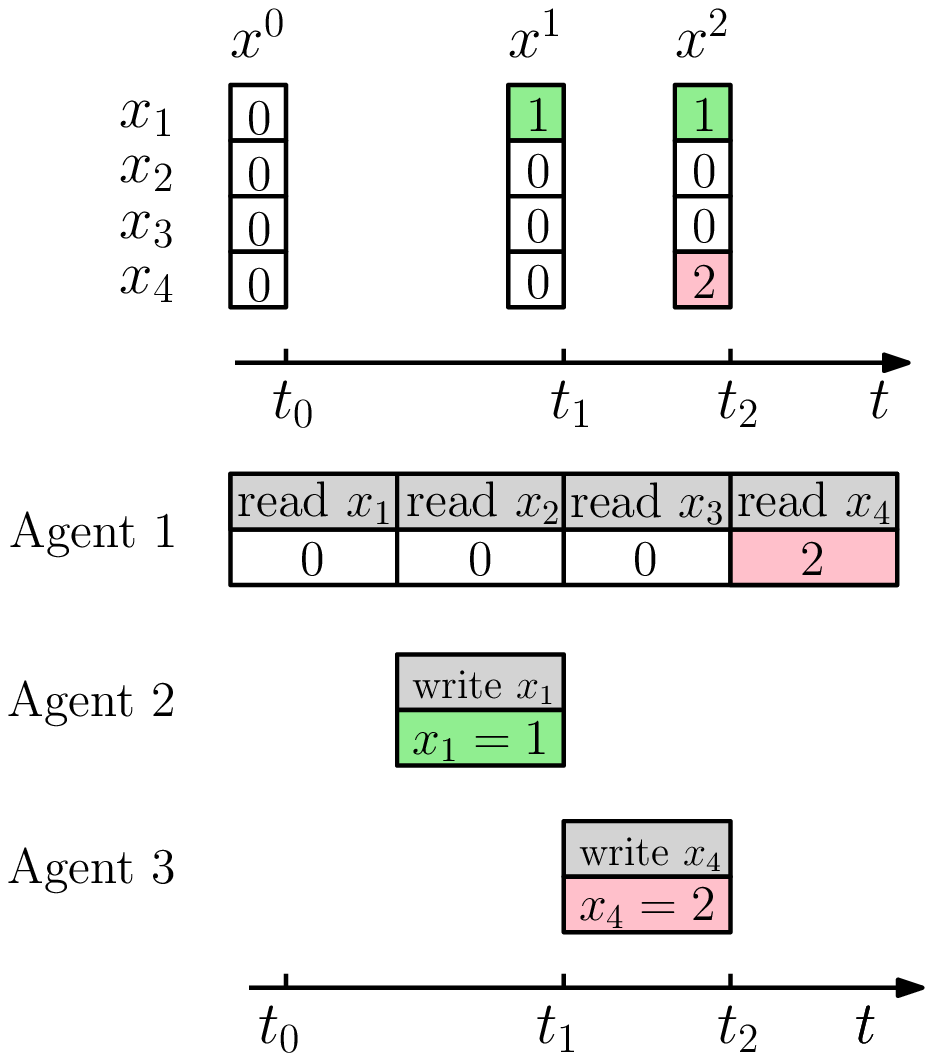}
                \caption{Inconsistent read. No lock. Agent 1 reads $(0,0,0,2)^T$, a non-existing state of $x$.}\label{fig:inconsist}
        \end{subfigure}
        \vspace{-3mm}
        \caption{Consistent read versus inconsistent read: A demonstration.}\label{fig:consist_vs_inconsistent}
        \vspace{-3mm}
\end{figure}

Even with inconsistent read, each component is consistent under the \emph{atomic coordinate update} assumption, which will be defined below. Therefore, we can express what has been read in terms of the changes of individual coordinates. In the above example, the first change is $x^1_1-x^0_1=1$, which is added to $x_1$ just before time $t_1$ by agent 2, and the second change is $x^2_4 - x^1_4=2$, added to $x_4$ just before time $t_2$ by agent 3. The  inconsistent read by agent 1, which gives the result $[0, 0, 0, 2]^T$, 
equals  $x^0 +0\times (x^1-x^0)+ 1\times (x^2 - x^1)$.

We have demonstrated that $\hat x^k$ can be  inconsistent, but each of its coordinates is consistent, that is, for each $i$, $\hat x^k_i$ is an ever-existed state of $x_i$ among $x_i^k,\ldots,x_i^{k-\tau}$. Suppose that  $\hat x^k_i=x_i^{\underline{d}}$, where $\underline{d}\in\{k,k-1,\ldots,k-\tau\}$. Therefore, $\hat x^k_i$ can be related to $x^k_i$ through the \emph{interim  changes} applied to $x_i$. Let $J_i(k)\subset\{k-1,\ldots,k-\tau\}$ be the index set of these interim changes. If $J_i(k)\not=\emptyset$, then $\underline{d}=\min\{d\in J_i(k)\}$; otherwise, $\underline{d}=k$. In addition, we have
 $\hat{x}^k_i =x_i^{\underline{d}}= x^k_i+\sum_{d\in J_i(k)}(x_i^{d}-x_i^{d+1})$.
Since the global counter $k$ is increased after each coordinate update, updates to $x_i$ and $x_j$, $i\not=j$, must occur at different $k$'s and thus $J_i(k)\cap J_j(k)=\emptyset,\,\forall i\neq j$. Therefore, by letting $J(k):=\cup_iJ_i(k)\subset\{k-1,\ldots,k-\tau\}$ and noticing $(x_i^{d}-x_i^{d+1})=0$ for $d\in J_j(k)$ where $i\not = j$, we have
$\textstyle \hat{x}^k_i = x^k_i+\sum_{d\in J(k)}(x_i^{d}-x_i^{d+1}), \forall i=1,\ldots,m$,
which is equivalent to~\eqref{eqn:inconsist}.
%
%
%
%
Here, we have made two assumptions:
\begin{itemize}
\item \emph{atomic coordinate update}: a coordinate is not further broken to smaller components during an update; they are all updated at once. 
\item \emph{bounded maximal delay} $\tau$: during any update cycle of an agent, $x$ in global memory is updated at most $\tau$ times by other agents.
\end{itemize}
When each coordinate is a single scalar, updating the scalar is a single atomic instruction on most modern hardware, so the first assumption naturally holds, and our algorithm is \emph{lock-free}. The case where a coordinate is a block that includes multiple scalars is discussed in the next subsection. 

\subsubsection{Block coordinate}\label{atomlock}
In the ``block coordinate" case (updating a block of several coordinates each time),  the atomic coordinate update assumption can be met by \emph{either} employing a per-coordinate memory  lock \emph{or} taking the following dual-memory approach: Store \emph{two} copies of each coordinate $x_i\in\cH_i$ in global memory, denoting them as $x_i^{(0)}$ and $x_i^{(1)}$; let a bit $\alpha_i\in\{0,1\}$ point to the active copy; an agent will only read $x_i$ from the active copy $x_i^{(\alpha_i)}$; before an agent updates the components of $x_i$, it  obtains a memory lock to the \emph{inactive} copy  $x_i^{(1-\alpha_i)}$ to prevent other agents from simultaneously updating it; then after it finishes updating $x_i^{(1-\alpha_i)}$, flip the bit $\alpha_i$ so that other agents will begin reading from the updated copy. This approach never blocks any read of $x_i$, yet it  eliminates inconsistency.


%

\subsection{Straightforward generalization}\label{sec:general}
Our async-parallel coordinate update scheme \eqref{eqn:asyn_update}  can be  generalized to (overlapping) block coordinate updates  after a  change to the step size.\cut{under the same convergence analysis. Such updates can be applied when an orthogonal matrix is a part of the problem data.} Specifically, the scheme \eqref{eqn:asyn_update} can be generalized to
\begin{equation}
\label{eqn:asyn_ortho_update}
\textstyle x^{k+1} = x^k - \frac{\eta_k}{n p_{i_k}} \, (U_{i_k} \circ S)\hat{x}^{k},
\end{equation}
where $U_{i_k}$ is randomly drawn from a set of operators $\{U_1,\ldots,U_n\}$ ($n \leq m$),  $U_i:\cH\to\cH$, following the probability $P(i_k = i) = p_{i}$, $i=1,\ldots, n$ ($p_i>0$, and $\sum_{i=1}^np_i=1$). The operators must satisfy
$\sum_{i=1}^n U_i = I_\cH$ and $\sum_{i=1}^n \|U_ix\|^2 \leq C\|x\|^2$ for some $C>0$.

Let $U_i: x\mapsto(0,\ldots,0,x_i,0,\ldots,0),~i=1,\ldots, m$, which has $C=1$; then~\eqref{eqn:asyn_ortho_update} reduces to~\eqref{eqn:asyn_update}. If  $\cH$ is endowed with a  metric  $M$ such that $\rho_1\|x\|^2\leq \|x\|_M^2\leq \rho_2\|x\|^2$ (e.g., the metric  in the Condat-V\~u primal-dual splitting \cite{condat2013primal,vu2013splitting}), then we have
\begin{align*}
\sum_{i=1}^m \|U_ix\|_M^2 \leq  \rho_2 \sum_{i=1}^m\|U_ix\|^2=\rho_2 \|x\|^2\leq {\frac{\rho_2}{\rho_1}}\|x\|_M^2.
\end{align*}
In general,  multiple coordinates can be updated in \eqref{eqn:asyn_ortho_update}. Consider linear $U_i:x\mapsto (a_{i1}x_1,\cdots,a_{im}x_{m})$, $i=1,\ldots, m$, where $\sum_{i=1}^na_{ij}=1$ for each $j$. Then, for $C:=\max\left\{\sum_{i=1}^n a_{i1}^2,\cdots,\sum_{i=1}^n a_{im}^2\right\}$, we have
\begin{align*}
\sum_{i=1}^n \|U_ix\|^2 = \sum_{i=1}^n\sum_{j=1}^m a_{ij}^2\|x_j\|^2=\sum_{j=1}^m\sum_{i=1}^n a_{ij}^2\|x_j\|^2\leq C\|x\|^2.
\end{align*}

\subsection{Special cases}
If there is only one agent ($p=1$), ARock (Algorithm \ref{alg:asyn_core}) reduces to  randomized coordinate update, which includes the special case of randomized coordinate descent  \cite{nesterov2012rcd} for  convex optimization. Sync-parallel coordinate update is another special case of ARock corresponding to $\hat{x}^k\equiv x^k$. In both cases, there is no delay, i.e., $\tau=0$ and $J(k)=\emptyset$. In addition, the step size $\eta_k$ can be more relaxed. In particular,  if $p_i=\frac{1}{m}$, $\forall i$, then we can let $\eta_k=\eta$, $\forall k$, for any $\eta<1$, or $\eta< 1/\alpha$ when $T$ is $\alpha$-averaged (see Definition~\ref{def:monotone} for the definition of an $\alpha$-averaged operator).


%

\begin{subsection}{Related work}\label{sec:rel-wk}

Chazan and Miranker~\cite{chazan1969chaotic} proposed the first async-parallel method  in 1969. The method was designed for solving linear systems. Later, async-parallel methods have been successful applied in many fields, e.g., linear systems~\cite{avron2014revisiting,bethune2014performance,FSS1997asyn-addSch,rosenfeld1969case}, nonlinear problems~\cite{BMR1997asyn-multisplit,baudet1978asynchronous}, differential equations~\cite{aharoni2000parallel,AAI1998implicit,Chau20081126,donzis2014asynchronous}, consensus problems~\cite{LMS1986asynchronous,leifang_information_2005}, and optimization~\cite{hsieh2015passcode,liu2014asynchronous,liu2013asynchronous,tai2002convergence,zhang2014asynchronous}. We review the theory for  async-parallel fixed-point iteration and its applications.

\cut{The asynchronous parallel computing method (also called \emph{chaotic relaxation}) appears to be first used by Rosenfeld~\cite{rosenfeld1969case} to solve linear equations arising in electrical network problem. Numerical simulations there show that using more processors can save total running time and sometimes achieve near-linear speedup. Chazan and Miranker~\cite{chazan1969chaotic} systematically analyzed a general asynchronous iterative method for solving linear systems. Assuming boundedness of the delay, it guarantees convergence to a solution if the spectral radius of the (componentwise) absolute of the iterating matrix is strictly less than \emph{one}. Since then, the theory and application has been studied by many authors.}

\textbf{General fixed point problems.} \emph{Totally async-parallel}\footnote{``Totally asynchronous" means  no upper bound on the delays; however, other conditions are required, for example: each coordinate must be updated infinitely many times. By default, ``asynchronous" in this paper assumes a finite maximum delay.} iterative methods for a  fixed-point problem go back as early as to  Baudet~\cite{baudet1978asynchronous}, where the operator was assumed to be \emph{P-contraction}.\footnote{An operator $T: \mathbb{R}^n \rightarrow \mathbb{R}^n$ is P-contraction if $|T(x) - T(y)| \leq P | x - y|$, component-wise, where $|x|$ denotes the vector with components $|x_i|, ~ i=1, ..., n$, and $P\in\RR^{n\times n}$ is a nonnegative matrix with a spectral radius strictly less than 1.
} Later, Bertsekas~\cite{bertsekas1983distributed} generalized the P-contraction assumption and showed convergence. \cut{Papers~\cite{Baz200591,el1998flexible} provide a convergence result for async-parallel iterations with simultaneous reading and writing of the same set of components. Strikwerda considered the unbounded but stochastic delays~\cite{Strikwerda2002125}.} Frommer and Szyld~\cite{frommer2000asynchronous} reviewed the theory and applications of totally async-parallel iterations prior to 2000. This review summarized convergence results under the conditions in~\cite{bertsekas1983distributed}.
However, ARock can be applied to solve many more problems since our nonexpansive assumption, though not strictly weaker than P-contraction, is more pervasive.  
As opposed to  totally asynchronous methods, Tseng, Bertsekas, and Tsitsiklis~\cite{bertsekas1989parallel,TB1990partially} assumed \emph{quasi-nonexpansiveness}\footnote{An operator $T: \cH \rightarrow \cH$ is quasi-nonexpansive if $\|Tx - x^*\| \leq \|x - x^*\|$, $\forall x \in \cH$, $x^*\in \Fix T$.} and proposed an {async-parallel} 
method, converging under an additional assumption, which is   difficult to justify in general but can be established   for problems such as linear systems and strictly convex network flow problems \cite{bertsekas1989parallel,TB1990partially}. 

{The above  works assign coordinates  in a deterministic manner. Different from them, ARock is stochastic, works for nonexpansive operators, and is more applicable.}

\textbf{Linear, nonlinear, and differential equations.}
The first async-parallel method for solving linear equations was introduced by Chazan and Miranker in~\cite{chazan1969chaotic}. They proved that on solving linear systems, P-contraction was necessary and sufficient for convergence. The performance of the algorithm was studied by Iain et al.~\cite{bethune2014performance,rosenfeld1969case} on different High Performance Computing (HPC) architectures. 
Recently, Avron et al.~\cite{avron2014revisiting} revisited the async-parallel coordinate update and showed its linear convergence for solving positive-definite linear systems. 
Tarazi and Nabih~\cite{el1982some} extended the poineering work \cite{chazan1969chaotic} to solving nonlinear equations, and the async-parallel methods have also been applied \cut{Applications of async-parallel methods }for solving differential equations, \cut{can be found}e.g., in~\cite{aharoni2000parallel,AAI1998implicit,Chau20081126,donzis2014asynchronous}. Except for~\cite{avron2014revisiting}, all these methods are totally async-parallel with the P-contraction condition or its variants. \cut{except~\cite{avron2014revisiting} employs random selection to solve a special linear system of equations. }On solving a positive-definite linear system, \cite{avron2014revisiting} made assumptions similar to ours, and it obtained better linear convergence rate on that special problem. 

\textbf{Optimization.}
The first async-parallel coordinate update gradient-projection method was due to Bertsekas and Tsitsiklis~\cite{bertsekas1989parallel}. The method solves constrained optimization problems with a smooth objective and simple constraints. It was shown that the objective gradient  sequence converges to zero. Tseng~\cite{tseng1991rate-asyn} further analyzed the convergence rate and obtained local linear convergence based on the assumptions of isocost surface separation and a local Lipschitz error bound. Recently, Liu et al.~\cite{liu2013asynchronous} developed an async-parallel stochastic coordinate descent algorithm for minimizing convex smooth functions. Later, Liu and Wright~\cite{liu2014asynchronous} suggested an async-parallel stochastic proximal coordinate descent algorithm for minimizing convex composite objective functions. They established the convergence of the expected objective-error sequence for convex functions. Hsieh et al.~\cite{hsieh2015passcode} proposed an async-parallel dual coordinate descent method  for solving $\ell_2$ regularized empirical risk minimization problems. Other async-parallel approaches include asynchronous ADMM~\cite{hong2014distributed,wei2013on,zhang2014asynchronous,iutzeler2013asynchronous}. Among them,~\cite{wei2013on,iutzeler2013asynchronous} use an asynchronous clock, 
and~\cite{hong2014distributed,zhang2014asynchronous} use a central node to update the dual variable; they do not deal with delay or inconsistency.  Async-parallel stochastic gradient descent methods have also been considered~in~\cite{nedic2001distributed,recht2011hogwild}.

Our framework differs from the recent surge of the aforementioned sync-parallel and async-parallel coordinate descent algorithms (e.g.,~\cite{peng2013parallel,kyrola2011parallel,liu2013asynchronous,liu2014asynchronous,hsieh2015passcode,richtarik2015parallel}). While they  apply to convex function minimization, ARock covers more cases (such as ADMM, primal-dual, and decentralized methods) and also provides sequence convergence. 
In Section \ref{sec:application}, we will show that some of the existing async-parallel coordinate descent algorithms are special cases of ARock, through relating their optimality conditions to nonexpansive operators. Another difference is that the convergence of ARock only requires a  nonexpansive operator with  a fixed point, whereas properties such as strong convexity, bounded feasible set, and bounded sequence, which  are  seen in some of the recent literature for async-parallel convex minimization, are unnecessary.

\textbf{Others.} Besides solving equations and optimization problems, there are also applications of  async-parallel algorithms to optimal control problems~\cite{LMS1986asynchronous}, network flow problems~\cite{ESMG1996asyn-flex}, and consensus problems of multi-agent systems~\cite{leifang_information_2005}.

\cut{The pioneering work for asynchronous iteration is due to Chazan and Miranker~\cite{chazan1969chaotic}, who proposed an asynchronous method for solving linear equations with convergence guarantee. The theory and application has since been studied by many authors. Noteworthy, Bertsekas and Tsitsiklis~\cite{bertsekas1989parallel} introduced an asynchronous parallel framework for general fixed point problems. More recently, Frommer and Szyld~\cite{frommer2000asynchronous} gave a review for the theory and application for asynchronous iterations. Baudet~\cite{baudet1978asynchronous} analyzed asynchronous iterations of contracting operators, so their results are only applicable to a very restive type of equations. Bertsekas and Tsitsiklis~\cite{bertsekas1989parallel} proved the convergence for asynchronous fixed point iterations of nonexpansive maps. However the convergence analysis is based on very restrictive conditions and in a hard to interpret manner.}
\end{subsection}

\begin{subsection}{Contributions}
Our contributions and techniques are summarized below:
\begin{itemize}
\item ARock is the first async-parallel coordinate update framework for finding a fixed point to a nonexpansive operator. 
\item 
By introducing a new metric and establishing stochastic Fej\'er monotonicity, we show  that, with probability one,   ARock   converges to a point in the solution set; 
linear convergence is obtained for~\emph{quasi-strongly monotone} operators.
\item Based on ARock, we introduce an async-parallel algorithm for linear  systems, async-parallel ADMM algorithms for  distributed or decentralized  computing problems, as well as async-parallel operator-splitting algorithms for  nonsmooth minimization problems. Some problems are treated in they async-parallel fashion for the first time in history. The developed algorithms are \emph{not} straightforward modifications to their serial versions because their underlying nonexpansive operators must be identified before applying ARock.
\end{itemize}
\end{subsection}

\begin{subsection}{Notation, definitions, background of monotone operators}
Throughout this paper, $\cH$ denotes a separable Hilbert space equipped with the inner product $\langle \cdot, \cdot \rangle$ and norm $\| \cdot \|$, and $(\Omega, \cF, P)$ denotes the underlying probability space, where  $\Omega$, $\cF$, and $P$ are the sample space,  $\sigma$-algebra, and  probability measure, respectively. The map $x: (\Omega, \cF) \rightarrow (\cH, \cB)$, where  $\cB$ is the Borel $\sigma$-algebra, is an $\cH$-valued random variable. Let  $(x^k)_{k\ge 0}$ denote \emph{either} a sequence of deterministic points in $\cH$ \emph{or} a sequence of $\cH$-valued random variables, which will be clear from the context, and let $x_i\in\cH_i$ denote the $i$th coordinate of $x$. In addition, we let $\cX^k := \sigma (x^0, \hat{x}^1, x^1, ..., \hat{x}^k, x^k)$ denote the smallest $\sigma$-algebra generated by $x^0, \hat{x}^1, x^1, ..., \hat{x}^k, x^k$.  ``Almost surely'' is abbreviated as ``a.s.'', and the $n$ product space of $\cH$ is denoted by $\cH^n$.  We use $\to$ and $\rightharpoonup$ for strong convergence and weak convergence, respectively.

We define $\Fix T:= \{x\in \cH ~|~ Tx = x\}$ as the set of fixed points of operator $T$, and, in the product space,  we let $\vX^{*} :=\{(x^*, x^*, ..., x^*) ~|~ x^* \in \Fix T\}\subseteq \cH^{\tau+1}$.


\begin{definition}\label{def:quasi_lip}
An operator $T: \cH \rightarrow \cH$ is \emph{$c$-Lipschitz}, where $c\ge 0$, if it satisfies $\|Tx - T y\| \leq c\|x - y\|$, $\forall x, y \in \cH$. In particular,
 $T$ is  \emph{nonexpansive} if $c\le 1$, and  \emph{contractive} if $c<1$.
\cut{
In addition, when $\Fix T\not=\emptyset$, $T$ is \emph{quasi $\bar{c}$-Lipschitz},  $\bar{c}\ge 0$, if it satisfies $\|Tx - y\| \leq \bar{c}\|x - y\|$, $\forall x \in \cH$ and $\forall y\in \Fix T$. $T$ is  \emph{quasi-nonexpansive} if $\bar{c}\le 1$, and \emph{quasi-contractive} if $\bar{c}<1$.}
\end{definition}

\begin{definition}\label{def:monotone} Consider an operator $T: \cH \rightarrow \cH$.
\begin{itemize}
\item $T$ is \emph{$\alpha$-averaged} with $\alpha \in (0, 1)$, if there is a nonexpansive operator $R:\cH \rightarrow \cH$ such that $T = (1 - \alpha) I_\cH + \alpha R$, where $I_\cH: \cH\rightarrow \cH$ is the identity operator.
\item $T$ is \emph{$\beta$-cocoercive} with $\beta>0$, if \cut{it satisfies }$\dotp{x - y, Tx - Ty} \geq \beta \|Tx - Ty\|^2,~\forall x, y \in \cH.$
\item $T$ is \emph{$\mu$-strongly monotone}, where $\mu>0$, if it satisfies $\dotp{x - y, Tx - Ty} \geq \mu\|x - y\|^2,~\forall x, y \in \cH.$ When the inequality holds for $\mu=0$, $T$ is \emph{monotone}.
\item $T$ is \emph{quasi-$\mu$-strongly monotone}, where $\mu>0$, if it satisfies $\dotp{x - y, Tx } \geq \mu\|x - y\|^2,~\forall x \in \cH,y\in\zer T:=\{y\in \cH\mid Ty=0\}$. When the inequality holds for $\mu=0$, $T$ is \emph{quasi-monotone}.
\item $T$ is \emph{demicompact~\cite{petryshyn1966construction} at $x\in\cH$} if for every bounded sequence $(x^k)_{k\geq0}$ in $\cH$ such that $Tx^k-x^k\to x$, there exists a strongly convergent subsequence.
\end{itemize}
\end{definition}
Averaged operators are nonexpansive. By the Cauchy-Schwarz inequality, a  {$\beta$-cocoercive} operator is $\frac{1}{\beta}$-Lipschitz; the converse is generally untrue, but true for the gradients of convex differentiable functions. Examples are given in the next section.\cut{(However, if $f$ is a convex differentiable function, by the Baillon-Haddad theorem, $\nabla f$ is  {$\beta$-cocoercive} if and only if it is $\frac{1}{\beta}$-Lipschitz.) It is worth noting that we can extend the definitions of (quasi) (strongly) monotone operators to \emph{set-valued operators} by replacing $Tx$ and $Ty$ in their definitions with any entries $p\in Tx$ and $q\in Ty$, respectively. When $T$ is monotone (possibly set-valued) and $\lambda >0$, its resolvent $(I+\lambda T)^{-1}$ is single-valued and $\frac{1}{2}$-averaged and thus also nonexpansive. If $T$ is strongly monotone, then$ (I+\lambda T)^{-1}$ is contractive.}

\cut{\begin{definition} Let $T:\cH\to\cH$ be a \emph{nonexpansive} or \emph{averaged} operator such that $\Fix T\not=\emptyset$. Let $\lambda\in(0,1]$. The Krasnosel'ski\u i--Mann (KM) iteration generates the sequence $(x^k)_{k\ge 0}$ by
$$ x^{k+1} = x^k + \lambda (Tx^k-x^k).$$
\end{definition}
The sequence $(x^k)_{k\ge 0}$ converges weakly to a point in $\Fix T$, and $(Tx^k-x^k)_{k\ge 0}$ converges strongly to 0.

Let us illustrate the above concepts using a proper closed convex function $f:\cH\to\RR\cup\{+\infty\}$. The following results are well known (cf.  textbook \cite{bauschke2011convex}). The subdifferential $\partial f$ is set-valued and~maximally monotone, and thus  $(I+\lambda\partial f)^{-1}$ is single-valued and  $\frac{1}{2}$-averaged. For any $x\in\cH$, $(I+\lambda\partial f)^{-1}x=\argmin_y f(y)+\frac{1}{2\lambda}\|y-x\|^2$. Therefore, the proximal-point algorithm is a  KM iteration.

If $f$ is differentiable and $\nabla f$ is $L$-Lipschitz, then $\nabla f$ is $\frac{1}{L}$-cocoercive and, by \cite[Prop. 4.33]{bauschke2011convex}, the gradient descent operator $(I-\gamma\nabla f)$ is $\frac{\gamma L}{2}$-averaged, where $\gamma$ satisfies $0<\gamma <\frac{2}{L}$ so that $0<\frac{\gamma L}{2}<1$. Therefore, the gradient descent iteration  is a KM iteration.

In addition, other special cases of the KM iteration are the forward-backward splitting (FBS) iteration \cite{passty1979ergodic}, Douglas-Rachford splitting (DRS) iteration \cite{lions1979splitting}, relaxed Peaceman-Rachford splitting (PRS) \cite{lions1979splitting}, and a three-operator splitting iteration \cite{davis2015three}, which involve multiple nonexpansive or averaged operators. They have numerous applications in many areas of computational mathematics.
}
\cut{If $f$ is $\mu$-strongly convex, then $\partial f$ is $\mu$-strongly monotone and, by \cite[Prop. 23.11]{bauschke2011convex}, $(I+\lambda\partial f)^{-1}$ is $\frac{1}{1+\lambda\mu}$-contractive. In this case, the proximal-point algorithm is a  Picard iteration (of a contractive operator). }

\end{subsection}

\begin{section}{Applications}\label{sec:application}
In this section, we provide some applications that are special cases of the fixed-point problem~\eqref{eqn:fix_point_set}. For each application, we identify its  nonexpansive operator $T$ (or the corresponding operator $S$) and implement the conditions in Theorem~\ref{thm:simple_convergence}. For simplicity, we use the uniform distribution, $p_1=\cdots=p_m=1/m$, and apply the simpler update~\eqref{eqn:simple_asyn_update} instead of~\eqref{eqn:asyn_update}. 

\begin{subsection}{Solving linear equations}\label{sec:linearequations}
Consider the linear system $Ax = b,$
where $A \in \mathbb{R}^{m\times m}$ is a nonsingular matrix with nonzero diagonal entries. Let $A = D + R$, where $D$ and $R$ are the diagonal and off-diagonal parts of $A$, respectively. Let $M:=-D^{-1}R$ and $T(x):=M x + D^{-1} b$. Then the system $Ax = b$ is equivalent to the fixed-point problem $x = D^{-1} (b - R x)= :T(x)$, where $T$ is nonexpansive if the spectral norm $\|M\|_2$ satisfies $\|M\|_2\leq 1$.
\cut{Define the relaxed operator as
$$T_{\alpha}(x) = \big((1 - \alpha) I + \alpha \, T \big) x,$$
where $\alpha \in (0, 1)$. Note that $T_{\alpha}$ is $\alpha$-averaged, $x^{k+1} = T_{\alpha} (x^k)$ gives the relaxed Jacobi iteration.}The iteration $x^{k+1}= T(x^k)$ is widely known as the Jacobi algorithm. Let $S = I -T$. Each update  $S_{i_k}\hat{x}^k$ involves multiplying just the $i_k$th  row of $M$ to $x$ and adding the $i_k$th entry of $D^{-1} b$, so we arrive at the following algorithm.

\begin{algorithm}[H]\label{alg:asyn_linear_eqn}
\SetKwInOut{Input}{Input}\SetKwInOut{Output}{output}
\Input{$x^0 \in \mathbb{R}^n$,  $K>0$.}
set the global iteration counter $k=0$\;
 \While{$k < K$, every agent asynchronously and continuously}{
  select $i_k\in\{1,\ldots, m\}$ uniformly at random\;
  subtract $\frac{\eta_k}{a_{i_k i_k}} (\sum_{j} a_{i_k j} \, \hat x_j^k - b_{i_k})$ from  the component $x_{i_k}$ of the variable $x$\;

  update the global counter $k \leftarrow k+1$\;
 }
 \caption{ARock for linear equations}
\end{algorithm}

\begin{proposition}\label{prop:lip_to_sm}
{\cite[Example~22.5]{bauschke2011convex}} Suppose that $T$ is $c$-Lipschitz continuous with $c \in [0, 1)$. Then, $I - T$ is $(1 - c)$-strongly monotone.
\end{proposition}

Suppose $\|M\|_2<1$. Since $T$ is $\|M\|_2$-Lipschitz continuous, by Proposition~\ref{prop:lip_to_sm}, $S$ is $(1-\|M\|_2)$-strongly monotone. By Theorem~\ref{thm:strongly_monotone}, Algorithm~\ref{alg:asyn_linear_eqn} converges linearly. 

\cut{\begin{lemma}\label{lem:arock-lin}
Assume that $A$ is symmetric positive definite. Then, there exists a constant $\mu>0$ such that
\begin{equation}\label{eq:S-lin} \langle S(x)-S(y), x - y \rangle \ge \mu\|x-y\|^2,\quad\forall x,
\end{equation}
where $S(x)=x-(M x + D^{-1} b)$ is  defined in this subsection, i.e., $S$ is $\mu$-strongly monotone.
\end{lemma}
\begin{proof}
Note
\begin{align*}
S(x)=x-(Mx+D^{-1}b)=x-\big(-D^{-1} (A-D)x+D^{-1}b\big)=D^{-1}(Ax-b).
\end{align*}
Hence,~\eqref{eq:S-lin} reduces to
$$ \langle D^{-1}A (x - y), x - y \rangle \geq \mu \|x - y\|^2$$
which holds for $\mu=  \lambda_{\min} (D^{-1} A) $ with $\lambda_{\min} (D^{-1}A) $ being the smallest eigenvalue of $D^{-1} A$. This completes the proof.\hfill
\end{proof}}

\cut{\begin{align}\label{eqn:linear_eqn_solver}
x^{k+1} &= x^k - \eta_k \, U_k( S({\hat x}^{k}) ) \nonumber \\
             &= x^{k} - \eta_k \, U_k( \alpha (I - T)\hat x^{k} ) \\
             &=x^{k} - \eta_k \, \alpha \, U_k (\hat x^{k} - M \hat x^{k} - D^{-1} b) \nonumber.
\end{align}
Since $T_{\alpha}$ is $\alpha$-averaged, Theorem ~\ref{thm:convergence} guarantees the convergence of iteration~\eqref{eqn:linear_eqn_solver}.} \cut{When $A$ is symmetric, we know that $\rho(|A|) \leq \rho(A) = \|A\|_2$, so our convergence condition is weaker compared to the existing result by \cite{chazan1969chaotic} which requires $\rho(|M|) < 1$. In addition, when the smallest eigenvalue of $I - M$ is strictly positive, Theorem~\ref{thm:strongly_monotone} implies the linear convergence of our algorithm.} \cut{, then if $\lambda_{\min}>0$, then we have
$$\big\langle S(x) - S(y), x - y \big\rangle = \alpha \big\langle (I - M) (x-y), x - y \big\rangle \geq \alpha \,\lambda_{\min} \|x - y\|^2,$$
i.e., $S$ is strongly monotone with modulus $\alpha \, \lambda_{\min}$. Hence, if  $\lambda_{\min} > 0$, then the sequence $(x^k)$ generated by~\eqref{eqn:linear_eqn_solver} converges to the linear equation solution with linear rate.}
\end{subsection}

\cut {\begin{subsection}{Solving nonlinear equations}
Consider the following nonlinear system
\begin{equation}
\min_{x\in {\cH}} F(x) = 0,
\end{equation}
\end{subsection}

\begin{subsection}{Solving ordinary differential equations (ODE)}
Consider the following nonlinear system
\begin{equation}
\min_{x\in {\cH}} F(x) = 0,
\end{equation}

\end{subsection}
}

\begin{subsection}{Minimize convex smooth function}\label{subsec:smooth}
Consider the optimization problem
\begin{equation}\label{eq:probmin}
\Min_{x\in {\cH}} f(x),
\end{equation}
where $f$ is a closed proper convex differentiable function and $\nabla f$ is $L$-Lipschitz continuous, $L>0$. Let $S := \frac{2}{L} \, \nabla f$. \cut{ By the Baillon-Haddad theorem, $S$ is $\frac{1}{\alpha L}$-cocoercive\footnote{An operator $S:\cH\to\cH$ is \emph{$\beta$-cocoercive},  where $\beta>0$, if it satisfies $\dotp{x - y, Tx - Ty} \geq \beta \|Tx - Ty\|^2,~\forall x, y \in \cH.$}, and by Lemma~\ref{lemma:a-avg}, $T = I - S$ is $\frac{\alpha L}{2}$-averaged, where $\frac{\alpha L}{2}<1$.} As $f$ is  convex and differentiable, $x$ is a minimizer of $f$ if and only if $x$ is a zero of $ S$. Note that $S$ is $\frac{1}{2}$-cocoercive. By Lemma \ref{lemma:a-avg}, $T \equiv I-S$ is nonexpansive. Applying ARock, we have the following iteration:
\begin{equation}\label{eqn:cvx_itr}
x^{k+1} = x^k - \eta_k  S_{i_k} \hat{x}^{k},
\end{equation}
where $S_{i_k} x = \frac{2}{L}(0, ..., 0, \nabla_{i_k}f(x), 0, ..., 0)^T$. Note that $\nabla f$ needs a structure that makes it cheap to compute $\nabla_{i_k} f(\hat{x}^{k})$. 
Let us give two such examples
: (i) quadratic programming: $f(x)= \frac{1}{2} x^T A x - b^T x$, where $\nabla f(x) = Ax - b$ and $\nabla_{i_k} f(\hat{x}^{k})$ only depends  on a part of $A$ and $b$; (ii) sum of sparsely supported functions: $f=\sum_{j=1}^N  f_j$ and $\nabla f = \sum_{j=1}^N \nabla f_j$, where each  $f_j$ depends on just a few variables.

Theorem~\ref{thm:convergence} below guarantees the convergence of $(x^k)_{k\geq0}$ if $\eta_k \in [\eta_{\min}, \frac{1}{ 2\tau/\sqrt{m} +1})$. In addition, If $f(x)$ is \emph{restricted strongly convex}, namely, for any $x \in \cH$ and $x^* \in X^*$, where $X^*$ is the  solution set to~\eqref{eq:probmin}, we have
$\langle x- x^*, \nabla f(x) \rangle \geq \mu \|x - x^*\|^2$
for some $\mu>0$, then $S$ is quasi-strongly monotone with modulus $\mu$. According to Theorem~\ref{thm:strongly_monotone}, iteration~\eqref{eqn:cvx_itr} converges at a linear rate if the step size meets the condition therein.

Our convergence and rates are given in term of  the distance to the solution set $X^*$. In comparison, the results in the work \cite{liu2013asynchronous} are given in terms of objective error under the assumption of a uniformly bounded $(x^k)_{k\ge 0}$. In addition, their step size decays like $O(\frac{1}{\tau \rho^\tau})$ for some $\rho > 1$ depending on $\tau$, and our $O(\frac{1}{\tau})$ is better. Under similar assumptions, Bertsekas and Tsitsiklis~\cite[Section 7.5]{bertsekas1989parallel} also describes an algorithm for~\eqref{eq:probmin} and proves  only subsequence convergence~\cite[Proposition 5.3]{bertsekas1989parallel}  in $\RR^n$. 
\end{subsection}

\begin{subsection}{Decentralized consensus optimization}\label{sec:dgd}
Consider that $m$ agents in a connected  network   solve the  consensus  problem of minimizing $\sum_{i=1}^m f_i (x)$, where $x\in\RR^d$ is the shared variable and the  convex differentiable function $f_i$ is held privately by agent $i$. We assume that $\nabla f_i$ is $L_i$-Lipschitz continuous for all $i$. A decentralized gradient descent algorithm \cite{nedic2009distributed} can be developed based on the equivalent formulation
\begin{equation}\label{eqn:dect_problem_org}
\Min_{x_1, ..., x_m \in \mathbb{R}^d}  ~ \textstyle f(\vx) := \sum_{i=1}^m f_i (x_i),\quad \St~  W\vx = \vx,
\end{equation}
where $\vx = (x_1, ..., x_m)^T\in \RR^{m\times d}$ and $W \in \mathbb{R}^{m\times m}$ is the so-called mixing matrix satisfying: $W\vx = \vx$ if and only if $x_1=\cdots=x_m$. For  $i\not=j$, if $w_{i,j} \neq 0$, then agent $i$ can communicate with agent $j$; otherwise they cannot. We assume that $W$ is  symmetric and doubly stochastic. Then, the decentralized consensus algorithm~\cite{nedic2009distributed} can be expressed as $\vx^{k+1}=W\vx^{k}-\gamma\nabla f(\vx^k)=\vx^k-\gamma(\nabla f(\vx^k)+\frac{1}{\gamma}(I-W)\vx^k)$, where $\nabla f(\vx)\in \RR^{m\times d}$ is a matrix with its $i$th row equal to $(\nabla f_i(x_i))^T$; see~\cite{yuan2013convergence}. The computation of $W\vx^k$ involves communication between agents, and $\nabla f_i(x_i)$ is independently computed by each agent $i$. The  iteration is equivalent to  the gradient descent iteration applied to $\min_{\vx}  \sum_{i=1}^m f_i (x_i) + \frac{1}{2\gamma} \vx^T(I - W)\vx$. To apply our algorithm, we let $S := \frac{2}{L} \nabla F= \frac{2}{L} (\nabla f+\frac{1}{\gamma}(I-W))$ with $L = \max_{i} L_i + (1 - \lambda_{\min}(W))/\gamma$, where $\lambda_{\min} (A)$ is the smallest eigenvalue of $W$. Computing $S_{i}\hat{\vx}^k$ reduces to computing  $\nabla f_i(\hat{x}_i^k)$ and the $i$th entry of $W\hat{\vx}^k$ or $\sum_{j}w_{i,j}\hat{x}^k_j$, which involves only $\hat{x}_i^k$ and $\hat x_j^k$ from the neighbors of agent $i$. Note that since each agent $i$ can store its own $x_i$ locally, we have $\hat{x}^k_i\equiv x^k_i$.

If the agents are $p$ independent Poisson processes and that each agent $i$ has activation rate $\lambda_i$, then the probability that agent $i$ activates before other agents is equal to $\frac{\lambda_i}{\sum_{i=1}^p \lambda_i}$~\cite{larson1981urban} and therefore our random sample scheme holds and ARock applies naturally. The algorithm is summarized as follows:

\begin{algorithm}[H]\label{alg:asyn_dgd}
\SetKwInOut{Input}{Input}\SetKwInOut{Output}{output}
\Input{Each agent $i$ sets $x_i^0\in\RR^d$, $K>0$.}
 \While{$k < K$}{
  when an agent $i$ is activated, $x^{k+1}_{i} = x^k_{i} -  \frac{\eta_k}{L} \, (\nabla f_{i}( x_{i}^k) + \frac{1}{\gamma} (x^k_{i}  -  \sum_j w_{i,j} \, \hat x_j^k) )$\;
  increase the global counter $k \leftarrow k+1$\;
 }
 \caption{ARock for decentralized optimization~\eqref{eqn:dect_problem_org}}
\end{algorithm}

\end{subsection}

\begin{subsection}{Minimize smooth $+$ nonsmooth functions}\label{subsec:nonsmooth}
Consider the problem
\begin{equation}\label{eqn:prob_fbs}
\Min_{x\in \cH} f(x) + g(x),
\end{equation}
where $f$ is closed proper convex  and $g$ is convex and $L$-Lipschitz differentiable with $L>0$. Problems in the form  of~\eqref{eqn:prob_fbs} arise in statistical regression, machine learning, and signal processing and include well-known problems such as the  support vector machine, regularized least-squares, and regularized logistic regression. For any $x \in \cH$ and scalar $\gamma\in (0, \frac{2}{L})$,   define the proximal operator $\prox_f: \cH \rightarrow \cH$ and the reflective-proximal operator $\refl_f: \cH \rightarrow \cH$ as
\begin{equation}\label{eqn:prox}
\prox_{\gamma f} (x) := \argmin_{y \in \cH} f(y) + \frac{1}{2\gamma} \|y - x\|^2 \quad \text{~and~} \quad\refl_{\gamma f} := 2\prox_{\gamma f} - I_{\cH},
\end{equation}
respectively, and define the following forward-backward operator $\TFBS := \prox_{\gamma f}  \circ (I - \gamma \nabla g).$
Because $\prox_{\gamma f}$ is $\frac{1}{2}$-averaged and $(I - \gamma \nabla g)$ is $\frac{\gamma L}{2}$-averaged,  $\TFBS$ is $\alpha$-averaged for
$\alpha \in [\frac{2}{3}, 1)$ \cite[Propositions 4.32 and 4.33]{bauschke2011convex}.
 Define
$S := I - \TFBS = I - \prox_{\gamma f}  \circ (I - \gamma \nabla g)$.   When we apply Algorithm~\ref{alg:asyn_core} to $T=\TFBS$ to solve ~\eqref{eqn:prob_fbs}, and assume $f$ is separable in  all coordinates, that is, $f(x) = \sum_{i=1}^mf_i(x_i)$,  the update for the $i_k$th selected coordinate is
\begin{equation}\label{asyncFBS}
x_{i_k}^{k+1} = x_{i_k}^k - \eta_k  \,   \big( \hat{x}_{i_k}^k - \prox_{\gamma f_{i_k}}( \hat{x}_{i_k}^k - \gamma \nabla_{i_k} g (\hat{x}^{k}))\big),
\end{equation}
Examples of separable functions include $\ell_1$ norm, $\ell_2$ norm square, the Huber function, and the indicator function of box constraints, i.e., ${\{x| a_i \leq x_i \leq b_i,~\forall i \}}$.  They all have simple $\prox$ maps.
If $\eta_k \in [\eta_{\min}, \frac{1}{2\tau/\sqrt{m} +1})$, then the convergence   is guaranteed by Theorem ~\ref{thm:convergence}. To show  linear convergence, we need to assume that $g(x)$ is strongly convex. Then, Proposition~\ref{prop:prox} below shows that $\prox_{\gamma f}  \circ (I - \gamma \nabla g)$ is a quasi-contractive operator, and  by Proposition~\ref{prop:lip_to_sm}, operator $I - \prox_{\gamma f}  \circ (I - \gamma \nabla g)$ is quasi-strongly monotone. Finally, linear convergence and its rate follow from Theorem~\ref{thm:strongly_monotone}.
\begin{proposition}\label{prop:prox}
Assume that $f$ is a closed proper convex function, and $g$ is $L$-Lipschitz differentiable and strongly convex with modulus $\mu > 0$. Let $\gamma \in (0, \frac{2}{L})$. Then, both $I - \gamma \nabla g$ and $\prox_{\gamma f}  \circ (I - \gamma \nabla g)$ are quasi-contractive operators.
\end{proposition}
\begin{proof}
We first show that $I - \gamma \nabla g$ is a quasi-contractive operator. Note
\begin{align*}
~ &\|(x - \gamma \nabla g (x)) - (x^* - \gamma \nabla g(x^*))\|^2 \\
= & \|x - x^*\|^2 - 2\gamma \langle x - x^*, \nabla g(x) - \nabla g(x^*) \rangle + \gamma^2 \|\nabla g(x) - \nabla g(x^*)\|^2 \\
\leq & \|x - x^*\|^2 - \gamma(2 - \gamma L)  \langle x - x^*, \nabla g(x) - \nabla g(x^*) \rangle \\
\leq & (1 - 2\gamma \mu + {\mu \gamma^2}L) \|x - x^*\|^2,
\end{align*}
where the first inequality follows from the  Baillon-Haddad theorem\footnote{Let $g$ be a convex differentiable function. Then, $\nabla g$ is $L$-Lipschitz if and only if it is $\frac{1}{L}$-cocoercive.} and the second one from the strong convexity of $g$.
Hence, \cut{we have
\begin{equation*}
\|(x - \gamma \nabla g (x)) - (x^* - \gamma \nabla g(x^*))\| \leq \sqrt{1 - 2\gamma \mu + {\mu \gamma^2}L} \, \|x - x^*\|,
\end{equation*}
which means that }$I - \gamma \nabla g$ is quasi-contractive if $0<\gamma<2/L$. Since $f$ is convex, $\prox_{\gamma f}$ is firmly nonexpansive, and thus we immediately have the quasi-contractiveness of $\prox_{\gamma f}  \circ (I - \gamma \nabla g)$ from that of $I-\gamma\nabla g$.
\hfill\end{proof}


\cut{\commwy{I think we can remove Remarks 1 and 2, and cite our forthcoming 2nd paper here.}
\begin{remark}
When $f = 0$, then~\eqref{eqn:fbs_itr} reduces to~\eqref{eqn:cvx_itr}. A special case is when $f(x)$ is an indicator function of a feasible set $S \subset \cH$. In this case $\prox_{\gamma f} = \Proj_{S}$, which is a projection to the feasible set $S$.
\end{remark}
\begin{remark}
The parallel asynchronous forward backward splitting requires the proximal operator $\prox_{\gamma f}$ can be component-wisely (or block-wisely) evaluated. Examples of those functions include $\ell_1$ norm, $\|x\|_{2,1}$ norm, huber function, $\|x\|_2^2$, and indicator function of box constraints $\cI_{\{x| a \leq x \leq b \} }(x)$. However, when $\prox_{\gamma f}$ is not separable, then evaluating $U_i(\prox_{\gamma f}  \circ (I - \gamma \nabla g)) (x)$ has the same complexity as evaluating $\prox_{\gamma f}  \circ (I - \gamma \nabla g) (x)$, which is not efficient. To avoid this, we can actually consider the backward forward splitting operator
\begin{equation*}
T_{FBS} :=   (I - \gamma \nabla g) \circ \prox_{\gamma f},
\end{equation*}
then we have the following iteration
\begin{equation}\label{eqn:bfs_itr}
z^{k+1} = z^k - \eta_k  \, U_k \Big( \hat z^k -  (I - \gamma \nabla g) \circ \prox_{\gamma f} (\hat{z}^{k}) \Big),
\end{equation}
then the solution can be recovered by $x^{k} = \prox_{\gamma f} (z^k)$. Applications of backward forward splitting doesn't require $\prox_{\gamma f}$ to be separable, however, it requires the proximal operator to be easy to compute, and its complexity is smaller than the evaluating $x_i - \gamma \nabla_i g (x)$, since each update requires evaluating $\prox_{\gamma f}$ once. Applications of backward splitting include $f(x)$ being the indicator function of some constraints such as $\{x ~|~ \|x\|_1 \leq 1\}$, $\{x ~|~ \|x\|_{2} \leq 1\}$, $\{ x ~|~ \sum_{i=1}^n x_i = 1, x_i \geq 0\}$. The proximal operators of those indicator functions can be easily evaluated with low computational complexity. If evaluating forward operator dominants the computational costs, we can expect linear speedup for the asynchronous algorithm.
\end{remark}
}

\end{subsection}

\begin{subsection}{Minimize nonsmooth $+$ nonsmooth functions}
Consider
\begin{equation}\label{fpg}
\Min_{x\in \cH} f(x) + g(x),
\end{equation}
where both $f(x)$ and $g(x)$ are closed proper convex and their $\prox$ maps are easy to compute. Define the Peaceman-Rachford \cite{lions1979splitting} operator:
$$\TPRS :=~\refl_{\gamma f}  \circ~\refl_{\gamma g}.$$
Since both $\refl_{\gamma f}$ and $\refl_{\gamma g}$ are nonexpansive, their composition $\TPRS$ is also nonexpansive.\cut{ Based on $T_{PRS}$ we define the following relaxed PRS operator
$$(T_{PRS})_{\alpha} := (1 - \alpha) \, I + \alpha \, T_{PRS},$$
which is a $\alpha$-averaged operator.} Let
$S :=  I - \TPRS.$
When applying ARock to $T=\TPRS$ to solve problem~\eqref{fpg}, the update~\eqref{eqn:asyn_ortho_update} reduces to:
\beq\label{aprs}z^{k+1} = z^k - \eta_k \, U_{i_k} \circ \big( I - ~\refl_{\gamma f}  \circ~\refl_{\gamma g}\big) \hat{z}^{k},
\eeq
where we use $z$ instead of $x$ since the limit $z^*$ of $(z^k)_{k\ge 0}$ is not a solution  to~\eqref{fpg}; instead, a solution  must be recovered via $x^*=\prox_{\gamma g}z^*$.
The convergence follows from Theorem ~\ref{thm:convergence} and that $\TPRS$ is nonexpansive. If either $f$ or $g$ is strongly convex, then $\TPRS$ is contractive and thus by Theorem~\ref{thm:strongly_monotone}, ARock converges linearly. Finer convergence rates follow from \cite{DavisYin2014,DavisYin2014c}. A naive implementation of~\eqref{aprs} is
\begin{subequations}\label{eqn:relaxed_prs}
\begin{align}
\hat x^{k} &= \prox_{\gamma g} (\hat z^{k}), \label{eqn:relaxed_prs-a} \\
\hat y^{k} & = \prox_{\gamma f} (2 \hat x^{k} - \hat z^{k}),  \label{eqn:relaxed_prs-b}\\
z^{k+1} & = z^k + 2\eta_k \, U_{i_k} (\hat y^{k} - \hat x^{k}),\label{eqn:relaxed_prs-c}
\end{align}
\end{subequations}
where $\hat x^{k}$ and $\hat y^{k}$ are intermediate variables. Note that the order in which the proximal operators are applied to  $f$ and $g$  affects both  $z^k$ \cite{YanYin2014} and whether coordinate-wise updates can be efficiently computed.
Next, we present two special cases of~\eqref{fpg} in Subsections~\ref{sec:Feasibility} and~\ref{sec:async-admm} and discuss how to efficiently implement the update~\eqref{eqn:relaxed_prs}.

\begin{subsubsection}{Feasibility problem}\label{sec:Feasibility}
Suppose that $C_1, ..., C_m$ are closed convex subsets of $\cH$ with a nonempty intersection. The problem is to find a point in the intersection. Let $\cI_{C_i}$ be the indicator function of the set $C_i$, that is, $\cI_{C_i}(x)=0$ if $x\in C_i$ and $\infty$ otherwise. The feasibility problem can be formulated as the following
\begin{align}
\Min_{x=(x_1,\ldots,x_m)\in \cH^m} ~& \sum_{i=1}^m \cI_{C_i}(x_i) + \cI_{\{x_1 = \cdots =x_m\}}(x)  \nonumber.
\end{align}
Let $z^k=(z^k_1,\ldots,z^k_m)\in \cH^m$, $\hat z^k=(\hat z^k_1,\ldots,\hat z^k_m)\in \cH^m$, and $\hat{\bar{z}}^k\in\cH$. We can implement~\eqref{eqn:relaxed_prs} as follows (see Appendix \ref{sec:appendix} for the step-by-step derivation):
\begin{subequations}\label{eqn:relaxed_prs_set}
\begin{align}
\label{eqn:relaxed_prs_set:a}
{\hat {\bar z}}^k &=\textstyle \frac{1}{m} \sum_{i=1}^m \hat z_i^k,  \\
\label{eqn:relaxed_prs_set:b}
\hat{y}_{i_k}^k & = \textstyle\Proj_{C_{i_k}} \big(2 {\hat {\bar z}}^k - \hat z^{k}_{i_k}\big),\\
\label{eqn:relaxed_prs_set:c}
z_{i_k}^{k+1} & = \textstyle z_{i_k}^k + 2\eta_k  (\hat y_{i_k}^{k} - {\hat {\bar z}}^k).
\end{align}
\end{subequations}
The update~\eqref{eqn:relaxed_prs_set} can be  implemented as follows. Let global memory hold $z_1,\ldots,z_m$, as well as  $\bar{z}=\frac{1}{m}\sum_{i=1}^mz_i$. At the $k$th update, an agent independently generates a random number $i_k\in\{1,\ldots,m\}$,  then reads $z_{i_k}$ as $\hat z_{i_k}^k$ and $\bar{z}$ as ${\hat {\bar z}}^k$, and finally computes $\hat{y}_{i_k}$ and updates $z_{i_k}$ in global memory according to~\eqref{eqn:relaxed_prs_set}. Since $\bar{z}$ is maintained in global memory, the agent  updates $\bar{z}$ according to $\bar{z}^{k+1} = \bar{z}^k+\frac{1}{m}(z_{i_k}^{k+1} - z_{i_k}^{k})$. This implementation saves each agent from computing~\eqref{eqn:relaxed_prs_set:a} or reading all $z_1,\ldots,z_m$. Each agent only reads $z_{i_k}$ and $\bar{z}$, executes~\eqref{eqn:relaxed_prs_set:b}, and updates  $z_{i_k}$ \eqref{eqn:relaxed_prs_set:c} and $\bar{z}$. 
\end{subsubsection}

\end{subsection}

\begin{subsection}{Async-parallel ADMM}\label{sec:async-admm}
This is another application of~\eqref{eqn:relaxed_prs}. Consider 
\beq \label{eqn:admm_form}
\Min_{x \in \cH_1,~ y \in \cH_2}~  f(x) + g(y)\quad\St~  Ax + By = b,
\eeq
where $\cH_1$ and $\cH_2$ are Hilbert spaces, $A$ and $B$ are bounded linear operators. We apply the update~\eqref{eqn:relaxed_prs} to the Lagrange dual of~\eqref{eqn:admm_form} (see~\cite{gabay1983chapter} for the derivation):
\begin{equation}\label{eq:lag-dual}
\Min_{w \in \cG}\, d_f(w) + d_g(w),
\end{equation}
where $d_f(w) := f^* (A^* w)$, $d_g(w) := g^*(B^* w) - \langle w, b \rangle$, and $f^*$ and $g^*$ denote the convex conjugates of $f$ and $g$, respectively. The proximal maps induced by  $d_f$ and $d_g$ can be computed via solving subproblems that involve only the original terms in~\eqref{eqn:admm_form}:  $z^+=\prox_{\gamma d_f}(z)$ can be computed by (see Appendix \ref{sec:appendix} for the derivation)
\begin{equation}\label{dualf}
\begin{cases}
x^+ \in \arg\min_x f(x) - \langle z, Ax \rangle + \frac{\gamma}{2} \|Ax\|^2, \\
z^+  = z - \gamma Ax^+,
\end{cases}
\end{equation}
and $z^+=\prox_{\gamma d_g}(z)$  by
\begin{equation}\label{dualg}
\begin{cases}
y^+ \in \arg\min_y g(y) - \langle z, By - b\rangle + \frac{\gamma}{2} \|By - b\|^2, \\
z^+  = z - \gamma (By^+ - b).
\end{cases}
\end{equation}
Plugging~\eqref{dualf} and~\eqref{dualg} into (\ref{eqn:relaxed_prs}) yields the following naive implementation
\begin{subequations}\label{eqn:asyn_admm}
\begin{align}
\label{eqn:asyn_admm:a}
\hat y^{k } &\in \arg\min_y g(y) - \langle \hat z^{k}, By - b \rangle + \frac{\gamma}{2}\|By -b\|^2, \\
\label{eqn:asyn_admm:b}
\hat w_{g}^{k} &= \hat z^{k} - \gamma (B \hat y^{k} - b), \\
\label{eqn:asyn_admm:c}
\hat x^{k} &\in \arg\min_x f(x) - \langle 2 \hat w_{g}^{k} -\hat z^{k}, Ax \rangle + \frac{\gamma}{2} \|Ax\|^2,  \\
\label{eqn:asyn_admm:d}
\hat w_{f}^{k} &=2 \hat w_{g}^{k} -\hat z^{k} - \gamma  A \hat x^k,\\
\label{eqn:asyn_admm:e}
z_{i_k}^{k+1} &= z_{i_k}^{k} + \eta_k (\hat w_{f,i_k}^{k} - \hat w_{g,i_k}^{k}).
\end{align}
\end{subequations}
Note that $2\eta_k$ in~\eqref{eqn:relaxed_prs-c} becomes $\eta_k$ in~\eqref{eqn:asyn_admm:e} because ADMM is equivalent to the Douglas-Rachford operator, which is the average of the Peaceman-Rachford operator and the identity operator~\cite{lions1979splitting}.
Under favorable structures,~\eqref{eqn:asyn_admm} can be implemented efficiently.
For instance, when $A$ and $B$ are block diagonal matrices and $f,g$ are corresponding block separable functions,  steps~\eqref{eqn:asyn_admm:a}--\eqref{eqn:asyn_admm:d} reduce to independent computation for each $i$. Since only $\hat w_{f,i_k}^{k}$ and $\hat w_{g,i_k}^{k}$ are needed to update the main variable $z^k$,  we only need to compute~\eqref{eqn:asyn_admm:a}--\eqref{eqn:asyn_admm:d} for the $i_k$th block. This is exploited in distributed and decentralized ADMM in the next two subsections.
\end{subsection}

\begin{subsubsection}{Async-parallel ADMM for consensus optimization}
Consider the consensus optimization problem:
\beq\label{eqn:consensus_problem}
\Min_{x_i, y \in \cH} ~ \textstyle\sum_{i=1}^m f_i(x_i)\quad \St~  x_i - y = 0, \quad\forall i = 1, ..., m,
\eeq
where $f_i(x_i)$ are proper close convex functions. Rewrite~\eqref{eqn:consensus_problem} to the ADMM form:
 \begin{align}\label{eqn:consensus_admm_form}
\Min_{x_i, y \in \cH} ~~& \textstyle\sum_{i=1}^m f_i(x_i)  + g(y)\nonumber \\
                        \St~ &
                        \begin{bmatrix} I_{\cH} & 0 & \cdots &0\\
                                0 & I_{\cH} &\cdots & 0\\
                                & & \ddots & \\
                                0 & 0 &\cdots & I_{\cH}
                        \end{bmatrix}
                        \begin{bmatrix} x_1 \\ x_2\\ \vdots \\x_m\\
                        \end{bmatrix}
                        -
                        \begin{bmatrix} I_{\cH} \\ I_{\cH}\\ \vdots \\ I_{\cH}\\
                        \end{bmatrix}
                        y = 0,
\end{align}
where $g = 0$. Now apply the async-parallel ADMM~\eqref{eqn:asyn_admm} to~\eqref{eqn:consensus_admm_form} with dual variables  $z_1, ..., z_m \in \cH$. In particular, the update~\eqref{eqn:asyn_admm:a}, \eqref{eqn:asyn_admm:b}, \eqref{eqn:asyn_admm:c}, \eqref{eqn:asyn_admm:d} reduce to
\begin{align}\label{eqn:async-par-final}
\hat{y}^k&=\textstyle\argmin_y \big\{ \sum_{i=1}^m \langle \hat z_i^k, y \rangle + \frac{\gamma m}{2} \|y\|^2 \big\}= -\frac{1}{\gamma m} \sum_{i=1}^m \hat z_i^k \cr
(\hat w_{d_g}^{k})_i &= \hat z_i^{k} + \gamma \, \hat y^k \\
\hat x_i^{k} &= \textstyle\argmin_{x_i} \big\{ f_i(x_i) - \langle 2 (\hat w_{d_g}^{k})_i -\hat z_i^{k}, x_i \rangle + \frac{\gamma}{2} \|x_i\|^2 \big\}, \cr
(\hat w_{d_f}^{k})_i &=2 (\hat w_{d_g}^{k})_i -\hat z_i^{k} - \gamma \, \hat x_i^k \nonumber
\end{align}
Therefore, we obtain the following async-parallel ADMM algorithm for  the  problem~\eqref{eqn:consensus_problem}. This algorithm applies to all the distributed applications  in~\cite{boyd2011distributed}. 

\begin{algorithm}[H]\label{alg:asyn_consensus}
\SetKwInOut{Input}{Input}\SetKwInOut{Output}{output}
\Input{set shared variables $y^0, z_i^0, ~\forall i$, and $K>0$.}
 \While{$k < K$ every agent asynchronously and continuously}{
  choose $i_k$ from $\{1, ..., m\}$ with equal probability\;
  evaluate $ (\hat w_{d_g}^{k})_{i_k}$, $\hat x_{i_k}^k$, and $(\hat w_{d_f}^{k})_{i_k}$ following~\eqref{eqn:async-par-final}\;
  update $z_{i_k}^{k+1} = z_{i_k}^{k} +  \eta_k \,  ((\hat w_{d_f}^{k})_{i_k} - (\hat w_{d_g}^{k})_{i_k})$\;
  update $ y^{k+1} =  y^k + \frac{1}{\gamma m} (z_{i_k}^k - z_{i_k}^{k+1})$\;
  update the global counter $k \leftarrow k+1$\;
 }
 \caption{ARock for consensus optimization}
\end{algorithm}

\begin{subsubsection}{Async-parallel ADMM for decentralized optimization}\label{subsec:async_admm_des}
Let $V=\{1, ..., m\}$ be a set of agents and $E = \{(i, j) ~|~ \text{if agent $i$ connects to agent $j$}, i < j\}$ be the set of undirected links between the agents.  Consider the following decentralized consensus optimization problem on the graph $G= (V, E)$:
\begin{equation}\label{eqn:des_prob}
\Min_{x_1, ..., x_m \in \mathbb{R}^d}  ~ \textstyle f(x_1, \ldots, x_m) := \sum_{i=1}^m f_i (x_i),\quad
                \St~  x_i = x_j, ~ \forall (i,j) \in E,
\end{equation}
where $x_1, ..., x_m \in \mathbb{R}^d$ are the local variables and each agent can only communicate with its neighbors in $G$. By introducing the auxiliary variable $y_{ij}$ associated with each edge $(i, j) \in E$, the problem~\eqref{eqn:des_prob} can be reformulated as:
\begin{equation}\label{eqn:des_prob2}
\Min_{x_i, y_{ij}} ~\textstyle\sum_{i=1}^m f_i (x_i),\quad \St~  x_i = y_{ij}, ~x_j = y_{ij}, \quad \forall (i, j) \in E.
\end{equation}
Define $x = (x_1, ..., x_m)^T$ and $y = (y_{ij})_{(i, j) \in E} \in \mathbb{R}^{|E|d}$ to rewrite~\eqref{eqn:des_prob2} as
\begin{equation}\label{eqn:des_prob3}
\Min_{x, y} ~ \textstyle\sum_{i=1}^m f_i (x_i),\quad \St ~ Ax + By = 0,
\end{equation}
for proper matrices $A$ and $B$.
Applying the async-parallel ADMM~\eqref{eqn:asyn_admm} to~\eqref{eqn:des_prob3} gives rise to the following simplified update:
Let $E(i)$ be the set of edges connected with agent $i$ and $|E(i)|$ be its cardinality. Let $L(i) = \{j ~|~ (j, i) \in E(i), j < i\}$ and $R(i) = \{j ~|~ (i, j) \in E(i), j > i\}$. To every pair of constraints $x_i = y_{ij}$ and $x_j = y_{ij}$, $(i,j)\in E$, we associate the dual variables $z_{ij,i}$ and $z_{ij,j}$, respectively. Whenever some agent $i$ is activated, it calculates
\begin{subequations}\label{eqn:asyn_admm_des}
\begin{align}
\label{eqn:asyn_admm_des_a}
 \hat x^{k}_{i}  = &\argmin_{x_i} f_i(x_i) +  \big(\sum_{l \in L(i)}\hat z^k_{li, l}  + \sum_{r\in R(i)}\hat z^k_{ir, r} \big) x_i +  \frac{\gamma}{2} |E(i)|\cdot \|x_i\|^2,  \\\label{eqn:asyn_admm_des_b}
 z_{li, i}^{k+1} = &z_{li, i}^k -   \eta_k (( \hat z_{li, i}^k + \hat z_{li, l})/2 +  \gamma \hat x^k_i), \quad\forall l \in L(i),\\
z_{ir, i}^{k+1} = &z_{ir, i}^k -  \eta_k (( \hat z_{ir, i}^k + \hat z_{ir, r})/2 +  \gamma \hat x^k_i), \quad\forall r \in R(i).
\label{eqn:asyn_admm_des_c}
\end{align}
\end{subequations}
We present the algorithm based on
\eqref{eqn:asyn_admm_des} for problem~\eqref{eqn:des_prob} in Algorithm~\ref{alg:asyn_decentralize}.

\begin{algorithm}[H]\label{alg:asyn_decentralize}
\SetKwInOut{Input}{Input}\SetKwInOut{Output}{output}
\Input{Each agent $i$ sets the dual variables $z^0_{e, i}=0$ for  $e \in E(i)$, $K>0$.}
 \While{$k < K$, any activated agent $i$}{
 (previously received $\hat z_{li, l}^k$ from neighbors $l \in L(i)$ and $\hat z_{ir,r}^k$ from  $r \in R(i)$)\;
 update $\hat x_i^{k}$ according to~\eqref{eqn:asyn_admm_des_a}\;
 update $z_{li, i}^{k+1}$ and $z_{ir, i}^{k+1}$ according to~\eqref{eqn:asyn_admm_des_b} and~\eqref{eqn:asyn_admm_des_c}, respectively\;
 send $z_{li, i}^{k+1}$ to neighbors $l \in L(i)$ and  $z_{ir, i}^{k+1}$ to neighbors $r \in R(i)$\;
 }
 \caption{ARock for the decentralized problem~\eqref{eqn:des_prob2}}
\end{algorithm}
Algorithm~\ref{alg:asyn_decentralize} activates one agent at each iteration and updates all the dual variables associated with the agent. In this case, only one-sided communication is needed, for sending the updated dual variables in the last step. We allow this communication to be delayed in the sense that agent $i$'s neighbors may be activated and start their computation before receiving the latest dual variables from agent $i$.

Our algorithm is different from the asynchronous ADMM algorithm by Wei and Ozdaglar~\cite{wei2013on}. Their algorithm activates an edge and its two associated agents at each iteration and thus requires two-sided communication at each activation.
We can recover their algorithm as a special case by activating an edge $(i,j)\in E$ and its associated agents $i$ and $j$ at each iteration, updating  the  dual variables $z_{ij,i}$ and $z_{ij,j} $ associated with the edge, as well as computing the intermediate variables $x_i$, $x_j$, and $y_{ij}$. The updates are derived from~\eqref{eqn:des_prob3} with the orders of $x$ and $y$ swapped. Note that \cite{wei2013on} does not consider the situation that adjacent edges are  activated in a short period of time, which may cause overlapped computation and delay communication. Indeed, their algorithm corresponds to $\tau=0$ and the corresponding stepsize $\eta_k\equiv1$. Appendix \ref{sec:appendix2} presents the steps to derive the algorithms in this subsection.
\end{subsubsection}

\cut{
\begin{subsubsection}{The generalized lasso}
The generalized Lasso is the following model
\begin{equation}
\min_{x} \lambda \, \|Dx\|_1 + \frac{1}{N}  \sum_{i=1}^N \ell_{i}(x),
\end{equation}
where $N$ is the number of samples, $x \in \mathbb{R}^n$ is the feature vector, $\ell_i (x)$ is the loss function, such as square loss, logistic loss or hinge loss. $D$ is the penalty matrix specifying the desired structured sparsity pattern of $x$. Different setting of $D$ reduces to different models such as fused lasso, trend filtering, and graph-guided fused lasso. The asynchronous ADMM can be applied by letting
$$f_i(x) =  \frac{\lambda}{M} \|Dx\|_1 + \frac{1}{N}  \sum_{j= i N/M}^{(i+1)N/M} \ell_{j}(x),$$
where $M$ is the number of workers. We assign $N/M$ samples to each worker. In practice, since the subproblems do not have closed form solution, they can be solved by using some inexact ADMM method \cite{zhang2011unified}.
\end{subsubsection}
}

\end{subsubsection}

\cut{
\begin{subsection}{Asynchronous Proximal Jacobi ADMM}
Consider the following problem,
\begin{align}\label{eqn:jacobi_admm_problem}
\min_{x_i, x} & \sum_{i=1}^m g_i(x_i) + f(x) \nonumber \\
                        \text{s.t.}~~ & A_1 x_1 + A_2 x_2 + \cdots + A_m x_m = b,
\end{align}
where $x_i \in \cH_i$, $x = (x_1, ..., x_m)^T \in \cH$, $b \in \cH_l$, $f(x):\cH \rightarrow R$ is a convex differentiable function, $g_i(x_i): \cH_i \rightarrow R$ are closed proper convex functions and $A_i: \cH_i \rightarrow \cH_l$ are linear operators. Define the Lagrangian function
\begin{equation}
L(x, \lambda) = \sum_{i=1}^m f_i(x_i) + g(x)  - \langle x_{m+1}, \sum_{i=1}^m A_i x_i - b \rangle,
\end{equation}
where $x_{m+1}$ is the Lagrangian multiplier. The optimality condition for~\eqref{eqn:jacobi_admm_problem} is the following
\begin{align}\label{eqn:jacobi_admm_kkt}
                        \begin{bmatrix} 0 \\ 0\\ \vdots \\ 0 \\ 0 \\
                        \end{bmatrix}
                         \in
                        \begin{bmatrix}
                        \nabla_1 f(x) + \partial g_1(x_1) \\
                        \nabla_2 f(x) + \partial g_2(x_2) \\
                        \vdots\\
                        \nabla_m f(x) + \partial g_m(x_m) \\
                        0 \\
                        \end{bmatrix}
                        +
                        \underbrace{
                        \begin{bmatrix} 0 & 0 & \cdots &0 &-A_1^*\\
                                0 & 0 &\cdots &0& -A_2^*\\
                                \vdots& \vdots& \vdots & \vdots &\vdots\\
                                0 & 0 &\cdots &0& -A_m^* \\
                                A_1 & A_2 & \cdots &A_m& 0 \\
                        \end{bmatrix}
                        }_{Q}
                        \begin{bmatrix} x_1 \\ x_2\\ \vdots \\x_m\\ x_{m+1} \\
                        \end{bmatrix}
                        -
                        \begin{bmatrix} 0 \\ 0\\ \vdots \\ 0 \\ b \\
                        \end{bmatrix}.
\end{align}
We know that $x$ is the optimal solution of~\eqref{eqn:jacobi_admm_problem} if and only if $x_1, ..., x_m, x_{m+1}$ satisfies~\eqref{eqn:jacobi_admm_kkt}. Define $z := (x_1, ..., x_m, x_{m+1})$, $F(z) := f(z)$, $G(z) := \sum_{i=1}^m g_i (x_i)$ and $Q$ is defined as shown in~\eqref{eqn:jacobi_admm_kkt}, then finding $z$ satisfying~\eqref{eqn:jacobi_admm_kkt} is equivalent to the following monotone inclusion problem
\begin{equation}\label{eqn:jacobi_admm_monotone}
\text{find $z$ such that } \quad 0 \in \nabla F(z) + \partial G(z) + Qz.
\end{equation}
Introduce a symmetric positive definite matrix $W$, and apply the forward backward splitting operator to this problem, then the monotone inclusion problem~\eqref{eqn:jacobi_admm_kkt} is equivalent to the following fixed point problem
\begin{equation}\label{eqn:jacobi_admm_fixed_point}
\text{find $z$ such that } z = J_{W^{-1} (\partial G + Q)} \big(z - W^{-1} \nabla F(z) \big) ,
\end{equation}
where $J_{A} := (I_{\cH} + A)^{-1}$ is called the resolvent of operator $A$. Let $T z = J_{W^{-1} (\partial G + Q)} \big(z - W^{-1} \nabla F(z) \big)$, we know that $T$ is a $\alpha$-averaged operator in the norm $\|\cdot\|_W$. Then we can let $S = I - T$, and apply the parallel asynchronous algorithm with $U_k$ as defined in Remark~\ref{remark:pd}, we have the following iteration
\begin{equation}\label{eqn:asyn_jacobi_admm_fixed_point}
z^{k+1} = z^k - \eta_k \, U_k \Big( \hat z^k - J_{W^{-1} (\partial G + Q)} \big(\hat z^k - W^{-1} \nabla F(\hat z^k) \big) \Big).
\end{equation}
Now we show that iteration~\eqref{eqn:asyn_jacobi_admm_fixed_point} corresponds to an asynchronous version of the proximal Jacobi ADMM method. Let
\begin{equation}\label{eqn:def_w}
W =
        \begin{bmatrix}
        P - \sigma A^T A & 0 \\
        0         & \frac{1}{\sigma} I \\
        \end{bmatrix},
\end{equation}
where $P = diag(P_1, ..., P_m)$, $P_i: \cH_i \rightarrow \cH_i$ are easy to invert symmetric matrices and $A = (A_1, ..., A_m)$. For convergence purposes, we require $P \succ \sigma A^T A$ so that $W$ is symmetric positive definite. The reason for choosing such a $W$ is to decouple $x_1, ..., x_m$ and $\lambda$. Define
\begin{equation}
\hat{y}^k := J_{W^{-1} (\partial G + Q)} \big(\hat z^k - W^{-1} \nabla F(\hat z^k) \big),
\end{equation}
where $\hat y^k = (\hat y_1^k, ..., \hat y_m^k, \hat y_{m+1}^k)$,
then we have
\begin{equation}\label{eqn:update_y}
W \hat y^k = W \hat z^k - \nabla F(\hat z^k) -  \tilde \nabla G(\hat y^k) - Q \hat y^k,
\end{equation}
where $\tilde \nabla G(\hat y^k) \in \partial G(\hat y^k)$. Substituting~\eqref{eqn:def_w} to~\eqref{eqn:update_y} gives the following updates
\begin{equation}\label{eqn:update_y_i}
\begin{cases}
\hat p^k = \sum_{j=1}^m A_j \hat x_j^k  \\
\hat y_i^k = \argmin_{x_i}  g_i(x_i) + \big\langle \nabla_i f(\hat x^k), x_i - \hat x_i^k \big\rangle + \big \langle \sigma A_i^T (p^k - b - \frac{\hat x_{m+1}^k}{\sigma}), x_i \big\rangle + \frac{1}{2} \|x_i - \hat x_i^k\|_{P_i}^2, \forall i = 1, ..., m \\
\hat y_{m+1}^k = \hat x_{m+1}^k - \sigma \sum_{i=1}^m A_i \hat y_i^k
\end{cases}
\end{equation}
{\color{red} [$\hat y^{k}_{m+1}$ depends on the $y^k_{i}$, for $i=1, ..., m$]}
If we let $P_i$ be a diagonal matrix, assume $\nabla_i f$ can be evaluated without calculating $\nabla f$, and assume $g_i(x_i)$ is simple functions, then the first step has closed form solution. Substituting $\hat y^k$ to~\eqref{eqn:asyn_jacobi_admm_fixed_point}, we have the following updates
\begin{equation}
z^{k+1} = z^k - \eta_k \, U_k (\hat z^k - \hat y^k).
\end{equation}

To conclude, we have the following asynchronous proximal Jacobi ADMM algorithm for solving~\eqref{eqn:jacobi_admm_problem}.

\begin{algorithm}[H]\label{alg:asyn_jacobi_admm}
\SetAlgoLined
\SetKwInOut{Input}{Input}\SetKwInOut{Output}{output}
\Input{Initial points $x_i^0, \lambda^0, y_i^0, p^0 = \sum_{i=1}^m A_i x_i^0$ are shared variables, number of iterations $K$}
 \While{$k < K$}{
  choose $i_k$ from $\{1, ..., m, m+1\}$ with equal probability\;
  read ${\hat x_{i_k}^k}$ and $\hat p^k$ from global memory\;
  evaluate $\hat y_{i_k}^k$ according to~\eqref{eqn:update_y_i}\;
  update $x_{i_k}^{k+1} =x_{i_k}^k - \eta_k \, (\hat x_{i_k}^k - \hat y_{i_k}^k)$\;
  \If{$i_k\leq m$}{
  update $p^{k+1} = p^k + A_{i_k} (x_{i_k}^{k+1} - \hat x_{i_k}^k)$\;
  }
  $k \leftarrow k+1$\;
 }
 \caption{Asynchronous proximal Jacobi ADMM for consensus optimization}
\end{algorithm}
Solve linear programming problem in standard form
\begin{align}\label{eqn:jacobi_admm_problem}
\min_{x} ~&   c^T x \nonumber \\
                        \text{s.t.}~ & A x = b \\
                                            & x \geq 0. \nonumber
\end{align}

Solve the Basis pursuit problem
\begin{align}\label{eqn:jacobi_admm_problem}
\min_{x} ~&  \|x\|_1 \nonumber \\
                        \text{s.t.}~ & A x = b.
\end{align}

Consider a classification problem. Given a datasets $\{(a_1, b_1), ..., (a_m, b_m)\}$ where $(a_i, b_i), i=1,..., m, a_i \in \mathbb{R}^n, b_i \in \{-1, 1\}$ are instance-label pairs. We use the following dual SVM method to train a classifier,
\begin{align}\label{eqn:jacobi_admm_problem}
\min_{x} ~& \frac{1}{2} x^T A x - 1^T x \nonumber \\
                        \text{s.t.}~ & \sum_{i=1}^n b_i x_i = 0 \\
                                            & 0 \leq x \leq C, \nonumber
\end{align}
where $a_{ij} = b_i b_j k(a_i, a_j)$, where $k(\cdot, \cdot)$ is a kernel function. All of the previously listed problems can be solved by using Algorithm~\ref{alg:asyn_jacobi_admm}.
\end{subsection}
}
\end{section}

\cut{\begin{section}{Assumptions}
To derive the convergence for Algorithm~\ref{alg:asyn_core}, we make the following two assumptions.

\begin{assumption}[Bounded delay] We assume that there is an upper bound $\tau$ on the delay, that is, $J(k)\subset \{k-1,k-2,\cdots,k-\tau\}$. 
\end{assumption}



\begin{assumption}[Sampling order vs updating order]
We assume that the sampling order is the same as the updating order.
\end{assumption}

\begin{remark}
Sampling order can be different from the actual updating order. However, if the delay is independent with the random choice of blocks, the actual block of coordinates will still be updated with the same probability. Hence, this assumption does not lose generality. 
\end{remark}

\end{section}
}

\begin{section}{Convergence}
We establish weak and strong convergence  in Subsection~\ref{sec:weak} and  linear convergence in Subsection~\ref{sec:linear}. Step size selection is also discussed.

\subsection{Almost sure convergence}\label{sec:weak}
\begin{assumption} Throughout the our analysis, we assume $p_{\min}:=\min_ip_i>0$ and
\begin{equation}\label{eq:independence}
\Prob(i_k=i\,|\,\cX^k)=\Prob(i_k=i)=p_i,\quad\forall i,k.
\end{equation} 
\end{assumption}
We let $|J(k)|$ be the number of elements in $J(k)$ (see Subsection~\ref{subset:aync_mem}).
Only for the purpose of analysis, we define the (never computed) full update at $k$th iteration:
\begin{align}\label{eqn:def_bar_x}
\bar{x}^{k+1} := x^k - \eta_k S\hat{x}^{k}.
\end{align}
Lemma~\ref{lemma:a-avg} below shows that $T$ is nonexpansive if and only if $S$ is ${1/2}$-cocoercive.
\begin{lemma}\label{lemma:a-avg}
Operator $T:\cH\to\cH$ is nonexpansive if and only if $S = I - T$ is {${1}/{2}$-cocoercive}, i.e., $\langle x-y,Sx- Sy\rangle \ge \frac{1}{2}\|Sx- Sy\|^2, {\forall~x,y\in\cH}$.
\end{lemma}
\begin{proof}
See textbook \cite[Proposition 4.33]{bauschke2011convex} for the proof of the ``if'' part, and the ``only if'' part, though missing there, follows by just reversing the proof.
\hfill\end{proof}

The lemma below develops an an upper bound for the expected distance between $x^{k+1}$ and any $x^*\in \Fix T$.
\begin{lemma}\label{lemma:fund}
Let $(x^k)_{k\geq 0}$ be the sequence generated by Algorithm \ref{alg:asyn_core}. Then for any $x^*\in \Fix T$ and $\gamma>0$ (to be optimized later), we have 
\begin{align}\label{eqn:fund_inquality0}
\begin{aligned}
\textstyle\mathbb{E}\big(\|x^{k+1} - x^* \|^2 \,\big|\, \cX^k\big)  \leq & \textstyle\|x^{k} - x^* \|^2  +\frac{\gamma}{m}\sum_{d\in J(k)}\|x^d-x^{d+1}\|^2\\
&\textstyle+ \frac{1}{m}\left(\frac{{|J(k)|}}{\gamma}+\frac{1}{mp_{\min}}-\frac{1}{\eta_k}\right)\|x^k-\bar x^{k+1}\|^2.
\end{aligned}
\end{align}
\end{lemma}

\begin{proof}
Recall $\Prob(i_k=i)=p_i$. Then we have
\begin{equation}\label{eqn:equality_inconsistent}
\begin{array}{rcl}
&&\mathbb{E}\left(\|x^{k+1} - x^*\|^2\,|\,\cX^k\right)\\
&\overset{\eqref{eqn:asyn_update}}=&\mathbb{E}\left(\|x^{k}  - \textstyle\frac{\eta_k}{mp_{i_k}}S_{i_k} \hat{x}^{k}- x^* \|^2\,|\,\cX^k\right)\cr
&=& \|x^k - x^*\|^2+\mathbb{E}\left(\textstyle\frac{2\eta_k}{mp_{i_k}} \left\langle S_{i_k} \hat{x}^{k}, x^* - x^k \right\rangle + \textstyle\frac{\eta_k^2}{m^2p_{i_k}^2} \|S_{i_k}\hat{x}^{k}\|^2\,\big|\,\cX^k\right)\cr
&\overset{\eqref{eq:independence}}=&\|x^k - x^*\|^2+\textstyle\frac{2\eta_k}{m} \sum_{i=1}^m\left\langle S_i \hat{x}^{k}, x^* - x^k \right\rangle + \frac{\eta_k^2}{m^2}\sum_{i=1}^m\frac{1}{p_i}\|S_i\hat{x}^{k}\|^2\cr
&=&\|x^k - x^*\|^2+\textstyle\frac{2\eta_k}{m} \left\langle S \hat{x}^{k}, x^* - x^k \right\rangle +\frac{\eta_k^2}{m^2} \sum_{i=1}^m\frac{1}{p_i} \|S_i\hat{x}^{k}\|^2.
\end{array}
\end{equation}

Note that
\begin{equation}\label{term2}\sum_{i=1}^m\frac{1}{p_i} \|S_i\hat{x}^{k}\|^2\le\frac{1}{p_{\min}} \sum_{i=1}^m\|S_i\hat{x}^{k}\|^2
=\frac{1}{p_{\min}} \|S\hat{x}^{k}\|^2\overset{\eqref{eqn:def_bar_x}}{=}\frac{1}{\eta_k^2p_{\min}}\|x^k-\bar{x}^{k+1}\|^2,\end{equation}
and
\begin{align}\label{term1}
&\langle S \hat{x}^{k}, x^* - x^k \rangle\cr
\overset{\eqref{eqn:inconsist}}=&\textstyle\langle S \hat{x}^{k},x^* - \hat{x}^k + \sum_{d\in J(k)} (x^{d} - x^{d+1})\rangle\cr
\overset{\eqref{eqn:def_bar_x}}=&\textstyle\langle S \hat{x}^{k}, x^* - \hat{x}^k\rangle + \frac{1}{\eta_k}\sum_{d\in J(k)}\langle x^k-\bar{x}^{k+1}, x^{d} - x^{d+1}\rangle\cr
\le&\textstyle \langle S \hat{x}^{k}-Sx^*, x^* - \hat{x}^k\rangle+\frac{1}{2\eta_k}\sum_{d\in J(k)}\big(\frac{1}{\gamma}\|x^k-\bar{x}^{k+1}\|^2+ \gamma\|x^{d} - x^{d+1}\|^2\big)\\
\le&\textstyle -\frac{1}{2}\|S \hat{x}^{k}\|^2+\frac{1}{2\eta_k}\sum_{d\in J(k)}(\frac{1}{\gamma}\|x^k-\bar{x}^{k+1}\|^2+ \gamma\|x^{d} - x^{d+1}\|^2)\cr
\overset{\eqref{eqn:def_bar_x}}=&\textstyle-\frac{1}{2\eta_k^2}\|x^k-\bar{x}^{k+1}\|^2+\frac{|J(k)|}{2\gamma\eta_k}\|x^k-\bar{x}^{k+1}\|^2+\frac{\gamma}{2\eta_k}\sum_{d\in J(k)}\|x^{d} - x^{d+1}\|^2,\nonumber
\end{align}
where the first inequality follows from the Young's inequality. Plugging~\eqref{term2} and~\eqref{term1} into~\eqref{eqn:equality_inconsistent} gives the desired result.\hfill\end{proof}

\cut{({\color{red}$\|x^k-\bar x^{k+1}\|^2$ can be replaced by the maximal component: $\max_i(x_i^k-\bar x_i^{k+1})^2$})}
We  need the following lemma on \emph{nonnegative almost supermartingales} \cite{robbins1985convergence}.
\begin{lemma}[{\cite[Theorem 1]{robbins1985convergence}}]\label{thm:sup}
Let $\mathscr{F} = (\cF^k)_{k\geq0}$ be a sequence of sub-sigma algebras of $\cF$ such that $\forall k \ge 0$, $\cF^k \subset \cF^{k+1}$. Define $\ell_+ (\mathscr{F})$ as the set of sequences of $[0, +\infty)$-valued random variables $(\xi_k)_{k\geq 0}$, where $\xi_k$ is $\cF^k$ measurable, and  $\ell_+^1 (\mathscr{F}) := \left\{(\xi_k)_{k\geq 0} \in \ell_+ (\mathscr{F})| \sum_{k} \xi_k < + \infty~ \text{a.s.} \right\}$. Let $(\alpha_k)_{k\geq 0}, (v_k)_{k\geq 0} \in \ell_{+}(\mathscr{F})$, and $(\eta_k)_{k\geq 0}, (\xi_k)_{k\geq 0} \in \ell_{+}^1(\mathscr{F})$ be such that
\begin{equation*}
\mathbb{E} (\alpha_{k+1} | \cF^k) + v_k \leq (1 + \xi_k) \alpha_k + \eta_k.
\end{equation*}
Then $(v_k)_{k\ge0} \in \ell_{+}^1 (\mathscr{F})$ and $\alpha_k$ converges to a $[0, +\infty)$-valued random variable a.s..
\end{lemma}

Let $\cH^{\tau+1}=\prod_{i=0}^{\tau}\cH$ be a product space and $\langle\cdot\, |\,\cdot \rangle$ be the induced  inner product:
$$\left\langle (z^0,\ldots,z^{\tau})\,|\,(y^0,\ldots,y^{\tau})\right\rangle=\sum_{i=0}^{\tau}\langle z^i,y^i\rangle,\quad\forall (z^0,\ldots,z^{\tau}), (y^0,\ldots,y^{\tau})\in\cH^{\tau+1}.$$
Let $M'$ be a symmetric $(\tau+1)\times(\tau+1)$ tri-diagonal matrix with its main diagonal as $\sqrt{p_{\min}}[\frac{1}{\sqrt{p_{\min}}}+\tau,2\tau-1,2\tau-3,\ldots,1]$ and first off-diagonal as $-\sqrt{p_{\min}}[\tau,\tau-1,\ldots,1]$,
and let $M=M'\otimes I_\cH$. Here \cut{$I_\cH$ is the identity operator on $\cH$ and }$\otimes$ represents the Kronecker product. For a given $(y^0,\cdots,y^\tau)\in\cH^{\tau+1}$, $(z^0,\cdots,z^\tau)=M(y^0,\cdots,y^\tau)$ is given by:
\begin{align*}
&z^0=\textstyle y^0+\sqrt{p_{\min}}(y^0-y^1),\\
&z^i = \textstyle\sqrt{p_{\min}}\left((i-\tau-1)y^{i-1}+(2\tau-2i+1)y^i+(i-\tau)y^{i+1}\right),\text{ if } 1\le i\le \tau-1,\\
&z^{\tau}=\textstyle\sqrt{p_{\min}}(y^{\tau}-y^{\tau-1}).
\end{align*}
Then $M$ is a self-adjoint and positive definite linear operator since $M'$ is symmetric and positive definite, and we define $\dotp{\cdot\, |\, \cdot}_M=\dotp{\cdot\, |\, M\cdot}$ as the $M$-weighted inner product and $\|\cdot\|_M$ the induced norm. Let
\begin{equation*}
\vx^k=(x^k,x^{k-1},\ldots,x^{k-\tau})\in \cH^{\tau+1},~k\ge 0,\,\mbox{and}~ \vx^* =(x^*,x^*,\ldots,x^*)\in\vX^*\subseteq\cH^{\tau+1},
\end{equation*}
where we set $x^{k}=x^{0}$ for $k<0$. With
\begin{equation}\label{eqn:xi}
\textstyle \xi_k(\vx^*) := \|\vx^k-\vx^*\|_M^2=\|x^{k} - x^*\|^2 + \sqrt{p_{\min}}\sum_{i=k-\tau}^{k-1} \left(i-(k-\tau)+1\right) \| x^{i} - x^{i+1}\|^2,
\end{equation}
we have the following fundamental inequality:

\begin{theorem}[Fundamental inequality]\label{thm:fund_inquality}
Let $(x^k)_{k\geq 0}$ be the sequence generated by ARock. Then for any $\vx^*\in\vX^*$, it holds that 
\begin{equation}\label{eqn:fund_inquality}
\begin{aligned}
\textstyle\mathbb{E}\left(\xi_{k+1}(\vx^*) \,\big|\, \cX^k\right) + \frac{1}{m} \left(\frac{1}{ \eta_k} - \frac{2\tau}{m\sqrt{p_{\min}}} - \frac{1}{mp_{\min}} \right) \|\bar{x}^{k+1} - x^k \|^2
\leq  \xi_k(\vx^*) .
\end{aligned}
\end{equation}
\end{theorem}
\begin{proof}Let $\gamma=m\sqrt{p_{\min}}$.  Since $J(k)\subset\{k-1,\cdots,k-\tau\}$, then \eqref{eqn:fund_inquality0} indicates
\begin{equation}\label{fund-ineq1}
\begin{aligned}
\textstyle\mathbb{E}\big(\|x^{k+1} - x^* \|^2 \,\big|\, \cX^k\big)  \leq & \textstyle\|x^{k} - x^* \|^2  +\frac{1}{\sqrt{p_{\min}}}\sum_{i=k-\tau}^{k-1}\|x^i-x^{i+1}\|^2\\
&\textstyle+ \frac{1}{m}\left(\frac{{\tau}}{m\sqrt{p_{\min}}}+\frac{1}{mp_{\min}}-\frac{1}{\eta_k}\right)\|x^k-\bar x^{k+1}\|^2.
\end{aligned}
\end{equation}
From \eqref{eqn:asyn_update} and \eqref{eqn:def_bar_x}, it is easy to have $\EE(\|x^k-x^{k+1}\|^2|\cX^k)\le\frac{1}{m^2p_{\min}}\|x^k-\bar{x}^{k+1}\|^2$, which together with \eqref{fund-ineq1} implies \eqref{eqn:fund_inquality} by using the definition of $\xi_k(\vx^*)$.
\cut{
\begin{align*}
~&\mathbb{E} (\xi_{k+1}(\vx^*) | \cX^k)  \\
\overset{\eqref{eqn:xi}}= & \textstyle \mathbb{E} (\|x^{k+1} - x^*\|^2| \cX^k) + \gamma\sum_{i=k+1-\tau}^{k} \frac{i-(k-\tau)}{m} \mathbb{E} (\| x^{ i} - x^{i+1}\|^2 | \cX^k) \\
  \overset{\eqref{eqn:asyn_update}}= & \textstyle\mathbb{E} (\|x^{k+1} - x^*\|^2| \cX^k) + \frac{\gamma\tau}{m} \mathbb{E}(\frac{\eta_k^2}{m^2p_{i_k}^2}\| S_{i_k}\hat x^k\|^2|\cX^k) +  \gamma\sum_{i=k+1-\tau}^{k-1} \frac{i-(k-\tau)}{m} \| x^{i} - x^{i+1}\|^2\\
          \le & \textstyle\mathbb{E} (\|x^{k+1} - x^*\|^2| \cX^k) + \frac{\gamma\tau}{m^3p_{\min}} \| x^{k } - \bar{x}^{k+1}\|^2 +  \gamma\sum_{i=k+1-\tau}^{k-1} \frac{i-(k-\tau)}{m} \| x^{i} - x^{i+1}\|^2\\
  \overset{\eqref{eqn:fund_inquality0}}\leq &\textstyle\|x^k - x^*\|^2 + \frac{1}{m} \left(\frac{|J(k)|}{\gamma} + \frac{\gamma\tau}{m^2p_{\min}} + \frac{1}{mp_{\min}} - \frac{1}{\eta_k}\right) \|x^k - \bar{x}^{k+1}\|^2   \\
&\textstyle+\frac{\gamma}{m}\sum_{d\in J(k)}\|{x}^{d} - {x}^{d+1}\|^2 +  \gamma\sum_{i=k+1-\tau}^{k-1} \frac{i-(k-\tau)}{m} \| x^{i} - x^{i+1}\|^2\\
 \leq &\textstyle\|x^k - x^*\|^2 + \frac{1}{m} \left(\frac{\tau}{\gamma} + \frac{\gamma\tau}{m^2p_{\min}} + \frac{1}{ mp_{\min}} - \frac{1}{ \eta_k}\right) \|x^k - \bar{x}^{k+1}\|^2   \\
&\textstyle+\frac{\gamma}{m}\sum_{i=k-\tau}^{k - 1}\|{x}^{i} - {x}^{i+1}\|^2 +  \gamma\sum_{i=k+1-\tau}^{k-1} \frac{i-(k-\tau)}{m} \| x^{i} - x^{i+1}\|^2\\
 \overset{\eqref{eqn:xi}}= &\textstyle\xi_k(\vx^*)  + \frac{1}{m} \left(\frac{2\tau}{m\sqrt{p_{\min}}} + \frac{1}{mp_{\min}} - \frac{1}{ \eta_k}\right) \|x^k - \bar{x}^{k+1}\|^2.
\end{align*}
The first inequality follows from the computation of the expectation and~\eqref{term2}, the second inequality holds because $J(k)\subset\{k-1,k-2,\cdots,k-\tau\}$, and the last equality uses $\gamma=m\sqrt{p_{\min}}$, which minimizes $\frac{\tau}{\gamma} + \frac{\gamma\tau}{m^2p_{\min}}$ over $\gamma>0$. Hence, \eqref{eqn:fund_inquality}\cut{the desired inequality} holds.}
\hfill\end{proof}
\begin{remark}[Stochastic Fej\'er monotonicity]\label{rmk:fejer}
From \eqref{eqn:fund_inquality}, if $0<\eta_k \leq \frac{mp_{\min}}{2\tau \sqrt{p_{\min}} + 1}$, then we have $\mathbb{E} (\|\vx^{k+1} - {\vx}^* \|^2_M | \cX^k)  \leq  \|\vx^{k} - {\vx}^* \|^2_M,\,\forall \vx^*\in \vX^*$.\cut{, namely, $(\vx^k)_{k\ge0}$ is stochastic Fej\'er monotone.}
\end{remark}
\begin{remark}\label{rm:stepsize}  Let us check our step size bound  $\frac{mp_{\min}}{2\tau\sqrt{p_{\min}}+1}$.  Consider the uniform case: $p_{\min}\equiv p_i\equiv\frac{1}{m}$. Then, the bound simplifies to $\frac{1}{1+2\tau/\sqrt{m}}$. If the max delay is no more than the square root of coordinates, i.e., $\tau = O(\sqrt{m})$, then the bound is  $O(1)$.  In general,  $\tau$ depends on several factors such as problem structure, system architecture, load balance, etc. If all updates and  agents are identical, then $\tau$ is  proportional to $p$, the number of agents. Hence, ARock takes an $O(1)$ step size for solving a problem with $m$ coordinates by $p=\sqrt{m}$ agents under balanced loads.
\end{remark}

The next lemma is a direct consequence of the invertibility of the metric  $M$. 
\begin{lemma}\label{lem:equiv-weak}
A sequence $(\vz^k)_{k\ge0}\subset\cH^{\tau+1}$ (weakly) converges to $\vz\in\cH^{\tau+1}$ under the metric $\dotp{\cdot\,|\,\cdot}$ if and only if it does so under the metric $\dotp{\cdot\,|\,\cdot}_M$.
\end{lemma}

In light of Lemma~\ref{lem:equiv-weak}, the metric of the inner product for weak convergence in the next lemma is not specified. The lemma and its proof are adapted from~\cite{combettes2014stochastic}. 

\begin{lemma}\label{lemma:convergence}
Let $(x^k)_{k\geq0}\subset \cH$ be the sequence generated by ARock with $\eta_k \in [\eta_{\min}, \frac{cmp_{\min}}{2\tau \sqrt{p_{\min}}+1}]$ for any $\eta_{\min}>0$ and $0<c<1$. \cut{$\vx^k$ and  ${\vx}^*$ are defined in~\eqref{eqn:def_x_k},} Then we have:
\begin{enumerate}
\item[\emph{(i)}] $\sum_{k=0}^{\infty}\|x^k - \bar{x}^{k+1}\|^2 < \infty$ a.s..
\item[\emph{(ii)}] $x^k-x^{k+1}\rightarrow 0$ a.s. and $\hat x^k-x^{k+1}\rightarrow 0$ a.s..
\item[\emph{(iii)}] The sequence $(\vx^k)_{k\geq 0}\subset \cH^{\tau+1}$ is bounded a.s..
\item[\emph{(iv)}] There exists $\tilde{\Omega} \in \cF$ such that $P(\tilde{\Omega}) = 1$ and, for every $\omega \in \tilde{\Omega}$ and every ${\vx}^* \in \vX^*$, $(\|\vx^k(\omega) - {\vx}^*\|_M)_{k\geq 0}$ converges.
\item[\emph{(v)}] Let $\mathscr{Z}(\vx^k)$ be the set of weakly convergent cluster points of $(\vx^k)_{k\geq 0}$. Then, $\mathscr{Z}(\vx^k) \subseteq \vX^*$ a.s..
\end{enumerate}
\end{lemma}
\begin{proof}
(i): Note that $\inf_k\left(\frac{1}{ \eta_k} - \frac{2\tau}{m\sqrt{p_{\min}}} - \frac{1}{mp_{\min}}\right)>0$. Also note that, in (\ref{eqn:fund_inquality}), $\|\bar{x}^{k+1} - x^k \|^2= \|\eta_k S\hat{x}^k\|^2$ is $\cX^k$-measurable. Hence, applying Lemma \ref{thm:sup} with $\xi_k=\eta_k=0$ and $\alpha_k=\xi_k(\vx^*),\,\forall k,$ to (\ref{eqn:fund_inquality}) gives this result directly.

(ii) From~(i), we have $x^k-\bar x^{k+1}\rightarrow 0$ a.s.. Since $\|x^k-x^{k+1}\|\leq \frac{1}{mp_{\min}}\|x^k-\bar x^{k+1}\|$, we have $x^k-x^{k+1}\rightarrow 0$ a.s.. Then from~\eqref{eqn:inconsist}, we have $\hat x^k-x^k\rightarrow 0$ a.s..

(iii): From Lemma \ref{thm:sup}, we have that $(\|\vx^{k} - {\vx}^* \|_M^2)_{k\geq 0}$ converges a.s. and so does $(\|\vx^{k} - {\vx}^* \|_M)_{k\geq 0}$, i.e., $\lim_{k\rightarrow \infty} \|\vx^{k} - {\vx}^* \|_M = \gamma$ a.s., where $\gamma$ is a $[0, +\infty )$-valued random variable. Hence, $(\|\vx^{k} - {\vx}^* \|_M)_{k\geq 0}$ must be bounded a.s. and so is $(\vx^k)_{k\geq 0}$.

\cut{(iv): From (ii), we have that for any $\vx^*=(x^*,x^*,\cdots,x^*)\in \vX^*$, there exists $\Omega_{x^*} \in \cF$ such that $P(\Omega_{x^*}) = 1$ and for any $\omega \in \Omega_{x^*}$, $(\|\vx^{k}(\omega) - {\vx}^* \|_M)_{k\geq 0}$ converges. Since $\cH$ is separable, there is a countable set $\cC \subseteq \cH$ such that $\bar{\cC} = \Fix T$, where $\bar{\cC}$ denotes the closure of $\cC$. Let $\tilde{\Omega} = \cap_{u\in \cC} \Omega_{u}$ and $\Omega_u^c$ be the complement of $\Omega_u$. We have
$$\textstyle P(\tilde{\Omega}) = 1 - P(\tilde{\Omega}^c)        = 1 - P(\cup_{u\in \cC} \Omega_{u}^c)
                \geq 1 - \sum_{u\in \cC} P(\Omega_{u}^c)  = 1.$$
Because $\bar{\cC} =\Fix T$, for a given $x^* \in \Fix T$, there must exist a sequence $(u^l)_{l\ge0} \subset \cC$ such that $\lim_{l\to\infty}u^l = x^*$. In addition, there are random variables $\gamma_l,\,l\ge 0$ such that
\begin{equation}\label{limgamma}
\textstyle \lim_{k\rightarrow \infty} \|\vx^{k}(\omega) - {\vu}^l \|_M = \gamma_l (\omega),\,\forall \omega\in\Omega_{u^l},\,\forall l\ge0,
\end{equation}
where ${\vu}^l = (u^l, u^l, ..., u^l)\in \vX^*$. By the triangle inequality, we have,
$$- \|{\vu}^l - {\vx}^* \|_M  \leq   \|\vx^{k}(\omega) - {\vx}^*\|_M -\|\vx^{k}(\omega) - {\vu}^l \|_M \leq \|{\vu}^l - {\vx}^* \|_M, \forall l, k,\,\forall \omega.$$
Therefore, for any $\omega \in \tilde{\Omega}$ and any $l\ge0$, it holds that
\begin{align*}
- \|{\vu}^l - {\vx}^* \|_M  &\textstyle \leq   \liminf_{k \rightarrow \infty} \big(\|\vx^{k}(\omega) - {\vx}^*\|_M -\|\vx^{k}(\omega) - {\vu}^l \|_M\big) \\
                                         &\textstyle  \overset{\eqref{limgamma}}=   \liminf_{k \rightarrow \infty} \|\vx^{k}(\omega) - {\vx}^*\|_M - \gamma_l(\omega) \\
                                         &\textstyle  \leq   \limsup_{k \rightarrow \infty} \|\vx^{k}(\omega) - {\vx}^*\|_M - \gamma_l(\omega) \\
                                         &\textstyle  = \limsup_{k \rightarrow \infty} \big(\|\vx^{k}(\omega) - {\vx}^*\|_M -  \|\vx^{k}(\omega) - {\vu}^l \|_M\big)\\
                                         & \textstyle \leq \|{\vu}^l - {\vx}^* \|_M.
\end{align*}
Let $l \rightarrow \infty$ in the above inequalites. Noting $\|\vu^l-\vx^*\|_M\to 0$ and $\|\vx^k(\omega)-\vx^*\|_M$ is independent of $l$, we have that $\gamma_\ell(\omega)$ converges to some $\gamma(\omega)$ and thus
$$\textstyle \lim_{k \rightarrow \infty} \|\vx^{k}(\omega) - {\vx}^*\|_M=\gamma(\omega),\,\forall \omega\in\tilde{\Omega}.$$
Since $\vx^*$ is arbitrary and $P(\tilde{\Omega})=1$, this completes the proof of (iv).
}
(iv): The proof follows directly from \cite[Proposition 2.3 (iii)]{combettes2014stochastic}. It is worth noting that $\tilde \Omega$ in the statement works for all ${\vx}^* \in \vX^*$, namely, $\tilde \Omega$ does not depend on ${\vx}^*$.

(v): By~(ii),  there exists $\hat{\Omega} \in \cF$ such that $P(\hat{\Omega} )=1$ and
\begin{align}\label{lim-xk}x^k(w)-x^{k+1}(w)\rightarrow 0,\quad\forall w\in\hat\Omega.\end{align}
For any $\omega\in\hat{\Omega}$, let $(\vx^{k_n}(\omega))_{n\ge0}$ be a weakly convergent subsequence of $(\vx^k(\omega))_{k\geq 0}$, i.e., $\vx^{k_n}(\omega) \rightharpoonup \vx$, where $\vx^{k_n}(\omega)  = (x^{k_n}(\omega), x^{k_n - 1}(\omega)..., x^{k_n - \tau}(\omega) )$ and $\vx = (u^0, ..., u^{\tau})$. Note that $\vx^{k_n}(\omega) \rightharpoonup \vx$ implies
$x^{k_n - j}(\omega) \rightharpoonup u^j,\, \forall j.$ 
Therefore, $u^i=u^j$, for any $i,j\in\{0,\cdots,\tau\}$ because $x^{k_n - i}(\omega) - x^{k_n - j}(\omega)\rightarrow 0$.


Furthermore, observing $\eta_k\ge\eta_{\min}>0$, we have
\begin{equation}\label{lim-xhat}
\lim_{n\rightarrow \infty} \hat x^{k_n} (\omega) - T \hat x^{k_n}(\omega)=\lim_{n\to\infty}S\hat x^{k_n}(\omega)
=\lim_{n\to\infty}\frac{1}{\eta_{k_n}}(x^{k_n}(\omega)-\bar{x}^{k_n+1}(\omega))=0.
\end{equation}
From the triangle inequality and the nonexpansiveness of $T$, it follows that
\begin{align*}
&\left\|x^{k_n}(\omega) - Tx^{k_n}(\omega)\right\| \\
=& \left\|x^{k_n}(\omega) - \hat x^{k_n} (\omega) + \hat x^{k_n}(\omega) - T\hat x^{k_n}(\omega) + T\hat x^{k_n}(\omega) - Tx^{k_n}(\omega)\right\| \\
 \leq& \left\|x^{k_n}(\omega) - \hat x^{k_n}(\omega)\right\| + \left\| \hat x^{k_n}(\omega) - T\hat x^{k_n}(\omega)\right\| + \left\| T\hat x^{k_n}(\omega) - Tx^{k_n}(\omega)\right\| \\
 \leq& 2 \left\|x^{k_n}(\omega) - \hat x^{k_n} (\omega)\right\| + \| \hat x^{k_n}(\omega) - T\hat x^{k_n}(\omega)\|\\
 \leq&\textstyle 2\sum_{d\in J(k_n)}\left\|x^{d}(\omega) -  x^{d+1} (\omega)\right\|+ \| \hat x^{k_n}(\omega) - T\hat x^{k_n}(\omega)\|.
\end{align*}
From~\eqref{lim-xk}, \eqref{lim-xhat}, and the above inequality, it follows
$\lim_{n\rightarrow \infty} x^{k_n}(\omega) - T x^{k_n}(\omega) = 0.$
Finally, the demiclosedness principle~\cite[Theorem 4.17]{bauschke2011convex} implies $u^0 \in \Fix T$. 
\hfill\end{proof}

\begin{theorem}\label{thm:convergence}
Under the assumptions of Lemma~\ref{lemma:convergence}, the sequence $(\vx^k)_{k\geq 0}$ weakly converges to an $\vX^*$-valued random variable a.s.. In addition, if $T$ is demicompact at 0, $(\vx^k)_{k\geq 0}$ strongly converges to an $\vX^*$-valued random variable a.s.. 
\end{theorem}

\begin{proof}
\cut{From Lemma~\ref{lemma:convergence}(i), there exists $\hat{\Omega}\in\cF$ such that $P(\hat{\Omega})=1$ and for every $w\in\hat\Omega$, we have $\sum_{k=0}^\infty \|x^k(w)-\bar x^{k+1}(w)\|^2<\infty$. Lemma~\ref{lemma:convergence}(iii) shows that there exists $\tilde\Omega\in\cF$ such that $P(\tilde\Omega)=1$ and for every $w\in\tilde\Omega$ and $\vx^*\in\vX^*$, $(\|\vx^k(w)-\vx^*\|_M)_{k\geq0}$ converges. Then $P(\tilde{\Omega}\cap\hat{\Omega} ) = 1$. We go to show that for any $\omega \in \tilde{\Omega} \cap \hat{\Omega}$, $\vx^k(\omega)$ weakly converges to a point in $\vX^*$.

For any $\omega \in \tilde{\Omega} \cap \hat{\Omega}$, it follows from Lemma~\ref{lemma:convergence}(iii) that $(\vx^k(\omega))_{k\ge0}$ is a bounded sequence and thus must have a weak cluster point~\cite[Lemma~2.37]{bauschke2011convex}. Assume $\vu(\omega)$ and $\vv(\omega)$ to be two weak cluster points of $(\vx^k(\omega))_{k\ge0}$. From Lemma~\ref{lemma:convergence}(iv), we have $\vu(\omega),\vv(\omega)\in \vX^*$. 
Note
\begin{equation*}
\begin{aligned}
&\big\langle \vx^k(\omega) ~|~ \vu(\omega) - \vv(\omega) \big\rangle_{M} \\
=&\textstyle \frac{1}{2} (\|\vx^k(\omega) - \vv(\omega)\|_M^2 - \|\vx^k(\omega)- \vu(\omega)\|_M^2 + \|\vu(\omega)\|_M^2 - \|\vv(\omega)\|_M^2),
\end{aligned}
\end{equation*}
and from Lemma~\ref{lemma:convergence}(iii), the right hand side of the above equality converges. Thus $\langle \vx^k(\omega) ~|~ \vu(\omega) - \vv(\omega) \rangle_{M}$ converges. Since $\vu(\omega)$ and $\vv(\omega)$ are cluster points of $(\vx^k(\omega))_{k\ge0}$, we have
$$\langle \vu(\omega) ~|~ \vu(\omega) - \vv(\omega) \rangle_{M} = \langle \vv(\omega) ~|~ \vu(\omega) - \vv(\omega) \rangle_{M},$$
which indicates
$\vu(\omega) = \vv(\omega)$. Therefore, $(\vx^k(\omega))_{k\ge0}$ has a unique weak cluster point. This completes the proof of almost sure weak convergence.}
{The proof for a.s. weak convergence follows from  Opial's Lemma~\cite{peypouquet_evolution_2009,opial_weak_1967} and Lemma~\ref{lemma:convergence} (iv)-(v).}
Next we assume that $T$ is demicompact at 0. From the proof of Lemma~\ref{lemma:convergence} (v), there is $\hat\Omega\in \cF$ such that $P(\hat\Omega)=1$ and, for any $w\in\hat\Omega$ and any weakly convergent subsequence of $(\vx^{k_n}(w))_{n\ge0}$, $\lim_{n\rightarrow \infty}x^{k_n}(w)-Tx^{k_n}(w)=0$. Since $T$ is demicompact, $(x^{k_n}(w))_{n\ge 0}$ has a strongly convergent subsequence, for which we still use $(x^{k_n}(w))_{n\ge 0}$. Hence, $x^{k_n}(w)\rightarrow \bar x(w)\in \Fix T$.   Lemma~\ref{lemma:convergence} (ii) yields $\vx^{k_n}(w)\rightarrow \bar\vx(w)\in \vX^*$. Then by Lemma~\ref{lemma:convergence} (iv), there is $\tilde\Omega\in\cF$ such that $P(\tilde\Omega)=1$ and, for every $w\in\tilde\Omega$ and every $\vx^*\in\vX^*$, $(\|\vx^k(w)-\vx^*\|_M)_{k\geq 0}$ converges. Thus, for any $w\in\hat\Omega\cap\tilde\Omega$, we have $\lim_{k\to\infty}\|\vx^k(w)-\bar\vx(w)\|_M=0$. Because $P(\hat\Omega\cap\tilde\Omega)=1$, we conclude that $(\vx^k)_{k\geq0}$ strongly converges to an $\vX^*$-valued random variable a.s..
\hfill\end{proof}

\cut{Next we show the almost sure strong convergence for uniformly quasi-monotone operators $S$. If $S$ is uniformly quasi-monotone, $\Fix T$ is a singleton. Indeed, let $x,y\in \Fix T$. Since $0\in Sx$ and $0\in Sy$, we have $0=\langle 0-0,x-y\rangle\geq \phi(\|x-y\|)$ and thus $\|x-y\|=0$ or $x=y$.

For any $w\in \hat\Omega\cup \tilde\Omega$, we have $x^k(w)-\bar x^{k+1}(w)\rightarrow 0$ because of Lemma~\ref{lemma:convergence}(i). Therefore we have $\hat x^k(w)-x^k(w)\rightarrow 0$ because of~\eqref{eqn:inconsist}. From the almost sure weak convergence, we have $x^k(w)\rightharpoonup  x^*\in \Fix T$. Thus $\hat x^k(w)\rightharpoonup x^*$ and
\begin{align*}
\textstyle\phi(\|\hat x^k(w)-x^*\|)\leq& \langle S\hat x^k(w)-Sx^*,\hat x^k(w)-x^*\rangle\\
=&\frac{1}{\eta_k}\langle x^k(w)-\bar x^{k+1}(w), \hat x^k(w)-x^*\rangle \rightarrow 0,
\end{align*}
or $\hat x^k(w)\rightarrow x^*$. Combining with $\hat x^k(w)-x^k(w)\rightarrow 0$, we have $x^k(w)\rightarrow x^*$.
\hfill\end{proof}}

\begin{remark} For the generalization in Section~\ref{sec:general}, we need to replace \eqref{term2} by
\begin{align*}
\textstyle\sum_{i=1}^m\frac{1}{p_i} \|U_i\circ S\hat{x}^{k}\|^2\le\frac{1}{p_{\min}} \sum_{i=1}^m\|U_i\circ S\hat{x}^{k}\|^2\leq \frac{C}{p_{\min}} \|S\hat{x}^{k}\|^2=\frac{C}{\eta_k^2p_{\min}}\|x^k-\bar{x}^{k+1}\|^2,
\end{align*}
and update the step size condition to $\eta_k \in [\eta_{\min}, \frac{cmp_{\min}}{2\tau \sqrt{p_{\min}}+C}]$. Then the proofs of Theorem~\ref{thm:convergence} and Lemma~\ref{lemma:convergence} will go through and yield the same convergence result. 
\end{remark}



\subsection{Linear convergence\cut{ for quasi-strongly monotone operator $S$}}\label{sec:linear}
In this section, we establish  linear convergence  under the assumption that $S$ is quasi-strongly monotone.
We first present a key lemma.

\begin{lemma}\label{lemma:linear_induction}
Assume that the step size is fixed, i.e., $\eta_k  = \eta$, and satisfies
\begin{equation}\label{eta-cond1}
\textstyle 0< \eta \leq \underline{\eta}_1:=(1 - \frac{1}{\rho}) \frac{m \sqrt{p_{\min}}}{8} \frac{\rho^{1/2} - 1}{ \rho^{(\tau + 1)/2} - 1}
\end{equation}
for some $\rho>1$. Then we have, for all $k\geq 1$,
\begin{equation} \label{eqn:linear_rate}
\mathbb{E} \|\bar{x}^{k} - x^{k-1}\|^2 \leq \rho \mathbb{E} \|\bar{x}^{k+1} - x^k\|^2.
\end{equation}
\end{lemma}

\cut{
\begin{proof}
We prove \eqref{eqn:linear_rate} by induction. First, for any $k \geq 1$, we observe that
\begin{align}
& \|\bar{x}^k - x^{k-1}\|^2 - \|\bar{x}^{k+1} - x^k\|^2 \cr
\leq \, &2 \| \bar{x}^k - x^{k-1} \| \| \bar{x}^{k+1} - x^k -  \bar{x}^k + x^{k-1}\| \quad \text{(by $\|a\|^2 - \|b\|^2 \leq 2 \|a\| \|b - a\|$)}\cr
=\, & 2 \| \bar{x}^k - x^{k-1} \| \| \eta S(\hat x^k) -  \eta S(\hat x^{k -1})\| \cr
\textstyle \leq \, & 4 \eta \| \bar{x}^k - x^{k-1} \| \| \hat x^{k} - \hat x^{k -1}\| \label{eq-mid}\\
 \leq \, & 4 \eta \| \bar{x}^k - x^{k-1} \| \big ( \| x^k  - \hat x^{k}\| + \|x^k - x^{k-1}\| + \| x^{k-1}- \hat x^{k -1}\| \big) \cr
= \, & \textstyle 4 \eta \| \bar{x}^k - x^{k-1} \| \left ( \| \sum_{d \in J(k)}(x^d - x^{d+1})\| + \|x^k - x^{k-1}\|\right.\cr
& \textstyle \hspace{3cm}+ \left.\|\sum_{d \in J(k-1)}(x^d - x^{d+1})\| \right) \cr
 \leq \, & \textstyle 4 \eta \| \bar{x}^k - x^{k-1} \| \left ( \sum_{d \in J(k)} \| x^d - x^{d+1}\| + \|x^k - x^{k-1}\| \right.\cr
 & \textstyle \hspace{3cm}+ \left.\sum_{d \in J(k-1)}\|x^d - x^{d+1}\| \right) \cr
\leq \, & \textstyle 4 \eta \| \bar{x}^k - x^{k-1} \|  \big( 2 \sum_{t=0}^{\tau } \| x^{k-t}  - x^{k - t -1}\| \big)\cr
= \, & \textstyle 8 \eta \sum_{t=0}^{\tau }  \| \bar{x}^k - x^{k-1} \| \| x^{k-t}  - x^{k - t -1}\| \label{eqn:linear_convergence_basic_step}
\end{align}

For the basic case,  we know that $\hat x^0 = x^0$, $\hat x^1 \in \{x^0, x^1\}$. Letting $k=1$ in \eqref{eq-mid} gives
\begin{equation}\label{eqn:linear_basic_case}
\begin{aligned}
       & \mathbb{E} \|\bar{x}^1 - x^{0}\|^2 -  \mathbb{E} \|\bar{x}^{2} - x^1\|^2 \\
\leq \, &  4\eta \mathbb{E}  \| \bar{x}^1 - x^{0} \|   \| x^1  - x^{0}\| \\
\leq \, &  \textstyle 2 \eta \big(\frac{1}{m \sqrt{p_{\min}}} \mathbb{E}  \| \bar{x}^1 - x^{0} \|^2  + m \sqrt{p_{\min}} \mathbb{E} \| x^1  - x^{0}\|^2 \big)\\
= \, &\textstyle  2 \eta \big(\frac{1}{m \sqrt{p_{\min}}} \mathbb{E}  \| \bar{x}^1 - x^{0} \|^2  + m \sqrt{p_{\min}} \sum_{i = 1}^m p_{i} \frac{\eta^2}{m^2 p_{i}^2} (S_i x^0)^2  \big)\\
\leq \, & \textstyle 2 \eta \big(\frac{1}{m \sqrt{p_{\min}}} \mathbb{E}  \| \bar{x}^1 - x^{0} \|^2  + \frac{1}{m \sqrt{p_{\min}}} \mathbb{E}  \| \bar{x}^1 - x^{0} \|^2  \big) \\
= \, & \textstyle \frac{4 \eta} {m \sqrt{p_{\min}}}  \mathbb{E}  \| \bar{x}^1 - x^{0} \|^2.
\end{aligned}
\end{equation}
Rearranging \eqref{eqn:linear_basic_case} gives $\mathbb{E} \|\bar{x}^1 - x^{0}\|^2 \leq \frac{1}{1 - \frac{4 \eta} {m \sqrt{p_{\min}}}} \mathbb{E} \| \bar{x}^2  - x^{1}\|^2.$ From \eqref{eta-cond1} and $\rho>1$, it holds that
$$0 < \eta \leq (1 - \frac{1}{\rho}) \frac{m \sqrt{p_{\min}}}{8} \frac{\rho^{1 / 2} - 1}{ \rho^{(\tau + 1) 2} - 1} \leq (1 - \frac{1}{\rho}) \frac{m\sqrt{p_{\min}}}{4}.$$ Hence, $\mathbb{E} \|\bar{x}^1 - x^{0}\|^2 \leq \rho \mathbb{E} \| \bar{x}^2  - x^{1}\|^2$.

For the induction step, we have
\begin{align*}
&\mathbb{E} \|\bar{x}^k - x^{k-1}\| \|x^{k-t} - x^{k-t-1}\| \\
\leq \, & \textstyle \frac{1}{2} \mathbb{E} \big\{ a \|x^{k-t} -x^{k-t-1}\|^2 + \frac{1}{a} \|\bar{x}^{k} - x^{k-1}\|^2 \big\} \\
\leq \, &  \textstyle \frac{1}{2} \mathbb{E} \big\{ \frac{a}{m^2 p_{\min}} \|\bar{x}^{k-t} -x^{k-t-1}\|^2 + \frac{1}{a} \|\bar{x}^{k} - x^{k-1}\|^2 \big\} \\
\leq \, &\textstyle \frac{1}{2}  \big\{ \frac{a \rho^t}{m^2 p_{\min}}  + \frac{1}{a} \big\} \mathbb{E} \|\bar{x}^{k} - x^{k-1}\|^2 \\
= \, & \textstyle \frac{ {\rho}^{t/2} }{m\sqrt{p_{\min}}}   \mathbb{E} \|\bar{x}^{k} - x^{k-1}\|^2. \qquad (\text{by letting } a = m \sqrt{p_{\min}}\rho^{ - { t/2}})\\
\end{align*}
Taking the expectation on both sides of the inequality \eqref{eqn:linear_convergence_basic_step} gives
\begin{align*}
& \textstyle \mathbb{E}\|\bar{x}^k - x^{k-1}\|^2 - \mathbb{E} \|\bar{x}^{k+1} - x^k\|^2 \\
\leq & \textstyle 8 \eta \sum_{t=0}^{\tau }  \mathbb{E} \| \bar{x}^k - x^{k-1} \| \| x^{k-t}  - x^{k - t -1}\| \\
\leq & \textstyle \frac{8 \eta}{m \sqrt{p_{\min}}} \sum_{t=0}^{\tau } \rho^{t / 2} \mathbb{E} \| \bar{x}^k - x^{k-1} \|^2 \\
\leq & \textstyle \frac{8 \eta}{m \sqrt{p_{\min}}} \frac{1 - \rho^{(\tau + 1 )/2}}{1 - \rho^{1 / 2}}\mathbb{E}  \| \bar{x}^k - x^{k-1} \|^2, 
\end{align*}
Rearranging the above inequality and using \eqref{eta-cond1} gives
 $$\mathbb{E} \|\bar{x}^{k} - x^{k-1}\|^2 \cut{\leq \frac{1}{1 - \eta \theta} \mathbb{E} \| \bar{x}^{k+1}  - x^{k}\|^2 }\leq \rho \mathbb{E} \| \bar{x}^{k+1}  - x^{k}\|^2. $$
 \hfill
\end{proof}
}

\begin{proof}
We prove \eqref{eqn:linear_rate} by induction. First, based on the inequality $\|a\|^2 - \|b\|^2 \leq 2 \|a\| \|b - a\|$ we observe that, for any $k \geq 1$,
\begin{align}\label{eq-mid}
\|\bar{x}^k - x^{k-1}\|^2 - \|\bar{x}^{k+1} - x^k\|^2 \leq \, &2 \| \bar{x}^k - x^{k-1} \| \| \bar{x}^{k+1} - x^k -  \bar{x}^k + x^{k-1}\| \cr
=\, & 2 \| \bar{x}^k - x^{k-1} \| \| \eta S(\hat x^k) -  \eta S(\hat x^{k -1})\| \cr
\textstyle \leq \, & 4 \eta \| \bar{x}^k - x^{k-1} \| \| \hat x^{k} - \hat x^{k -1}\|.
\end{align}
Applying the triangle inequality and \eqref{eqn:inconsist} yields
\begin{align}\label{eqn:linear_convergence_basic_step}
\| \hat x^{k} - \hat x^{k -1}\| \leq & \| x^k  - \hat x^{k}\| + \|x^k - x^{k-1}\| + \| x^{k-1}- \hat x^{k -1}\| \cr
 \leq \, &  \textstyle \underset{d \in J(k)}\sum \| x^d - x^{d+1}\| + \|x^k - x^{k-1}\| + \underset{d \in J(k-1)}\sum\|x^d - x^{d+1}\| \cr
\leq \, & \textstyle  2 \sum_{t=0}^{\tau } \| x^{k-t}  - x^{k - t -1}\|.
\end{align}
For the basic case,  we have $\hat x^0 = x^0$, $\hat x^1 \in \{x^0, x^1\}$. Letting $k=1$ in \eqref{eq-mid} gets us
\begin{align*}
        \mathbb{E} \|\bar{x}^1 - x^{0}\|^2 -  \mathbb{E} \|\bar{x}^{2} - x^1\|^2
\leq \, &  4\eta \mathbb{E}  \| \bar{x}^1 - x^{0} \|   \| x^1  - x^{0}\| \\
\leq \, &  \textstyle 2 \eta \big(\frac{1}{m \sqrt{p_{\min}}} \mathbb{E}  \| \bar{x}^1 - x^{0} \|^2  + m \sqrt{p_{\min}} \mathbb{E} \| x^1  - x^{0}\|^2 \big)\\
= \, &\textstyle  2 \eta \big(\frac{1}{m \sqrt{p_{\min}}} \mathbb{E}  \| \bar{x}^1 - x^{0} \|^2  + m \sqrt{p_{\min}} \sum_{i = 1}^m p_{i} \frac{\eta^2}{m^2 p_{i}^2} (S_i x^0)^2  \big)\\
\leq \, & \textstyle 2 \eta \big(\frac{1}{m \sqrt{p_{\min}}} \mathbb{E}  \| \bar{x}^1 - x^{0} \|^2  + \frac{1}{m \sqrt{p_{\min}}} \mathbb{E}  \| \bar{x}^1 - x^{0} \|^2  \big) \\
= \, & \textstyle \frac{4 \eta} {m \sqrt{p_{\min}}}  \mathbb{E}  \| \bar{x}^1 - x^{0} \|^2.
\end{align*}
Rearranging the above inequality yields $\mathbb{E} \|\bar{x}^1 - x^{0}\|^2 \leq \frac{1}{1 - \frac{4 \eta} {m \sqrt{p_{\min}}}} \mathbb{E} \| \bar{x}^2  - x^{1}\|^2.$ By \eqref{eta-cond1} and $\rho>1$, it holds that
$\textstyle 0 < \eta \leq (1 - \frac{1}{\rho}) \frac{m \sqrt{p_{\min}}}{8} \frac{\rho^{1/2} - 1}{ \rho^{(\tau + 1)/2} - 1} \leq (1 - \frac{1}{\rho}) \frac{m\sqrt{p_{\min}}}{4}.$ Hence, $\mathbb{E} \|\bar{x}^1 - x^{0}\|^2 \leq \rho \mathbb{E} \| \bar{x}^2  - x^{1}\|^2$.

For the induction step, applying Young's inequality gives us
\begin{align*}
\mathbb{E} \|\bar{x}^k - x^{k-1}\| \|x^{k-t} - x^{k-t-1}\|
\leq \, & \textstyle \frac{1}{2} \mathbb{E} \big\{ a \|x^{k-t} -x^{k-t-1}\|^2 + \frac{1}{a} \|\bar{x}^{k} - x^{k-1}\|^2 \big\} \\
\leq \, &  \textstyle \frac{1}{2} \mathbb{E} \big\{ \frac{a}{m^2 p_{\min}} \|\bar{x}^{k-t} -x^{k-t-1}\|^2 + \frac{1}{a} \|\bar{x}^{k} - x^{k-1}\|^2 \big\} \\
\leq \, &\textstyle \frac{1}{2}  \big\{ \frac{a \rho^t}{m^2 p_{\min}}  + \frac{1}{a} \big\} \mathbb{E} \|\bar{x}^{k} - x^{k-1}\|^2 \\
= \, & \textstyle \frac{ {\rho}^{t/2} }{m\sqrt{p_{\min}}}   \mathbb{E} \|\bar{x}^{k} - x^{k-1}\|^2. \qquad (\text{letting } a = m \sqrt{p_{\min}}\rho^{ - { t/ 2}})
\end{align*}
Taking the expectation on \eqref{eqn:linear_convergence_basic_step} and combining it with \eqref{eq-mid} yield
\begin{align*}
& \textstyle \mathbb{E}\|\bar{x}^k - x^{k-1}\|^2 - \mathbb{E} \|\bar{x}^{k+1} - x^k\|^2
\leq  \textstyle 8 \eta \sum_{t=0}^{\tau }  \mathbb{E} \| \bar{x}^k - x^{k-1} \| \| x^{k-t}  - x^{k - t -1}\| \\
\leq & \textstyle \frac{8 \eta}{m \sqrt{p_{\min}}} \sum_{t=0}^{\tau } \rho^{t /2} \mathbb{E} \| \bar{x}^k - x^{k-1} \|^2
\leq  \textstyle \frac{8 \eta}{m \sqrt{p_{\min}}} \frac{1 - \rho^{(\tau + 1 )/2}}{1 - \rho^{1 /2}}\mathbb{E}  \| \bar{x}^k - x^{k-1} \|^2. 
\end{align*}
Finally, rearranging the above inequality and using \eqref{eta-cond1} lead to
 $\mathbb{E} \|\bar{x}^{k} - x^{k-1}\|^2 \cut{\leq \frac{1}{1 - \eta \theta} \mathbb{E} \| \bar{x}^{k+1}  - x^{k}\|^2 }\leq \rho \mathbb{E} \| \bar{x}^{k+1}  - x^{k}\|^2. $
 This completes the proof.
 \end{proof}

With this lemma, we are ready to derive the linear convergence rate of ARock.

\begin{theorem}[Linear convergence\cut{ for quasi-$\mu$-strongly monotone operator}]\label{thm:strongly_monotone}
Assume that $S$ is quasi-$\mu$-strongly monotone with $\mu >0$. Let $\beta\in (0,1)$ and $(x^k)_{k\geq0}$ be the sequence generated by ARock with a constant stepsize $\eta \in (0, \min \{\underline{\eta}_1, \underline{\eta}_2\}]$, where $\underline{\eta}_1$ is given in~\eqref{eta-cond1} and
\begin{equation}\label{eta-cond2}\textstyle
\underline{\eta}_2=\frac{-b + \sqrt{b^2 + 4(1-\beta)a}}{2a},\quad a = \frac{2\beta \mu \tau }{ m^2 p_{\min}} \frac{\rho(\rho^\tau - 1)}{\rho - 1},\quad b = \frac{1}{ m p_{\min}} + \frac{2}{ m}\sqrt{ \frac{\rho(\rho^\tau - 1)\tau}{(\rho - 1) p_{\min}}}.
\end{equation}
Then
\begin{equation}\label{linear-convg}
\textstyle\mathbb{E}\left(\|x^{k} - x^* \|^2\right)\le \left(1-\frac{\beta\mu\eta}{m}\right)^k \|x^{0} - x^* \|^2.
\end{equation}

\end{theorem}
\begin{proof}
Following the proof of Lemma \ref{lemma:fund} and starting from \eqref{term1}, we have
\begin{equation*}\label{thm:linear_cvg}
\begin{aligned}
&\langle S \hat{x}^{k}, x^* - x^k \rangle\\
\le&\textstyle \langle S \hat{x}^{k}-Sx^*, x^* - \hat{x}^k\rangle+\frac{1}{2\eta}\sum_{d\in J(k)}\big(\frac{1}{\gamma}\|x^k-\bar{x}^{k+1}\|^2+ \gamma\|x^{d} - x^{d+1}\|^2\big)\cr
\le&\textstyle -{\beta\mu}\|\hat{x}^k-x^*\|^2-\frac{1-\beta}{2}\|S \hat{x}^{k}\|^2+\frac{1}{2\eta}\sum_{d\in J(k)}(\frac{1}{\gamma}\|x^k-\bar{x}^{k+1}\|^2+ \gamma\|x^{d} - x^{d+1}\|^2)\cr
=&\textstyle -{\beta\mu}\|{x}^k-x^*+\sum_{d\in J(k)} (x^{d} - x^{d+1})\|^2-\frac{1-\beta}{2\eta^2}\|x^k-\bar{x}^{k+1}\|^2 \\
 &\textstyle +\frac{|J(k)|}{2\gamma\eta}\|x^k-\bar{x}^{k+1}\|^2+\frac{\gamma}{2\eta}\sum_{d\in J(k)}\|x^{d} - x^{d+1}\|^2\cr
\le&\textstyle -\frac{\beta\mu}{2}\|{x}^k-x^*\|^2+{\beta\mu}\|\sum_{d\in J(k)} (x^{d} - x^{d+1})\|^2-\frac{1-\beta}{2\eta^2}\|x^k-\bar{x}^{k+1}\|^2 \\
& \textstyle+\frac{|J(k)|}{2\gamma\eta}\|x^k-\bar{x}^{k+1}\|^2+\frac{\gamma}{2\eta}\sum_{d\in J(k)}\|x^{d} - x^{d+1}\|^2\cr
\le & \textstyle -\frac{\beta\mu}{2}\|{x}^k-x^*\|^2+{\beta\mu}|J(k)|\sum_{d\in J(k)}\| x^{d} - x^{d+1}\|^2-\frac{1-\beta}{2\eta^2}\|x^k-\bar{x}^{k+1}\|^2 \\
&\textstyle +\frac{|J(k)|}{2\gamma\eta}\|x^k-\bar{x}^{k+1}\|^2+\frac{\gamma}{2\eta}\sum_{d\in J(k)}\|x^{d} - x^{d+1}\|^2,
\end{aligned}
\end{equation*}
where the second inequality holds because $S$ is $\frac{1}{2}$-cocoercive and also quasi-$\mu$-strongly monotone, and the last one comes from the Cauchy-Schwartz inequality.
Plugging the above inequality and \eqref{term2} into \eqref{eqn:equality_inconsistent} and noting $|J(k)|\subset\{k-\tau,\ldots,k-1\}$ gives
\begin{align*}
\begin{aligned}
\textstyle\mathbb{E}\big(\|x^{k+1} - x^* \|^2 \,\big|\, \cX^k\big)  \leq & \textstyle\big(1-\frac{\beta\mu\eta}{m}\big)\|x^{k} - x^* \|^2  +\frac{1}{ m}\left(2\beta\eta\mu\tau+\gamma\right)\sum_{d=k-\tau}^{k-1}\|x^d-x^{d+1}\|^2\\
&\textstyle+ \frac{1}{m}\left(\frac{{\tau}}{\gamma}+\frac{1}{ mp_{\min}}-\frac{1-\beta}{{\eta}}\right)\|x^k-\bar x^{k+1}\|^2.
\end{aligned}
\end{align*}
Taking expectation over both sides of the above inequality, noting $\EE\|x^d-x^{d+1}\|^2\le \frac{1}{m^2p_{\min}}\EE\|x^d-\bar{x}^{d+1}\|^2$, and using Lemma \ref{lemma:linear_induction}, we have
\begin{align*}
\begin{aligned}
~ &\mathbb{E}\big(\|x^{k+1} - x^* \|^2\big)  \\
\le& \textstyle \big(1-\frac{\beta\mu\eta}{m}\big)\EE\|x^{k} - x^* \|^2  +\frac{{1}}{ m^3 p_{\min}}\left(2\beta\eta\mu\tau+{\gamma}\right)\sum_{d= 1}^{\tau}\rho^d\EE\|x^k-\bar{x}^{k+1}\|^2\\
&\textstyle + \frac{1}{m}\left(\frac{{\tau}}{\gamma}+\frac{1}{mp_{\min}}-\frac{1-\beta}{{\eta}}\right)\EE\|x^k-\bar x^{k+1}\|^2\cr
=& \textstyle \big(1-\frac{\beta\mu\eta}{m}\big)\EE\|x^{k} - x^* \|^2  +\frac{{1}}{m^3 p_{\min}}\left(2\beta\eta\mu\tau+{\gamma}\right)\frac{\rho(\rho^\tau-1)}{\rho-1}\EE\|x^k-\bar{x}^{k+1}\|^2\\
&\textstyle + \frac{1}{m}\left(\frac{{\tau}}{\gamma}+\frac{1}{mp_{\min}}-\frac{1-\beta}{{\eta}}\right)\EE\|x^k-\bar x^{k+1}\|^2\\
= & \textstyle \big(1-\frac{\beta\mu\eta}{m}\big)\EE\|x^{k} - x^* \|^2 \cr
&\textstyle+\frac{1}{m}\left(\frac{2\beta\eta \mu\tau}{ m^2p_{\min}}\frac{\rho(\rho^\tau - 1)}{\rho - 1}+\frac{2}{m}\sqrt{\frac{\rho(\rho^\tau-1)\tau}{(\rho-1) p_{\min}}}+\frac{1}{ mp_{\min}}-\frac{1-\beta}{{\eta}}\right)\EE\|x^k-\bar x^{k+1}\|^2\cr
\le & \textstyle \big(1-\frac{\beta\mu\eta}{m}\big)\EE\|x^{k} - x^* \|^2,
\end{aligned}
\end{align*}
\cut{Letting $\gamma = m\sqrt{\frac{\tau (\rho - 1) p_{\min}}{ \rho(\rho^\tau - 1)}}$ in the above inequality, we have
\begin{align*}
&\mathbb{E}\big(\|x^{k+1} - x^* \|^2\big)  \cr
\le & \textstyle \big(1-\frac{\beta\mu\eta}{m}\big)\EE\|x^{k} - x^* \|^2 \cr
&\textstyle+\frac{1}{m}\left(\frac{2\beta\eta \mu\tau}{ m^2p_{\min}}\frac{\rho(\rho^\tau - 1)}{\rho - 1}+\frac{2}{m}\sqrt{\frac{\rho(\rho^\tau-1)\tau}{(\rho-1) p_{\min}}}+\frac{1}{ mp_{\min}}-\frac{1-\beta}{{\eta}}\right)\EE\|x^k-\bar x^{k+1}\|^2\cr
\le & \textstyle \big(1-\frac{\beta\mu\eta}{m}\big)\EE\|x^{k} - x^* \|^2,
\end{align*}}where we have let $\gamma = m\sqrt{\frac{\tau (\rho - 1) p_{\min}}{ \rho(\rho^\tau - 1)}}$ in the second equality, and the last inequality holds because of the choice of $\eta$. 
Therefore, \eqref{linear-convg} holds. 
\hfill\end{proof}

\begin{remark}\label{rmk:linear}
Assume $i_k$ is chosen uniformly at random, so $p_{\min}=\frac{1}{m}$. We consider the case when $m$ and $\tau$ are large. Let $\sqrt{\rho}=1+\frac{1}{\tau}$.  Then from the fact that $(1+\frac{1}{k})^k$ increasingly converges to the natural number $e$, we have from \eqref{eta-cond1} that $\underline{\eta}_1=O(\frac{\sqrt{m}}{\tau^2})$. In addition, note from \eqref{eta-cond2} that $a=O(b^2)=O(\frac{\tau^2}{m})$, and thus $\underline{\eta}_2=O(\frac{\sqrt{m}}{\tau})$. Therefore, if $\tau=O(m^{\frac{1}{4}})$, then the stepsize in Theorem \ref{thm:strongly_monotone} can be $\eta=O(1)$. Hence, linear speedup can be achieved.
\end{remark}

\end{section}

\begin{section}{Experiments}\label{sec:numerical}
We illustrate the behavior of ARock for solving the $\ell_1$ regularized logistic regression problem. Our primary goal is to show the efficiency of the async-parallel implementation compared to the single-threaded implementation and the sync-parallel implementation.

Our experiments run on 1 to 32 threads on a machine with eight Quad-Core AMD Opteron\textsuperscript{TM} Processors (32 cores in total) and $64$ Gigabytes of RAM. All of the experiments were coded in C++ and OpenMP. We use the Eigen library\footnote{\url{http://eigen.tuxfamily.org}} for sparse matrix operations. Our codes as well as numerical results for other applications will be publicly available on the authors' website.

The running times and speedup ratios of  both sync-parallel and async-parallel algorithms are sensitive to a number of factors, such as the size of each coordinate update (granularity), sparsity  of the problem data,  compiler optimization flags,  and  operations that affect cache performance and memory access contention. In addition, since all  agents in the sync-parallel implementation  must wait for the last agent to finish an iteration, a large load imbalance will significantly degrade the performance. We do not have the space in this paper to present  numerical results under all variations of these cases. 
\cut{
\begin{subsection}{Linear equations}\label{sec:lin-eq}
We apply ARock presented in Subsection \ref{sec:linearequations} to solve linear equations $Ax=b$ on two synthetic datasets. In Dataset I, $A \in \mathbb{R}^{n\times n}$ is a banded matrix with a small bandwidth. Given the bandwidth $w$, we have  $a_{ij}=0$ whenever $|j-i|>w$. In Dataset II, $A\in \mathbb{R}^{n\times n}$ is a dense random symmetric matrix.
Both matrices $A$ are   diagonally dominant. They are summarized in Table \ref{tab:lq_data}. For each dataset, the entries of $x$ independently follow the standard Gaussian distribution, and  $b$ is generated as $Ax$.
\begin{table}[htbp]
\centering
 \begin{tabular}{lccc}
  \toprule
  Name  & Type & Size ($n$) & Bandwidth ($w$)\\
  \midrule
 Dataset I & Sparse & 1, 000, 000 & 5\\
  Dataset II & Dense & 5, 000 & N/A \\
  \bottomrule
 \end{tabular}
 \caption{\label{tab:lq_data}Linear equation datasets}
\end{table}

Throughout the experiments, $A$, $b$ and $x$ are shared variables. We set $\eta_k  = 1, \,\forall k$. 
Although the theory (Theorem~\ref{thm:convergence}) requires a smaller step size if more cores are used (usually leading to a larger delay $\tau$), we observe numerically that the unit step size can serve well for the algorithm on different numbers of cores. 
The components of $x$ are  evenly partitioned into coordinates, and each coordinate contains $\frac{n}{\text{number of cores}}$ components of $x$, where the number of cores varies from  1 to 32. Each core is in charge of updating one pre-assigned coordinate\footnote{Deterministic assignments are applied since generating pseudo-random numbers is more expensive than updating a coordinate in this test. We tested random assignments and observed that they  give nearly the same epoch-versus-convergence performance but takes a much longer running time.}. We run both  sync-parallel Jacobi and ARock  algorithms to 100 epochs on Dataset I, and to 50 epochs on Dataset II, where an epoch is counted for every   $n$ updates to the components of $x$.

Figure~\ref{fig:lq_speedup} depicts how the size of residual, $\|Ax - b\|$, reduces over the wall-clock time. From the figure, we see that  both sync-parallel Jacobi and ARock show almost-linear speedup as the number of cores increases, but ARock takes much less time than sync-parallel Jacobi. To compare their  epoch convergence rates, we compute the size of residual  after each epoch (though they take different times to finish an epoch). Figure~\ref{fig:lq_residual} plots the sizes of residual at the end of different epochs.  Note that the  curves of sync-parallel Jacobi  running on different numbers of cores are nearly identical. Hence, only one curve is shown. The curves  of ARock running on different numbers of cores are also nearly identical, 
and they closely match with the reference curve of the standard Gauss-Seidel method, which is a sequential method. Therefore, ARock enjoys not only almost-linear speedup but also the Gauss-Seidel type of fast convergence, thanks to the asynchrony. This phenomenon was also  observed in the previous work~\cite{goos_numerical_1992,bertsekas1991some}.

\begin{figure}[!h]
        \centering
       \begin{subfigure}[b]{0.48\textwidth}
                \includegraphics[width=\textwidth]{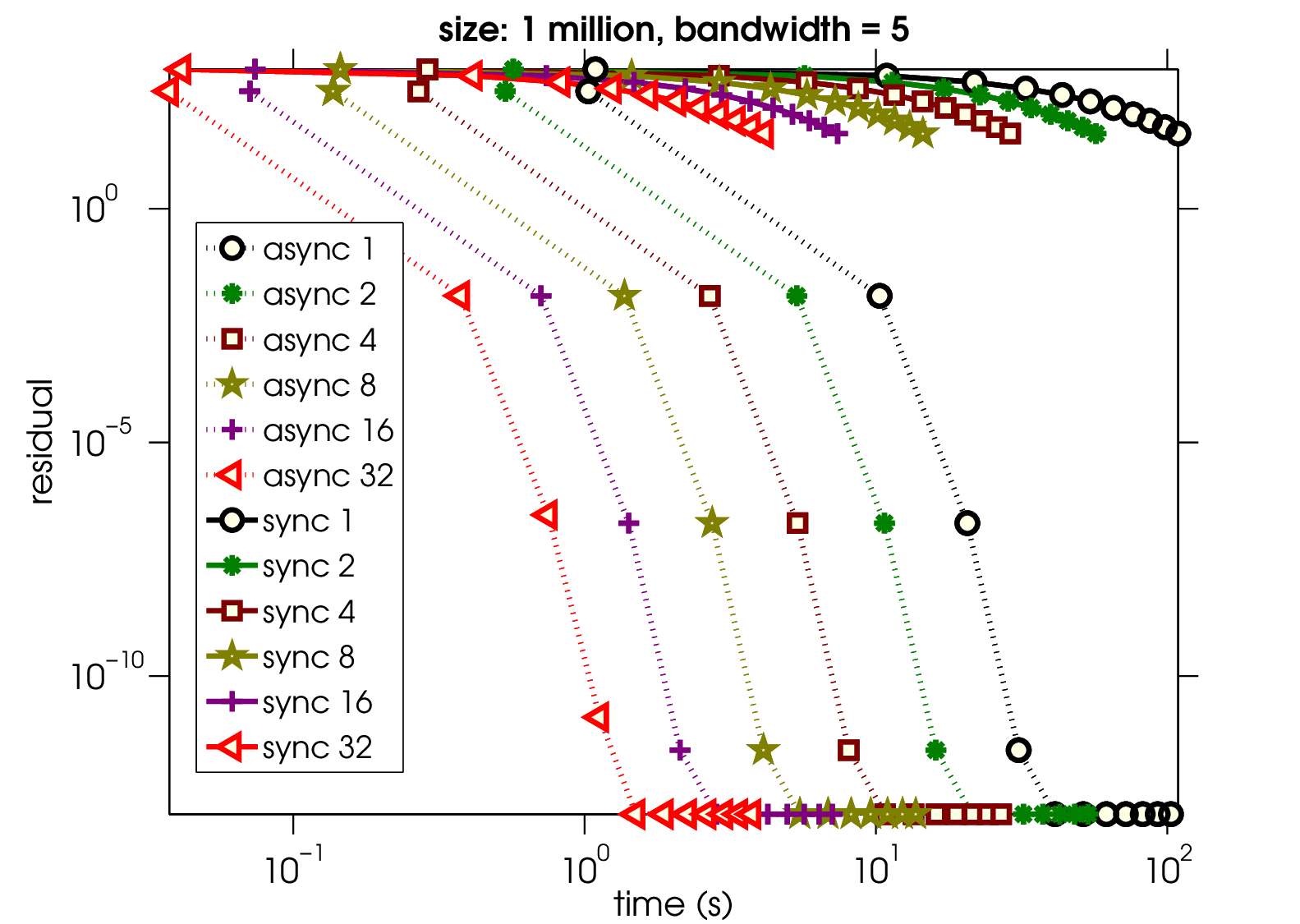}
                \caption{Dataset I}
        \end{subfigure}
        ~ 
        \begin{subfigure}[b]{0.48\textwidth}
                \includegraphics[width=\textwidth]{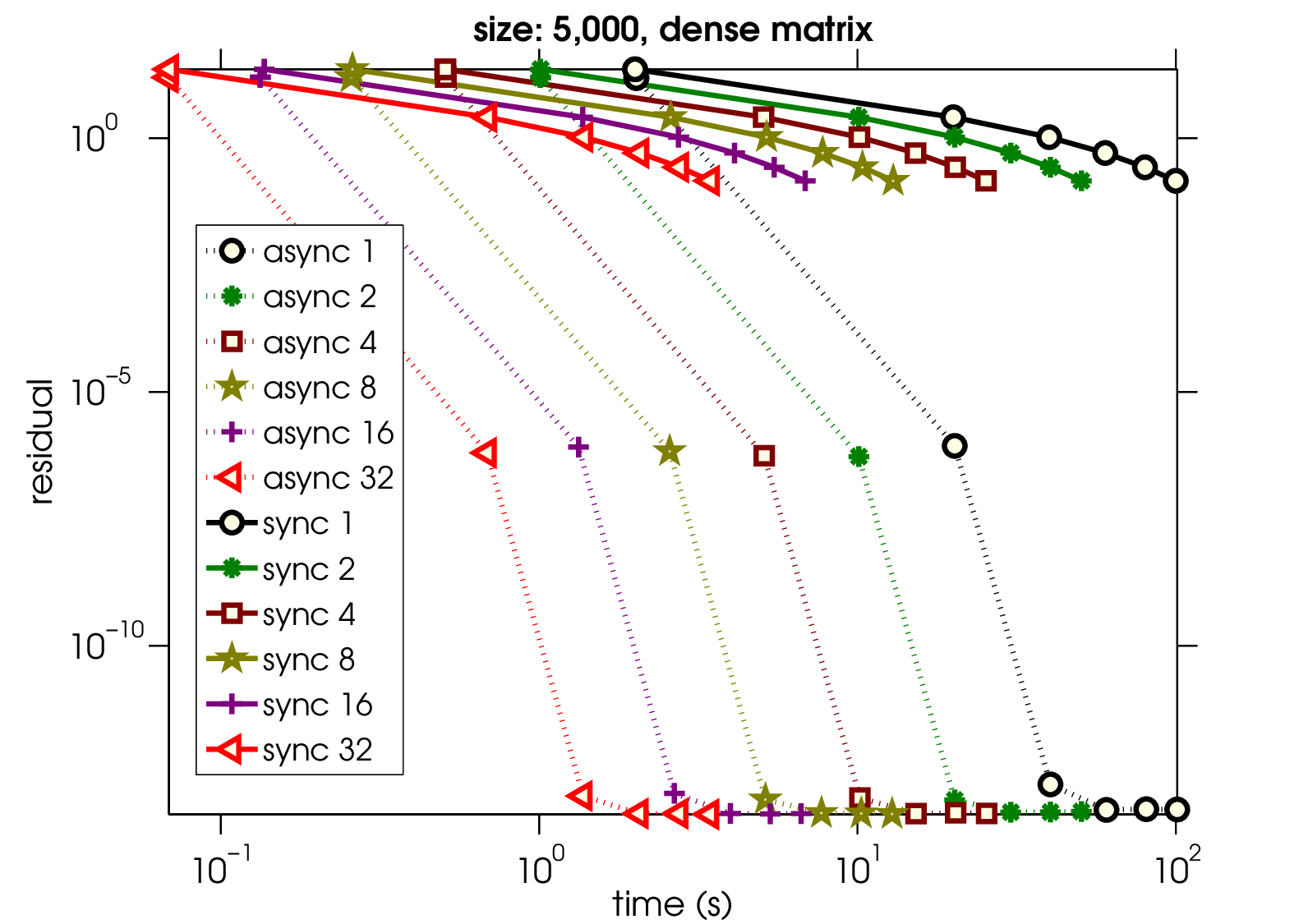}
                \caption{Dataset II}
        \end{subfigure}
        \caption{Residual $\|Ax - b\|$ versus  wall-clock time for sync-parallel Jacobi and ARock (async-parallel) running on 1--32 cores. All runs are terminated at 100 epochs on Dataset I and 50 epochs on Dataset II.}\label{fig:lq_speedup}
\end{figure}

\begin{figure}[!h]
        \centering
       \begin{subfigure}[b]{0.48\textwidth}
                \includegraphics[width=\textwidth]{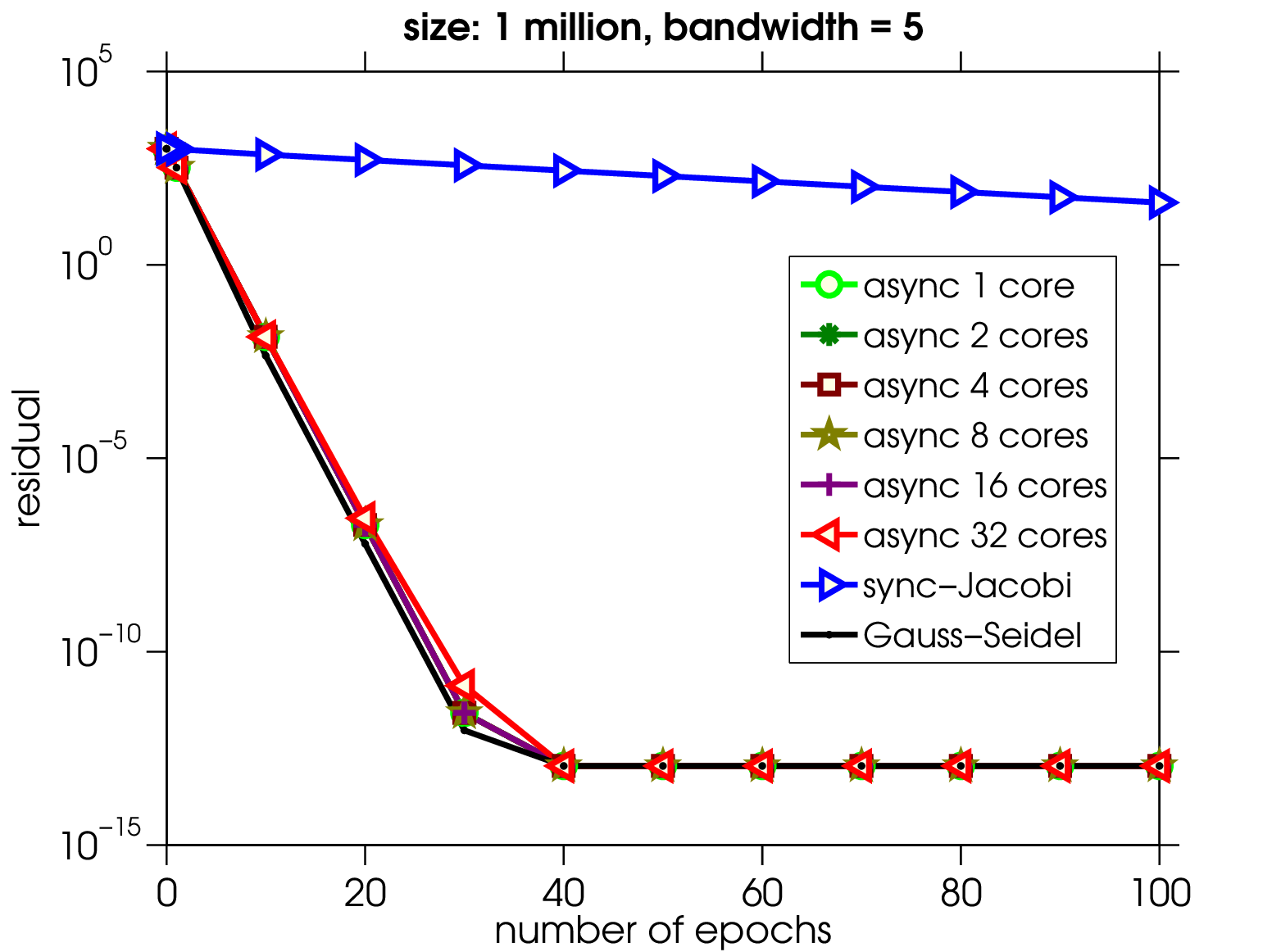}
                \caption{Dataset I}
        \end{subfigure}
        ~ 
        \begin{subfigure}[b]{0.48\textwidth}
                \includegraphics[width=\textwidth]{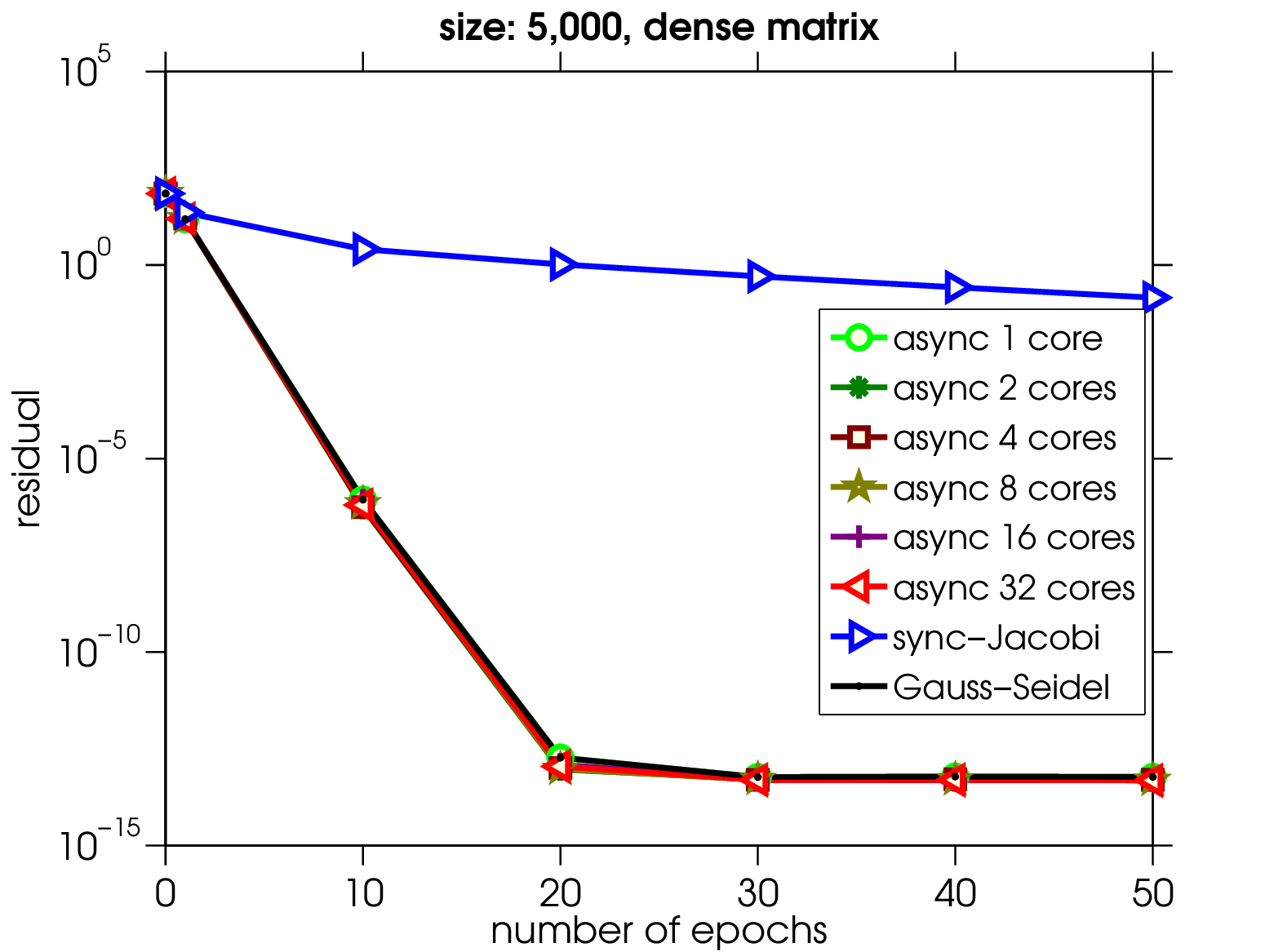}
                \caption{Dataset II}
        \end{subfigure}
        \caption{Epoch-convergence (regardless of time) of sync-parallel Jacobi and ARock (async-parallel). The black residual curve of the standard Gauss-Seidel method is plotted for reference.}\label{fig:lq_residual}
\end{figure}


\end{subsection}
}

\begin{subsection}{$\ell_1$ regularized logistic regression}
In this subsection, we apply ARock with the update \eqref{asyncFBS} to the $\ell_1$ regularized logistic regression problem:
\begin{equation}\label{eqn:log}
\Min_{x \in \mathbb{R}^n} \lambda \|x\|_1 + \frac{1}{N} \sum_{i=1}^N \log\big(1 + \exp(-b_i \cdot a_i^T x)\big),
\end{equation}
where $\{(a_i, b_i)\}_{i=1}^N$ is the set of sample-label pairs with $b_i \in \{1, -1\}$, $\lambda=0.0001$, and $n$ and $N$ represent the numbers of features and samples, respectively. This test uses the datasets\footnote{\url{http://www.csie.ntu.edu.tw/~cjlin/libsvmtools/datasets/}}: rcv1 and news20, which are summarized in Table \ref{tab:log_data}.

\begin{table}[htbp]
\centering
{\small \begin{tabular}{crrr}
\hline
  Name & \# samples & \# features & \# nonzeros in $\{a_1,\ldots,a_N\}$ \\
  \hline
 rcv1 & 20, 242 & 47, 236 & 1, 498, 952\\
  news20 & 19, 996 & 1, 355, 191 & 9, 097, 916\\
  \hline
 \end{tabular}}
 \caption{\label{tab:log_data}Two  datasets for  sparse logistic regression.}
 \vspace{-8mm}
\end{table}

\begin{figure}[!h]
        \centering
       \begin{subfigure}[b]{0.4\textwidth}
                \includegraphics[width=\textwidth]{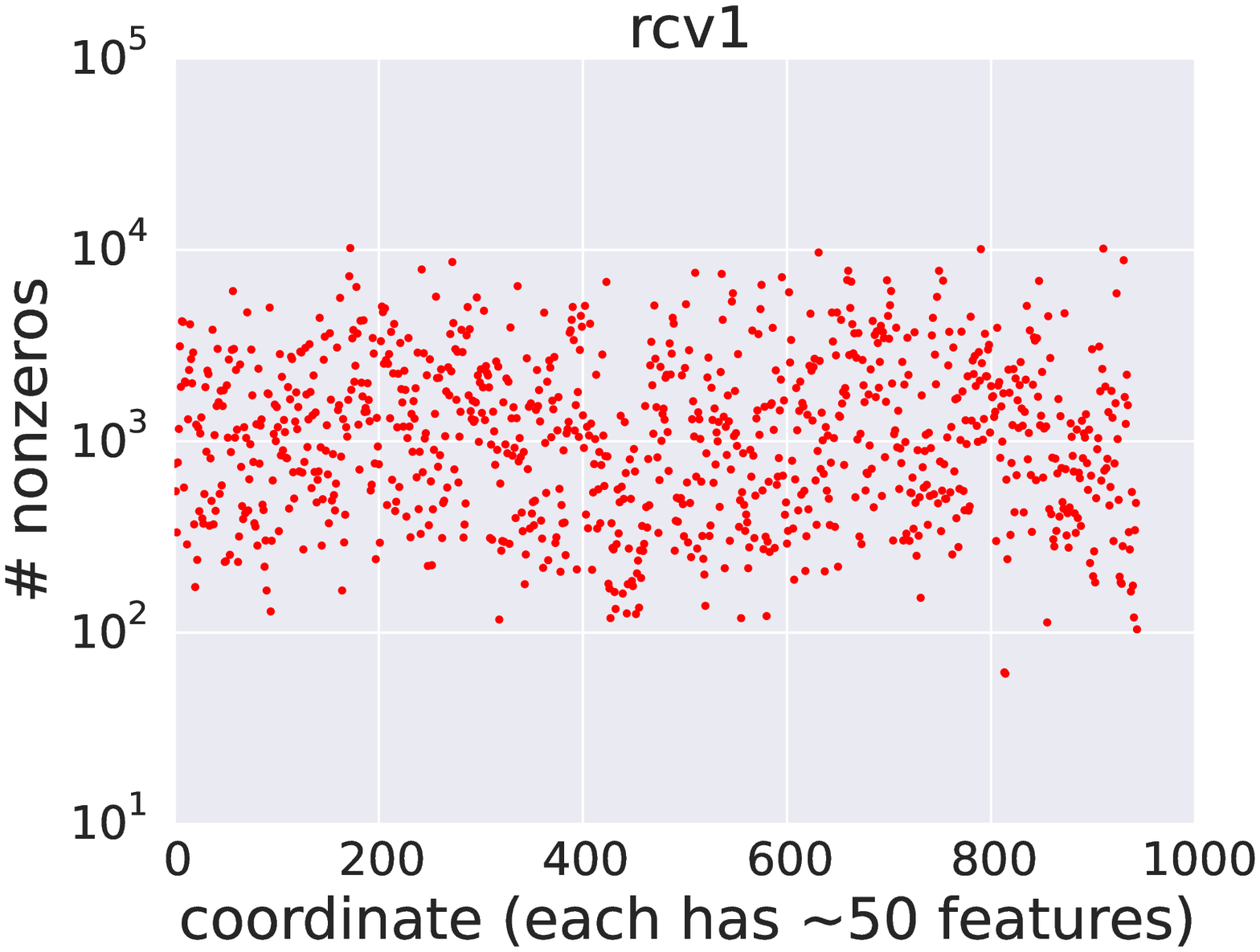}
                \caption{ dataset rcv1}
        \end{subfigure}
        ~ 
        \begin{subfigure}[b]{0.4\textwidth}
                \includegraphics[width=\textwidth]{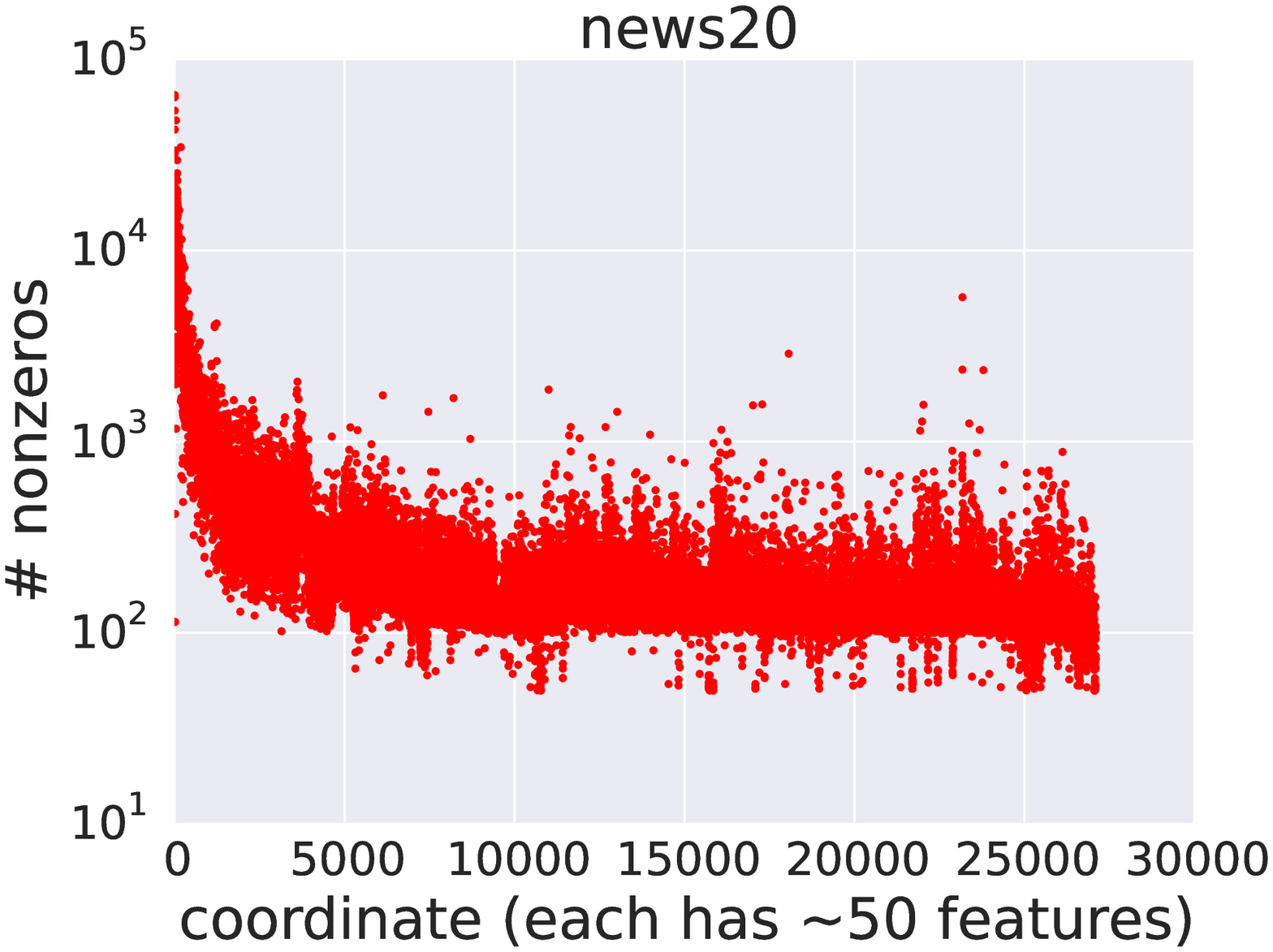}
                \caption{  dataset news20}
        \end{subfigure}
        \vspace{-3mm}
        \caption{The distribution of coordinate sparsity. Each dot represents the total number of nonzeros in the vectors $a_i$ that correspond to each coordinate. The large distribution in (b) is responsible  for the large load imbalance and thus the poor sync-parallel performance.}\label{fig:sparsity}
 \vspace{-3mm}
\end{figure}

We let each coordinate hold roughly 50 features. {Since the total number of features is not divisible by 50, some coordinates have 51 features.} We let each agent draw a coordinate uniformly at random at each iteration. We stop all the tests after 100 epochs since they have nearly identical progress per iteration. 
The step size is set to $\eta_k=0.9,\,\forall k$. Let $A = [a_1, \ldots, a_N]^T$ and $b = [b_1, ..., b_N]^T$. In global memory, we store $A,~b$, and $x$. We also store the product $Ax$ in global memory so that the forward step can be efficiently computed. Whenever a coordinate of $x$ gets updated, $Ax$ is immediately updated at a low cost. Note that if $Ax$ is \emph{not} stored in global memory, every coordinate update will have to compute $Ax$ from scratch, which involves the entire $x$ and will be very expensive.  

Table \ref{tab:log_time} gives the running times of  the sync-parallel and ARock (async-parallel) implementations on the two datasets. We can observe that ARock achieves almost-linear speedup, but sync-parallel scales very poorly as we explain below.

In the sync-parallel implementation,  all the running cores have to wait for the last core to finish an iteration, and therefore if a core has a large load, it slows down the iteration. Although every core is (randomly) assigned to roughly the same number of features (either 50 or 51 components of $x$) at each iteration, their  $a_i$'s have very different numbers of nonzeros (see Figure \ref{fig:sparsity} for the distribution), and the core with the largest number of nonzeros is the slowest (Sparse matrix computation is used for both datasets, which are very large.) As more cores are used,  despite that they altogether do more work at each iteration,  the per-iteration time reduces as the slowest core tends to be slower. The very large imbalance of load explains why the 32 cores only give speedup ratios of 4.0 and 1.3 in Table \ref{tab:log_time}.

On the other hand, being asynchronous, ARock does not suffer from the  load imbalance. Its performance grows nearly linear with the number of cores. In theory, a large load imbalance may cause a large \cut{maximum delay, }$\tau$, and thus a small $\eta_k$. However, the uniform $\eta_k=0.9$ works well in all the tests, possibly because the $a_i$'s are sparse.

Finally, we have observed that the progress toward solving \eqref{eqn:log} is mainly a function of the number of epochs and does not change appreciably  when the number of cores increases or between sync-parallel and async-parallel. Therefore, we always stop at 100 epochs.

\begin{table}[htbp]
\centering
{\small \begin{tabular}{|c|r|r|r|r|r|r|r|r|}
  \hline
  \multirow{3}{*}{\# cores} & \multicolumn{4}{|c|}{rcv1} & \multicolumn{4}{c|}{news20} \\
  \cline{2-9}
  & \multicolumn{2}{|c|}{Time (s)} &  \multicolumn{2}{c|}{Speedup} &  \multicolumn{2}{c|}{Time (s)} & \multicolumn{2}{c|}{Speedup}\\
  \cline{2-9}
  & async & sync &  async & sync &  async & sync &  async & sync \\
  \hline
   1 &   122.0 &  122.0 & 1.0   & 1.0 & 591.1   & 591.3 & 1.0   & 1.0\\
   2 &   63.4   &  104.1 & 1.9   & 1.2 & 304.2   & 590.1 & 1.9   & 1.0\\
   4 &   32.7   &  83.7   & 3.7   & 1.5 & 150.4   & 557.0 & 3.9   & 1.1\\
   8 &   16.8   &  63.4   & 7.3   & 1.9 & 78.3     & 525.1 & 7.5   & 1.1\\
   16 & 9.1     &  45.4   & 13.5 & 2.7 & 41.6     & 493.2 & 14.2 & 1.2\\
   32 & 4.9     &  30.3   & 24.6 & 4.0 & 22.6     & 455.2 & 26.1 & 1.3\\
  \hline
 \end{tabular}}
 \caption{\label{tab:log_time}Running times of  ARock (async-parallel) and sync-parallel FBS implementations for the $\ell_1$ regularized logistic regression on two datasets. Sync-parallel has a very poor speedup  due to the large distribution of coordinate sparsity (Figure \ref{fig:sparsity}) and thus the large load imbalance across cores.}
 \vspace{-4mm}
\end{table}
\end{subsection}
\end{section}

\begin{section}{Conclusion}
We have proposed an async-parallel framework, ARock, for finding a fixed-point of a nonexpansive operator by coordinate updates. We establish the almost sure weak and strong convergence, linear convergence rate and almost-linear speedup of ARock under certain assumptions.
Preliminary numerical results on real data illustrate the high efficiency of the proposed framework compared to the traditional parallel (sync-parallel) algorithms.
\end{section}

\begin{section}{Acknowledgements}
We would like to thank Brent Edmunds for offering invaluable suggestions on the organization and writing of this paper. We would also like to thank Robert Hannah for coming up with the dual-memory approach.
The authors are grateful to Kun Yuan for helpful discussions on decentralized optimization. \end{section}

\bibliographystyle{spmpsci}
\bibliography{asyn}

{\small
\appendix
\section{Derivation of certain updates}\label{sec:appendix}
We show in details how to obtain the updates in \eqref{eqn:relaxed_prs_set} and \eqref{eqn:asyn_admm}.

\subsection{Derivation of updates in \eqref{eqn:relaxed_prs_set}}
Let $x=(x_1,\ldots,x_m)\in \cH^m$, $$f(x):=\sum_{i=1}^m\cI_{C_i}(x_i),\qquad g(x):=\cI_{\{x_1=\cdots=x_m\}}(x),$$
where $g(x)$ equals $0$ if $x_1=\cdots=x_m$ and $\infty$ otherwise.
Then \eqref{eqn:relaxed_prs-a} reduces to
\begin{align*}
\hat{x}^k=&\argmin_{z\in\cH^m}\, g(z) +\frac{1}{2\gamma}\|z-\hat{z}^k\|^2=\argmin_{z\in\cH^m: z_1=\cdots=z_m} \|z-\hat{z}^k\|^2\\
=&\argmin_{z\in\cH^m: z_1=\cdots=z_m} \sum_{i=1}^m\|z_1-\hat{z}_i^k\|^2
=\left(\frac{1}{m}\sum_{i=1}^m\hat{z}_i^k,\ldots,\frac{1}{m}\sum_{i=1}^m\hat{z}_i^k\right)\in\cH^m,
\end{align*}
where the last equality is obtained by noting that $z_1=\frac{1}{m}\sum_{i=1}^m\hat{z}_i^k$ is the unique minimizer of $\sum_{i=1}^m\|z_1-\hat{z}_i^k\|^2$. Next, \eqref{eqn:relaxed_prs-b} reduces to
\begin{align*}
\hat{y}^k=\argmin_{z\in\cH^m} f(z) +\frac{1}{2\gamma}\|z-(2\hat{x}^k-\hat{z}^k)\|^2
=\argmin_{z: z_i\in C_i,\forall i} \sum_{i=1}^m\|z_i-(2\hat{x}_i^k-\hat{z}_i^k)\|^2.
\end{align*}
It is easy to see that $\hat{y}_i^k= \Proj_{C_i}(2\hat{x}_i^k-\hat{z}_i^k),\,\forall i$.

Since \eqref{eqn:relaxed_prs-c} only updates the $i_k$th coordinate of $z$, we only need $\hat{x}_{i_k}^k$ and $\hat{y}_{i_k}^k$, and thus in \eqref{eqn:relaxed_prs_set:a} and \eqref{eqn:relaxed_prs_set:b}, we only compute $\hat{x}_{i_k}^k$ and $\hat{y}_{i_k}^k$. Plugging the above $\hat{x}^k$ and $\hat{y}^k$ into \eqref{eqn:relaxed_prs-c} gives \eqref{eqn:relaxed_prs_set:c} directly.

\subsection{Derivation of \eqref{eqn:asyn_admm}}
We first show how to get \eqref{eq:lag-dual}. The Lagrangian of \eqref{eqn:admm_form} is
$L(x,y,w)= f(x)+g(y)-\langle w, Ax+By-b\rangle,$
and the Lagrange dual function is
\begin{align*}
d(w)=&\min_{x\in \cH_1,y\in\cH_2}L(x,y,w)\\
=&\big(\min_{x\in\cH_1} f(x)-\langle A^*w, x\rangle\big) +\big( \min_{y\in\cH_2} g(y)-\langle B^*w, y\rangle\big) +\langle w, b\rangle\\
=&-\big(\max_{x\in\cH_1} -f(x)+\langle A^*w, x\rangle\big) -\big( \max_{y\in\cH_2} -g(y)+\langle B^*w, y\rangle\big) +\langle w, b\rangle\\
=&-f^*(A^*w)-g^*(B^*w)+\langle w, b\rangle,
\end{align*}
where the last equality is from the definition of convex conjugate: $f^*(z)=\max_x \langle z, x\rangle-f(x).$ Hence, the dual problem is $\max_{w} d(w)$, which is equivalent to \eqref{eq:lag-dual}.

Secondly, we show why $z^+=\prox_{\gamma\cdot d_g}(z)$ is given by \eqref{dualg}. Note
\begin{align*}
\textstyle\min_s d_g(s)+\frac{1}{2\gamma}\|s-z\|^2
=&\textstyle\min_s g^*(B^*s)-\langle s, b\rangle +\frac{1}{2\gamma}\|s-z\|^2\\
=&\textstyle\min_s\max_y \langle B^*s, y\rangle - g(y)-\langle s, b\rangle +\frac{1}{2\gamma}\|s-z\|^2\\
=&\textstyle\max_y\min_s \langle B^*s, y\rangle - g(y)-\langle s, b\rangle +\frac{1}{2\gamma}\|s-z\|^2\\
=&\textstyle\max_y\min_s \langle s, By-b\rangle - g(y) +\frac{1}{2\gamma}\|s-z\|^2\\
=&\textstyle\max_y -g(y)+\langle z, By-b\rangle -\frac{\gamma}{2}\|By-b\|^2\\
=&\textstyle-\min_y g(y)-\langle z, By-b\rangle +\frac{\gamma}{2}\|By-b\|^2,
\end{align*}
where the fifth equality holds because $s^*=z-\gamma(By-b) =\arg\min_s\langle s, By-b\rangle  +\frac{1}{2\gamma}\|s-z\|^2.$
Hence, by the definition of the proximal operator and the above arguments, we have that $z^+=\prox_{\gamma\cdot d_g}(z)$ can be obtained from \eqref{dualg}. Then \eqref{dualf} is from \eqref{dualg} through replacing $g$ to $f$, $B$ to $A$, and $b$ to $0$.

Finally, it is straightforward to have \eqref{eqn:asyn_admm} by plugging \eqref{dualf} and \eqref{dualg} into \eqref{eqn:relaxed_prs}.

\section{Derivation of async-parallel ADMM for decentralized optimization}\label{sec:appendix2}
This section describes how to implement the updates \eqref{eqn:asyn_admm} for the model~\eqref{eqn:des_prob3}.

In~\eqref{eqn:des_prob3}, $g(y)$ and $b$ vanish and, corresponding to the two constraints $x_i = y_{ij}$ and $x_j = y_{ij}$, the two rows of matrices $A$ and $B$ are
$\begin{bmatrix}
     \cdots &  1 & \cdots  & \cdots  & \cdots  \\
     \cdots & \cdots & \cdots & 1 & \cdots
  \end{bmatrix}\quad
  \begin{bmatrix}
     \cdots  & -1 & \cdots  \\
     \cdots  & -1 & \cdots
  \end{bmatrix}, $
where $\cdots$ are zeros, the two coefficients 1 correspond to $x_i$ and $x_j$, and the two coefficients $-1$ correspond to $y_{ij}$.
Then, \eqref{eqn:asyn_admm:a} and~\eqref{eqn:asyn_admm:b} can be calculated as
\begin{align*}
\hat y^k_{li} &= (\hat z_{li,l}^k+\hat z_{li,i}^k)/(2\gamma)\quad\forall l\in L(i),\\
 (\hat w_g^k)_{li,i} & = (\hat z^k_{li,i}-\hat z^k_{li,r})/2\quad\forall l\in L(i),\\
\hat y^k_{ir} &= (\hat z_{ir,i}^k+\hat z_{ir,r}^k)/(2\gamma)\quad\forall r\in R(i),\\
 (\hat w_g^k)_{ir,i} & = (\hat z^k_{ir,i}-\hat z^k_{ir,r})/2\quad\forall r\in R(i).
\end{align*}
In addition, $\hat x^k_i$ can be obtained by solving~\eqref{eqn:asyn_admm_des_a}, and both $z_{li,i}^{k+1}$ and $z_{ir,i}^{k+1}$ can be updated from~\eqref{eqn:asyn_admm_des_b} and~\eqref{eqn:asyn_admm_des_c}.

Furthermore, as mentioned in Section~\ref{subsec:async_admm_des}, we can derive another version of async-parallel ADMM for decentralized optimization, which reduces to the algorithm in \cite{wei2013on}, by activating an edge $(i,j)\in E$ instead of an agent $i$ each time. In this version, the agents $i$ and $j$ associated with the edge $(i,j)$ must also be activated. Here we derive the update~\eqref{eqn:asyn_admm} for the model~\eqref{eqn:des_prob3} with the update order of $x$ and $y$ swapped. Following~\eqref{eqn:asyn_admm} we obtain the following steps whenever an edge $(i,j)\in E$ is activated:
\begin{subequations}
\begin{align*}
\hat x^{k}_{i} & = \arg\min_{x_i} f_i(x_i) -  \big(\sum_{l \in L(i)}\hat z^k_{li, i}  + \sum_{r\in R(i)}\hat z^k_{ir, i} \big) x_i +  \frac{\gamma}{2} |E(i)|\cdot \|x_i\|^2 \\
\hat x^{k}_{j}  &= \arg\min_{x_j} f_j(x_j) -  \big(\sum_{l \in L(j)}\hat z^k_{lj, j}  + \sum_{r\in R(j)}\hat z^k_{jr, j} \big) x_j +  \frac{\gamma}{2} |E(j)|\cdot \|x_j\|^2 \\
(\hat w_{f}^{k})_{ij,i} &= \hat z_{ij,i}^{k} - \gamma \hat x_i^{k} \\
(\hat w_{f}^{k})_{ij,j} &= \hat z_{ij,j}^{k} - \gamma \hat x_j^{k} \\
\hat y_{ij}^{k} &= \arg\min_{y_{ij}}  \langle 2 (\hat w_{f}^{k})_{ij,i} -\hat z_{ij,i}^{k}+ 2(\hat w_{f}^{k})_{ij,j} -\hat z_{ij,j}^{k}, y_{ij} \rangle + \frac{\gamma}{2} \|y_{ij}\|^2  \\
(\hat w_{g}^{k})_{ij,i} &=2 (\hat w_{f}^{k})_{ij,i} -\hat z_{ij,i}^{k} + \gamma  \hat y_{ij}^k\\
(\hat w_{g}^{k})_{ij,j} &=2 (\hat w_{f}^{k})_{ij,j} -\hat z_{ij,j}^{k} + \gamma  \hat y_{ij}^k\\
z_{ij,i}^{k+1} &= z_{ij,i}^{k} + \eta_k ((\hat w_{g}^{k})_{ij,i} - (\hat w_{f}^{k})_{ij,i})\\
z_{ij,j}^{k+1} &= z_{ij,j}^{k} + \eta_k ((\hat w_{g}^{k})_{ij,j} - (\hat w_{f}^{k})_{ij,j}).
\end{align*}
\end{subequations}
Every agent $i$ in the network maintains the dual variables $z_{li,i}$, $l\in L(i)$, and $z_{ir,i}$, $r\in R(i)$, and the variables $x,y,w$ are intermediate and do not need to be maintained between the activations. When an edge $(i,j)$ is activated, the agents $i$ and $j$ first compute their $\{\hat x_i^k,(\hat w_{f}^{k})_{ij,i}\}$ and $\{\hat x_j^k,(\hat w_{f}^{k})_{ij,j}\}$ independently and respectively, then they collaboratively compute $\hat y_{ij}^k$, and finally they update their own $z_{ij,i}^k$ and $z_{ij,j}^k$, respectively. We allow adjacent edges (which share agents) to be activated in a short period of time when their updates are possibly overlapped in time. When $\tau=0$, i.e., there is no simultaneous activation or overlap, it reduces to the algorithm in~\cite{wei2013on}.
}

\end{document}